\documentclass[oneside]{amsart}

\usepackage{amsmath,amssymb,amsfonts,amssymb,amsthm}

\usepackage{rotating}
\usepackage{verbatim}
\usepackage[usenames]{color}
\usepackage{hyperref}
\usepackage{url}
\usepackage{tikz,tikz-qtree,ifthen,cancel}
\usepackage{array,tikz-qtree,ifthen,cancel}
\usepackage{graphicx}
\usepackage{adjustbox}
\usepackage{amsthm,graphicx,tikz,appendix,tikz-qtree,ifthen,cancel}
\usetikzlibrary{calc,shapes,patterns,positioning}
\usepackage{mathrsfs}

\usepackage{pifont}
\usepackage{enumitem}

\usepackage{vruler}
\newtheorem{thm}{Theorem}[section]
\newtheorem{prop}[thm]{Proposition}
\newtheorem{lem}[thm]{Lemma}
\newtheorem{cor}[thm]{Corollary}

\newtheorem*{thm.indivisibility}{Theorem \ref{thm.indivisibility}}

\newtheorem{fact}[thm]{Fact}

\newtheorem*{CharBRDSDAP}{Simple Characterization of  big Ramsey degrees}

\theoremstyle{remark}
\newtheorem{rem}[thm]{Remark}

\theoremstyle{definition}
\newtheorem{defn}[thm]{Definition}
\newtheorem{convention}[thm]{Convention}

\newtheorem{notation}[thm]{Notation}

\newtheorem{example}[thm]{Example}
\newtheorem{question}[thm]{Question}
\newtheorem{problem}[thm]{Problem}

\theoremstyle{remark}

\newcommand{\al}{\alpha}
\newcommand{\om}{\omega}

\newcommand{\sse}{\subseteq}
\newcommand{\contains}{\supseteq}
\newcommand{\forces}{\Vdash}

\DeclareMathOperator{\ran}{ran}

\DeclareMathOperator{\Sim}{Sim}
\DeclareMathOperator{\Ext}{Ext}

\DeclareMathOperator{\type}{tp}
\DeclareMathOperator{\cl}{cl}
\newcommand{\re}{\!\restriction\!}
\DeclareMathOperator{\Emb}{Emb}
\DeclareMathOperator{\Forb}{Forb}

\newcommand{\bC}{\mathbb{C}}
\newcommand{\bD}{\mathbb{D}}

\newcommand{\bJ}{\mathbf{J}}
\newcommand{\bK}{\mathbf{K}}

\newcommand{\bM}{\mathbf{M}}

\newcommand{\bP}{\mathbb{P}}
\newcommand{\bQ}{\mathbb{Q}}

\newcommand{\bT}{\mathbb{T}}
\newcommand{\bS}{\mathbb{S}}

\newcommand{\bU}{\mathbb{U}}

\newcommand{\K}{\mathrm{K}}
\newcommand{\A}{\mathrm{A}}
\newcommand{\B}{\mathrm{B}}

\newcommand{\M}{\mathrm{M}}
\newcommand{\N}{\mathrm{N}}

\newcommand{\bA}{\mathbf{A}}
\newcommand{\bB}{\mathbf{B}}
\newcommand{\bF}{\mathbf{F}}
\newcommand{\bG}{\mathbf{G}}

\newcommand{\bfA}{\mathbf{A}}
\newcommand{\bfB}{\mathbf{B}}
\newcommand{\bfC}{\mathbf{C}}
\newcommand{\bfD}{\mathbf{D}}
\newcommand{\bfE}{\mathbf{E}}
\newcommand{\bfM}{\mathbf{M}}
\newcommand{\bfN}{\mathbf{N}}
\newcommand{\bfO}{\mathbf{O}}

\newcommand{\wsim}{\stackrel{w}{\sim}}

\newcommand{\plussim}{\stackrel{+}{\sim}}

\newcommand{\ra}{\rightarrow}

\newcommand{\Lra}{\Longrightarrow}

\newcommand{\lgl}{\langle}
\newcommand{\rgl}{\rangle}

\newcommand{\rl}{\!\downarrow\!}

\newcommand{\Erdos}{Erd{\H{o}}s}
\newcommand{\Fraisse}{Fra{\"{i}}ss{\'{e}}}
\newcommand{\Hubicka}{Hubi{\v{c}}ka}
\newcommand{\Lauchli}{L{\"{a}}uchli}

\newcommand{\Nesetril}{Ne{\v{s}}et{\v{r}}il}

\newcommand{\Rodl}{R{\"{o}}dl}

\newcommand{\cmark}{\ding{51}}%
\newcommand{\xmark}{\ding{55}}%

\newcommand{\EEAP}{SDAP}
\newcommand{\EEAPnonacronym}{Substructure  Disjoint Amalgamation Property}
\newcommand{\SFAP}{SFAP}
\newcommand{\SFAPnonacronym}{Substructure  Free  Amalgamation Property}
\newcommand{\BEEAP}{BSDAP}

\newcommand{\TwoEEAP}{$2$-SDAP}

\newcommand{\noprint}[1]{\relax}

\title[Simply characterized big Ramsey structures]{Fra{\"{i}}ss{\'{e}} classes with simply characterized\\ big Ramsey structures}

\author{R. Coulson}
\address{United State Military Academy, West Point\\
Department of Mathematical Sciences, Thayer Hall 233, West Point, USA}
\email{Rebecca.Coulson@westpoint.edu}
\urladdr{\url{https://urldefense.com/v3/__https://www.westpoint.edu/mathematical-sciences/profile/rebecca_coulson__;!!NCZxaNi9jForCP_SxBKJCA!E2nu0Rcsfw15sLu6DrXLF-lHQGsf2b2Bh3Muun-jup9REtFQIVMZEMLjZ8j0vRkB-kAM$ }}

\author{N. Dobrinen}
\address{University of Denver\\
Department of Mathematics, 2390 S. York St., Denver, CO USA}
\email{natasha.dobrinen@du.edu}
  \urladdr{\url{http://cs.du.edu/~ndobrine}}

  \author{R. Patel}
  \address{
  African Institute for Mathematical Sciences\\
    M'bour-Thi\`{e}s, Senegal
    }
\email{rpatel@aims-senegal.org}

\thanks{The second author is grateful for support from   National Science Foundation Grant DMS-1901753, which also supported research visits to the University of Denver by the first and third authors.
She also is grateful for support
from Menachem Magidor for hosting her visit to
 The Hebrew University of Jerusalem  in December 2019, during which some of the ideas in this paper were formed. 	
The third author's work on this paper was supported by the National Science Foundation under Grant No. DMS-1928930 while she was in residence at the Mathematical Sciences Research Institute in Berkeley, California, during the Fall 2020 semester.
}

\subjclass[2010]{05D10, 05C55,  05C15, 05C05,  03C15, 03E75}

\keywords{Ramsey theory, \Fraisse\ structure,  canonical partitions, big Ramsey degrees, trees}

\begin{document}

\maketitle

\tableofcontents

\begin{abstract}
We formulate
a
 property strengthening the Disjoint  Amalgamation Property and
  prove that
 every \Fraisse\  structure
 in a finite relational language with relation symbols of arity at most two having
this  property
has finite big Ramsey degrees which have a  simple characterization.
It follows that
any such \Fraisse\ structure
admits
 a big  Ramsey structure.
Furthermore, we prove indivisibility for
 every \Fraisse\ structure in an arbitrary finite relational language
satisfying
this property.
 This work offers a
 streamlined  and unifying approach to  Ramsey theory on some seemingly disparate classes of \Fraisse\ structures.
Novelties    include
a new  formulation of  coding trees in terms of 1-types over initial segments of the \Fraisse\ structure, and  a direct
characterization of  the  degrees  without appeal to the standard method of ``envelopes''.
\end{abstract}


\section{Introduction}\label{sec.intro}

In recent years,
the
Ramsey theory
of
infinite structures has seen quite an expansion.
This area seeks to understand which infinite structures satisfy some analogue of the infinite
Ramsey theorem for the natural numbers.

\begin{thm}[Ramsey, \cite{Ramsey30}]\label{thm.RamseyInfinite}
Given integers $k,r\ge 1$ and a coloring
of the $k$-element subsets of the natural numbers into $r$ colors,
there is an infinite set of natural numbers, $N$, such that all $k$-element subsets of $N$ have the same color.
\end{thm}

For infinite structures, exact analogues of  Ramsey's theorem usually fail, even when the class of finite substructures has the Ramsey property.
This is due to some unseen structure  which persists in  every
infinite substructure isomorphic to the original, but  which dissolves
when considering Ramsey properties of
classes of finite substructures.
This was first seen in
Sierpi\'{n}ski's use of a well-ordering on the rationals to construct a coloring of  unordered pairs of rationals
with two colors such that both colors persist in any subcopy of the rationals.
The interplay between the well-ordering and the rational order forms
 additional structure   which is in some sense essential, as it  persists upon taking any subset forming another dense linear order
 without endpoints.
The quest to characterize and quantify  the  often hidden but essential  structure for infinite structures, more generally, is the area of
 {\em big Ramsey degrees}.

Given an infinite structure
$\bfM$, we say that
 $\bfM$ has {\em finite big Ramsey degrees} if for each finite substructure
$\bfA$ of $\bfM$, there is an integer $T$
such that the following  holds:
For any coloring of
the copies of $\bfA$ in $\bfM$
into finitely many colors, there is a
 substructure $\bfM'$ of $\bfM$
such that $\bfM'$ is isomorphic to $\bfM$, and
 the copies of $\bfA$ in $\bfM'$ take no more than $T$ colors.
When a $T$ having this property exists, the least  such value
is called the {\em big Ramsey degree} of $\bfA$ in
$\bfM$,
denoted
$T(\bfA, \bM)$.
In particular, if
the big Ramsey degree of $\bfA$ in $\bfM$ is one,
then any finite coloring of the copies of
$\bfA$ in $\bfM$ is
constant
on some subcopy of $\bfM$.

While the area of big Ramsey degrees on infinite structures traces back to Sierpi\'{n}ski's result  that
the big Ramsey degree for unordered pairs of rationals is at least two, and progress on the rationals and other binary relational structures was made in  the decades since,
the question of which infinite structures have finite big Ramsey degrees
attracted
extended interest due to
the flurry of
results in
 \cite{Laflamme/NVT/Sauer10},
\cite{Laflamme/Sauer/Vuksanovic06}, \cite{NVT08},
and \cite{Sauer06}
in tandem with
the publication
of
 \cite{Kechris/Pestov/Todorcevic05}, in which
 Kechris, Pestov, and Todorcevic
asked for an analogue of their
correspondence between
 the Ramsey property  of \Fraisse\ classes and  extreme amenability to the setting of
 big Ramsey degrees for \Fraisse\ limits.
 This was
 addressed
 by Zucker in \cite{Zucker19},
 where he proved a
 connection
between
 \Fraisse\
 limits
 with finite big Ramsey degrees and completion flows in topological dynamics.
Zucker's results apply to {\em big Ramsey structures},
expansions of  \Fraisse\
limits
in which the big Ramsey degrees of the \Fraisse\
limits
can be exactly characterized using the additional structure induced by the expanded language.
This additional structure
involves a well-ordering, and
characterizes
the  essential structure which persists  in every infinite subcopy of the \Fraisse\
limit.
It is this essential structure  we seek to understand
 in the study of big Ramsey degrees.

In this paper, we
describe
an amalgamation property,
called the \EEAPnonacronym\ (\EEAP),
forming a strengthened version of  disjoint  amalgamation.
We then
 characterize the exact big Ramsey degrees for all \Fraisse\
 limits
in
finite relational languages
with relation symbols of arity at most two
whose ages have
 \EEAP\ and which have
  two additional properties,
 which  we call
the Diagonal Coding Tree Property and the Extension Property.
 Our characterization, together with results of
Zucker in \cite{Zucker19}, imply that  \Fraisse\ limits having these three properties admit big Ramsey structures, and their automorphism groups have
metrizable universal completion flows.

\begin{thm}\label{thm.main}
Let $\mathcal{K}$ be a  \Fraisse\  class
in a finite relational language
with relation symbols of arity at most two
such that $\mathcal{K}$ satisfies
\EEAP, and such that the \Fraisse\ limit  {\rm{Flim}}$(\mathcal{K})$ of $\mathcal{K}$ has
the Diagonal Coding Tree Property and the Extension Property.
Then
{\rm{Flim}}$(\mathcal{K})$
has finite big Ramsey degrees and, moreover,
admits  a big Ramsey structure.
Hence, the topological group
{\rm{Aut(Flim}}$(\mathcal{K}))$
 has a metrizable universal completion flow, which is unique up to isomorphism.
\end{thm}

We say that a \Fraisse\
structure $\bK$
satisfies  the \EEAPnonacronym$^+$ (\EEAP$^+$)  whenever
the age of $\bK$ has \EEAP\
and $\bK$ has
the Diagonal Coding Tree Property and the Extension Property.

An immediate consequence  of
our proofs
is that a \Fraisse\ limit $\bK$ with \EEAP$^+$
and
with relations of any arity
is {\em indivisible}, by which we mean that every one-element
substructure of $\bK$ has big Ramsey degree equal to one.
In the case when $\bK$ has exactly one substructure of size one
(up to isomorphism),
as happens for instance when the language of $\bK$ has no unary relation symbols and there are no ``loops''
in $\bK$, this definition reduces to the usual one for indivisibility of structures like the Rado graph and the Henson graphs
(see \cite{Sauer06}, \cite{Komjath/Rodl86}, and  \cite{El-Zahar/Sauer89}).

\begin{thm}\label{thm.indivisibility}
Suppose  $\mathcal{K}$ is a  \Fraisse\ class
in a finite relational language
with relation symbols  in any arity such that
its  \Fraisse\ limit  satisfies
\EEAP$^+$.
Then $\bK$ is indivisible.
\end{thm}

Further, we are able
to show that
the  age of any \Fraisse\ structure
 with relations of arity at most two
 satisfying \EEAP$^+$ has ordered expansion with the Ramsey property.
 (See Theorem \ref{thm.SESAPimpliesORP}.)

Theorem \ref{thm.main}  follows from Theorems
\ref{thm.onecolorpertype},
 \ref{thm.bounds},
and  \ref{thm.apply}
of this paper.
Upper bounds
for big Ramsey degrees of \Fraisse\ limits with \EEAP$^+$
are found in Theorem \ref{thm.onecolorpertype}.
These bounds are then proved to be exact in Theorem
 \ref{thm.bounds},
where the big Ramsey degrees are characterized
by
so-called similarity types of diagonal antichains.
This characterization is presented at the end of this
introduction.

For \Fraisse\ classes in relational languages with relation symbols of arity at most two,
Theorem \ref{thm.apply} shows that the characterization
in Theorem  \ref{thm.bounds}  implies that the conditions of a theorem of Zucker in \cite{Zucker19},
guaranteeing existence of big Ramsey structures, hold for any \Fraisse\ limit satisfying \EEAP$^+$.
Moreover,
the
big Ramsey structures
can be  obtained
 by simply adding two
 new relation symbols to the language,
 similarly to the
 construction of
 big Ramsey structures
 for the rationals and the Rado graph
 in \cite{Zucker19}. Thus,
 in Theorem \ref{thm.BRS},
  we obtain
a simple characterization of big Ramsey structures for \Fraisse\ structures with \EEAP$^+$.

Theorem \ref{thm.main} provides new classes of examples
 of big Ramsey structures while
 recovering results in
 \cite{DevlinThesis},
\cite{HoweThesis},
 \cite{Laflamme/NVT/Sauer10}, and
 \cite{Laflamme/Sauer/Vuksanovic06}
and extending special cases of the results in \cite{Zucker20} to obtain exact big Ramsey degrees.
Theorem \ref{thm.indivisibility} provides new  classes of examples of indivisible \Fraisse\ structures, in particular for ordered structures,
while recovering results in
\cite{El-Zahar/Sauer89},
\cite{Komjath/Rodl86},  and
\cite{El-Zahar/Sauer94} and some of the results in
 \cite{Sauer20}.
We point out that  Sauer's work in \cite{Sauer20}
provides the  full picture on indivisibility for
 free amalgamation classes.
Thus for
free amalgamation classes, Theorem \ref{thm.indivisibility} recovers certain cases of Sauer's results,
while for
\Fraisse\ classes without free amalgamation but with \Fraisse\
limits
satisfying \EEAP$^+$, our  results  provide new examples  of indivisible \Fraisse\ structures.

 We now discuss
 several theorems which  follow from  Theorem \ref{thm.main}
or   \ref{thm.indivisibility},
 as well as  new
examples
 obtained  from our results.
A fuller description is provided in
Section \ref{sec.EEAPClasses}.

We show that \EEAP$^+$ holds for  disjoint
 amalgamation classes which are ``unrestricted''
(see Definition \ref{defn.unconst}).
 Particular instances  of
 Theorem  \ref{thm.main}
include
 classes of
  structures with finitely many unary and binary relations such as
graphs and
tournaments, as well as
their ordered versions.
Our examples encompass those
unconstrained
binary relational structures considered in \cite{Laflamme/Sauer/Vuksanovic06}
as well as their ordered expansions.
We also  show that \EEAP$^+$ holds for \Fraisse\
limits
of  free amalgamation classes which  forbid $3$-irreducible substructures, namely, substructures in which any three distinct elements appear  in a tuple of which some relation holds, as well as their ordered versions.
Hence, Theorem \ref{thm.indivisibility} implies that all unrestricted classes,
all free amalgamation classes which  forbid $3$-irreducible substructures, and their ordered expansions have \Fraisse\ limits which are indivisible.
See Theorem
 \ref{thm.supercool} for more details.

Certain
\Fraisse\ structures
derived from the rational linear order
satisfy
 \EEAP$^+$ and hence, by Theorem \ref{thm.main}, admit  big Ramsey structures.
Theorem \ref{thm.LOEqRels} shows in particular that
$\bQ_n$, the rational
linear order with a partition into $n$ dense
pieces,
admits a big Ramsey structure.
Theorem \ref{thm.LOEqRels} also shows
that  the structure
$\bQ_{\bQ}$ admits a big Ramsey structure,
answering  a question raised by Zucker at the 2018 Banff Workshop on {\em Unifying Themes in Ramsey Theory}.
This is the dense linear order without endpoints with an equivalence relation such that  all equivalence classes are convex copies of the rationals.
More generally,  Theorem \ref{thm.LOEqRels} applies to
members of
a natural hierarchy of finitely many convexly ordered equivalence relations, where each  successive  equivalence relation coarsens the previous one;
these also admit big Ramsey structures.
\Fraisse\ structures with finitely
many independent linear orders satisfy a slightly weaker property
than \EEAP$^+$,
which we call the
{\em Bounded \EEAP$^+$}
and discuss in Section \ref{sec.infdiml}. We expect that
 the methods in this paper can be
 adjusted to
 find canonical partitions  for these structures,
 but the characterization  will be
 more complex.

We call structures handled by Theorem \ref{thm.LOEqRels}
{\em $\bQ$-like}, as they have enough rigidity, similarly to $\bQ$, for
\EEAP$^+$ to hold and hence, for the proof methods in this paper  to apply.
Known results  in this  genre of $\bQ$-like structures  which
our methods recover include
Devlin's  characterization of  the  big Ramsey degrees of the rationals \cite{DevlinThesis}
 as well as
 results
 of Laflamme, Nguyen Van Th\'{e}, and Sauer  in
 \cite{Laflamme/NVT/Sauer10}
 characterizing the big Ramsey degrees of
 the $\bQ_n$.

While many of the known big Ramsey degree results use sophisticated versions of Milliken's Ramsey theorem for trees \cite{Milliken79},
and while proofs using the method of forcing to produce new pigeonhole principles in ZFC have appeared in \cite{DobrinenRado19}, \cite{DobrinenH_k19}, \cite{DobrinenJML20}, and  \cite{Zucker20},
there are two  novelties to our approach in this paper which produce a clarity about  big Ramsey degrees.
Given a \Fraisse\ class $\mathcal{K}$,
we fix an  enumerated \Fraisse\ limit of $\mathcal{K}$, which we denote by $\bK$.
By {\em enumerated \Fraisse\ limit}, we mean that the universe of $\bK$ is ordered via the natural numbers.
The first novelty of our approach is that  we
work with trees of   quantifier-free $1$-types
(see Definition
\ref{defn.treecodeK})
and develop  forcing arguments
directly on them to prove
upper bounds for the big Ramsey degrees.
It was suggested to the second author  by Sauer during the 2018 BIRS Workshop,
{\em Unifying Themes in Ramsey Theory}, to try moving the forcing methods from \cite{DobrinenH_k19}  and \cite{DobrinenJML20} to forcing directly on the structures.
Using trees of quantifier-free $1$-types seems to come as close as possible to fulfilling this request, as the $1$-types  allow one to see the essential hidden structure (the interplay of a well-ordering of the universe with first instances where $1$-types disagree), whereas working only on   the \Fraisse\ structures, with no reference to $1$-types, obscures this central feature of big Ramsey degrees from view.
We will be calling such trees {\em coding trees}, as there will be special nodes, called {\em coding nodes},
 representing the vertices of
 $\bK$:
 The $n$-th coding node will be the quantifier-free $1$-type of the  $n$-th vertex of $\bK$ over the
 substructure of $\bK$
 induced on the  first $n-1$ vertices of $\bK$.
 (The $0$-th coding node is the quantifier-free $1$-type of the $0$-th vertex over the empty
 set.)
A second  novelty of our approach is that we find the exact big Ramsey degrees directly from the trees of $1$-types, without appeal to the standard method of ``envelopes''.
This means that
the upper bounds which we  find via forcing arguments
 are shown to be exact.

Using trees
of
quantifier-free
$1$-types
(partially ordered by inclusion)
allows us to prove   a  characterization of  big Ramsey degrees for \Fraisse\ classes with \EEAP$^+$ which is a simple extension of the so-called ``Devlin types'' for the rationals in \cite{DevlinThesis},
and of the  characterization of the big Ramsey degrees of the  Rado graph achieved by Laflamme, Sauer, and Vuksanovic in \cite{Laflamme/Sauer/Vuksanovic06}.
Here, we present  the characterization for  structures  without unary relations.
The full  characterization is given in Theorem \ref{thm.bounds}.

\begin{CharBRDSDAP}
Let $\mathcal{L}$   be a  language
consisting of
finitely many  relation symbols, each of arity
two.
Suppose $\mathcal{K}$ is a \Fraisse\ class in $\mathcal{L}$
such that the \Fraisse\ limit
$\bK$
of $\mathcal{K}$ satisfies
\EEAP$^+$.
Fix a structure  $\bfA\in\mathcal{K}$.
 Let
 $(\bfA,<)$ denote
  $\bfA$ together with a fixed enumeration
  $\lgl\mathrm{a}_i:i<n\rgl$
 of the universe
 of $\bfA$.
We say that a tree $T$ is a
{\em diagonal tree coding $(\bfA,<)$}
if  the following hold:
\begin{enumerate}
\item
$T$ is a finite tree with $n$ terminal nodes and
 branching degree two.
\item
$T$ has  at most one branching node  in
any given
level,
and
no two distinct nodes from among the branching nodes and terminal nodes have the same length.
Hence, $T$ has $2n-1$ many levels.
\item
Let $\lgl \mathrm{d}_i:i<n\rgl$ enumerate the terminal nodes in $T$ in order of increasing length.
Let $\bfD$ be the $\mathcal{L}$-structure induced on the set $\{\mathrm{d}_i:i<n\}$ by the increasing
bijection from   $\lgl\mathrm{a}_i:i<n\rgl$ to $\lgl \mathrm{d}_i:i<n\rgl$, so that $\bfD \cong \bfA$.
Let $\tau_i$ denote the quantifier-free
$1$-type of $\mathrm{d}_i$ over
$\bfD_i$,
the substructure of
$\bfD$ on vertices $\{\mathrm{d}_m:m<i\}$.
Given  $i<j<k<n$,
if $\mathrm{d}_j$ and $\mathrm{d}_k$
both
extend
some node
in $T$
that is at the same level as $\mathrm{d}_i$,
then $\mathrm{d}_j$ and $\mathrm{d}_k$
 have the same quantifier-free $1$-types over
$\mathbf{D}_i$.
That is,
$\tau_j\re \mathbf{D}_i
=\tau_k\re \mathbf{D}_i$.
\end{enumerate}
Let $\mathcal{D}(\bfA,<)$ denote the number of distinct diagonal trees coding $(\bfA,<)$;
let $\mathcal{OA}$ denote a set
consisting of
one representative from each isomorphism class of ordered copies of $\bfA$.
Then
 $$
 T(\bfA,\bK)
=
\sum_{(\bfA,<)\in\mathcal{OA}}\mathcal{D}(\bfA,<)
$$
\end{CharBRDSDAP}

If $\mathcal{L}$ also has unary relation symbols,
in the case that
$\mathcal{K}$ is a free amalgamation class, the simple characterization above holds when modified to  diagonal coding trees with the same number of roots as unary relations.
In the case that $\mathcal{K}$ contains a transitive relation,
then the  above characterization still holds.

 We see  our main contribution as providing a clear and   unified analysis of
a wide class of \Fraisse\ structures with relations of
arity at most two for which the big Ramsey degrees have a simple characterization.
 \vskip.1in

\it Acknowledgements. \rm
The second author thanks
Norbert Sauer for discussions at the  2018 Banff Workshop on
{\it Unifying Themes in Ramsey Theory},
where he suggested trying
  to move  the forcing  directly on the structures.
She also thanks
 Menachem Magidor
 for hosting her
  at the Hebrew University of Jerusalem in December 2019, and  for fruitful discussions on big Ramsey degrees during that time.
She thanks Itay Kaplan for  discussions on big Ramsey degrees and  higher arity relational structures during that visit, and Jan \Hubicka\ for helpful conversations.  The third author thanks Nathanael Ackerman, Cameron Freer and Lynn Scow for extensive and clarifying discussions.
All three authors thank Jan \Hubicka\ and Mat\v{e}j Kone\v{c}n\'{y} for pointing out  a mistake in an earlier version.


\section{Amalgamation properties implying  big Ramsey structures}\label{sec.Structures}

 Theorem \ref{thm.main}
shows that
for finite
 relational languages
of arity at most two,
\Fraisse\ structures
satisfying the
\EEAPnonacronym$^+$
(\EEAP$^+$)
have simply
characterized exact big Ramsey degrees, from which the existence of
 big Ramsey structures follows.
Theorem \ref{thm.indivisibility}  shows that for
arbitrary
finite relational languages,  \Fraisse\ structures
satisfying the \EEAP$^+$ are indivisible.
The inspiration for this
property comes from
a
strengthening of the free amalgamation property,
which we  call the \SFAPnonacronym\ (\SFAP).
We originally found that any binary relational  \Fraisse\
structure
 with an age satisfying  \SFAP\
has finite
 big Ramsey degrees that are characterized in a manner  similar to  the characterizations,
  in \cite{Laflamme/Sauer/Vuksanovic06},
 of big Ramsey degrees
  for the
  Rado graph and other
 unconstrained
  binary relational structures
  with
  disjoint
  amalgamation.
  \SFAP\
  is satisfied by
  the ages of
  all
  unconstrained
  relational
structures having free amalgamation,
  as well as
by \Fraisse\ classes
  with forbidden
  irreducible and
  $3$-irreducible substructures.
 The \EEAPnonacronym\ (\EEAP)
  is a natural
 extension
 of \SFAP\
  to a broader collection of \Fraisse\ classes with
 disjoint amalgamation.
 When the \Fraisse\ limit of an age with \EEAP\ has certain additional properties, which we call the Diagonal Coding Tree Property and the Extension Property (defined in Sections \ref{sec.sct} and \ref{sec.FRT}), we say that it has \EEAP$^+$.  The property \EEAP$^+$
ensures a simple characterization of exact big Ramsey degrees and big Ramsey structures
for \Fraisse\ structures with relations of arity at most two.

In Subsection \ref{subsec.Fcrs}
we review  the basics of \Fraisse\ theory,
the Ramsey property,
 big Ramsey degrees and big Ramsey structures.
More general background  on \Fraisse\ theory can be found in \Fraisse's original paper  \cite{Fraisse54},   as well as
\cite{HodgesBK97}.
 The
 properties
 \SFAP,
 \EEAP\ and  \EEAP$^+$
 are presented in
 Subsection \ref{subsec.EEAP}.


\subsection{
\Fraisse\ theory, big Ramsey degrees, and big Ramsey  structures}\label{subsec.Fcrs}

All relations in this paper will be finitary,
and all languages will consist of finitely many
relation symbols (and no constant or function symbols).
We use the set-theoretic notation $\om$ to denote the set of natural numbers, $\{0,1,2,\dots\}$
and treat $n \in \om$ as the set $\{i\in\om:i<n\}$.

Let  $\mathcal{L}=\{R_i:i< I\}$ be a finite
language
 where each $R_i$ is a relation symbol with associated
arity $n_i\in \om$.
An {\em $\mathcal{L}$-structure} is an object
\begin{equation}
\bfM=\lgl \M, R_0^{\bfM},\dots, R_{I-1}^{\bfM}\rgl
\end{equation}
where
$\M$ is a nonempty set, called the \emph{universe} of $\bfM$,
and each $R^{\bfM}_i\sse \M^{n_i}$.
Finite structures will typically be denoted by $\bfA,\bfB$, etc.,
and  their universes  by $\A,\B$, etc.  Infinite structures will typically be denoted by $\bJ, \bK$ and their universes
by $\mathrm{J}, \K$. We will call the elements of the universe of a structure {\em vertices}.

An {\em embedding} between
$\mathcal{L}$-structures
$\bfM$ and $\bfN$
is an injection $\iota:\M\ra \N$ such that for each $i<I$
and for all $a_0, \ldots , \, a_{n_i - 1} \in \M$,
\begin{equation}
R_i^{\bfM}(a_0,\dots,a_{n_i - 1})\Longleftrightarrow
R_i^{\bfN}(\iota(a_0),\dots,\iota(a_{n_i - 1})).
\end{equation}
A surjective embedding is an \emph{isomorphism}, and an isomorphism from
$\bfM$ to
itself
$\bfM$
is an {\em automorphism}.
The set of embeddings of $\bfM$ into $\bfN$ is denoted $\text{Emb}(\bfM, \bfN)$, and
the set of automorphisms
of $\bfM$ is denoted $\text{Aut}(\bfM)$.
When
$\M \subseteq \N$ and the inclusion map is an embedding, we say
$\bfM$ is a {\em substructure} of
$\bfN$.
When there exists an embedding $\iota$ from $\bfM$ to $\bfN$, the substructure of $\bfN$ having universe $\iota[\M]$ is
called a {\em copy} of $\bfM$ in $\bfN$, and it is a {\em subcopy} of $\bfN$ if $\bfM$ is isomorphic to $\bfN$. The {\em age} of $\bfM$,
written Age($\bfM$),
is the class of all finite $\mathcal{L}$-structures that embed into $\bfM$.
We write $\bfM \le \bfN$
when there is an embedding of
$\bfM$ into $\bfN$, and $\bfM\cong\bfN$ when there is an isomorphism from
$\bfM$ to $\bfN$.

A class $\mathcal{K}$ of finite structures
in
a
finite
relational language
 is called a {\em \Fraisse\ class}  if it is
nonempty,  closed under isomorphisms,
 hereditary, and satisfies the joint embedding and amalgamation properties.
The class $\mathcal{K}$ is  {\em hereditary} if whenever $\bfB\in\mathcal{K}$ and  $\bfA\le\bfB$, then also $\bfA\in\mathcal{K}$.
The class $\mathcal{K}$ satisfies the {\em joint embedding property} if for any $\bfA,\bfB\in\mathcal{K}$,
there is a $\bfC\in\mathcal{K}$ such that $\bfA\le\bfC$ and $\bfB\le\bfC$. The class
 $\mathcal{K}$ satisfies the {\em amalgamation property} if for any embeddings
$f:\bfA\ra\bfB$ and $g:\bfA\ra\bfC$, with $\bfA,\bfB,\bfC\in\mathcal{K}$,
there is a $\bfD\in\mathcal{K}$ and  there are embeddings $r:\bB\ra\bfD$ and $s:\bfC\ra\bfD$ such that
$r\circ f = s\circ g$.
 Note that in a finite relational language, there are only countably many finite structures up to isomorphism.

An $\mathcal{L}$-structure $\bK$ is called {\em ultrahomogeneous} if
every isomorphism between finite substructures of $\bK$  can be extended to an
automorphism of $\bK$.
We call a
countably infinite, ultrahomogeneous structure
a {\em \Fraisse\ structure}.
\Fraisse\ showed \cite{Fraisse54} that the age of a \Fraisse\ structure is a \Fraisse\ class, and that conversely,
given
a \Fraisse\ class $\mathcal{K}$, there is, up to isomorphism, a unique
\Fraisse\ structure
whose age is $\mathcal{K}$.
Such a
\Fraisse\
structure is called the {\em \Fraisse\ limit} of $\mathcal{K}$
or the {\em generic} structure for $\mathcal{K}$.

Throughout this paper, $\bK$ will denote the \Fraisse\ limit of a \Fraisse\ class $\mathcal{K}$.
We will sometimes write Flim$(\mathcal{K})$ for $\bK$.
We will assume  that $\bK$ has universe $\om$, and call such a structure
an {\em enumerated \Fraisse\ structure}.
 For $m<\om$, we let $\bK_m$
 denote the substructure of $\bK$ with universe $m = \{0, 1, \ldots , m -1\}$.

The following amalgamation property will be assumed in
this paper:
A \Fraisse\ class $\mathcal{K}$ satisfies the  {\em Disjoint Amalgamation Property}
if,
given
embeddings
$f:\bfA\ra\bfB$ and $g:\bfA\ra\bfC$, with $\bfA,\bfB,\bfC\in\mathcal{K}$,
there is an amalgam  $\bfD\in\mathcal{K}$ with  embeddings $r:\bB\ra\bfD$ and $s:\bfC\ra\bfD$ such that
$r\circ f = s\circ g$ and  moreover,
$r[\mathrm{B}]\cap s[\mathrm{C}]=r\circ f[\mathrm{A}]=s\circ g[\mathrm{A}]$.
The disjoint amalgamation property is also called  the {\em strong amalgamation property}.
It
is
 equivalent  to the {\em strong embedding property},
 which
 requires that
 for any $\bfA\in\mathcal{K}$, $v\in\mathrm{A}$, and embedding $\varphi:(\bfA-v) \ra\bK$,
there are infinitely many different extensions of $\varphi$ to embeddings of $\bfA$ into $\bK$.  (See \cite{CameronBK90}.)

A \Fraisse\ class has the {\em Free Amalgamation Property}
if  it satisfies
the Disjoint Amalgamation Property
and moreover,
the amalgam $\bfD$ can be chosen so that no tuple satisfying a relation in $\bfD$ includes elements of both
$r[\mathrm{B}] \setminus r\circ f[\mathrm{A}]$ and $s[\mathrm{C}] \setminus s\circ g[\mathrm{A}]$; in other words, $\bfD$ has
no additional relations on its universe other than those inherited from $\bfB$ and $\bfC$.

For languages $\mathcal{L}_0$ and $\mathcal{L}_1$
 such that $\mathcal{L}_0\cap\mathcal{L}_1=\emptyset$, and given
\Fraisse\ classes
$\mathcal{K}_0$ and $\mathcal{K}_1$ in $\mathcal{L}_0$ and $\mathcal{L}_1$, respectively,
 the {\em free superposition} of $\mathcal{K}_0$ and $\mathcal{K}_1$ is the \Fraisse\ class consisting of
all
 finite $(\mathcal{L}_0\cup\mathcal{L}_1)$-structures $\bfA$
 such that
  the $\mathcal{L}_i$-reduct of $\bfA$ is in $\mathcal{K}_i$,  for each $i<2$.
(See also \cite{Bodirsky15} and \cite{Hubicka_CS20}.)  Note that the free superposition of $\mathcal{K}_0$ and $\mathcal{K}_1$ has free  amalgamation if and only if each $\mathcal{K}_i$ has free amalgamation; and similarly for disjoint amalgamation.

Given a \Fraisse\ class $\mathcal{K}$ and substructures
$\bfM,\bfN$
of $\bK$  (finite or infinite)
 with $\bfM\le\bfN$,
we use
 ${\bfN\choose\bfM}$
to denote the set of all substructures of
$\bfN$ which are isomorphic to
$\bfM$. Given
$\bfM\le\bfN\le\bfO$,
substructures of $\bK$, we write
$$
\bfO\ra(\bfN)_{\ell}^{\bfM}
$$
to denote that for each coloring of
${\bfO\choose \bfM}$
into $\ell$ colors, there is an
 $\bfN' \in {\bfO\choose\bfN}$
 such that
${\bfN'\choose\bfM}$
is  {\em monochromatic}, meaning that
 all members of
 ${\bfN'\choose\bfM}$
 have the same color.

\begin{defn}\label{defn.RP}
A \Fraisse\ class  $\mathcal{K}$ has the {\em Ramsey property} if  for any two structures $\bfA\le\bfB$ in $\mathcal{K}$ and any  $\ell \ge 2$,
there is a $\bfC\in\mathcal{K}$ with $\bfB\le\bfC$ such that
$\bfC\ra (\bB)^{\bfA}_\ell$.
\end{defn}

Equivalently, $\mathcal{K}$ has the Ramsey property if for any two structures $\bfA\le \bfB$ in $\mathcal{K}$,
\begin{equation}\label{eq.RP}
\forall \ell\ge 2,\ \ {\bK}\ra ({\bfB})^{\bfA}_{\ell}.
\end{equation}
This equivalent formulation makes comparison with big Ramsey degrees, below, quite  clear.

\begin{defn}[\cite{Kechris/Pestov/Todorcevic05}]\label{defn.bRd}
Given a \Fraisse\ class $\mathcal{K}$ and its \Fraisse\ limit $\bK$,
for any $\bfA\in\mathcal{K}$,
write
\begin{equation}\label{eq.bRd}
\forall \ell\ge 1,\ \ {\bK}\ra ({\bK})^{\bfA}_{\ell,T}
\end{equation}
when there is an
 integer $T\ge 1$ such that
for any integer $\ell \ge 1$,
given any coloring of ${\bK\choose \bfA}$ into
$\ell$ colors,
there is a
 substructure $\bK'$ of $\bK$, isomorphic to $\bK$,  such that ${\bK'\choose \bfA}$ takes no more than $T$ colors.
We say that
 $\bK$ has {\em finite big Ramsey degrees} if for each
 $\bfA \in \mathcal{K}$,
there is an integer $T\ge 1$
such that
equation
(\ref{eq.bRd}) holds.
For a given finite $\bfA\le\bK$, when
such a $T$ exists,
we  let
$T(\bfA,\bK)$  denote the least
one, and call this number the {\em big Ramsey degree} of $\bfA$ in $\bK$.
\end{defn}

Comparing equations (\ref{eq.RP}) and
 (\ref{eq.bRd}),
 we see
that the
 difference between the Ramsey property and having finite big Ramsey degrees  is that
 the former finds a substructure of $\bK$ isomorphic to the {\em finite} structure $\bfB$ in which all copies of $\bfA$ have the {\em same} color, while
  the latter finds an {\em infinite} substructure of $\bK$ which is  isomorphic to $\bK$ in which the copies of $\bfA$ take {\em few} colors.
  It is only when  $T(\bfA,\bK)=1$ that
  there is  a subcopy of $\bK$ in which all copies of $\bfA$ have the same color.

It is normally the case  that for structures $\bfA$ with universe of size greater than one, $T(\bfA,\bK)$ is at least two, if it exists at all.
The fundamental reason
for this stems from Sierpi\'{n}ksi's example that $T(2,\bQ)\ge 2$:
The enumeration of the universe $\om$ of $\bK$ plays against the relations in the structure to preserve more than one color in every subcopy of $\bK$.

On the other hand, many
classes $\mathcal{K}$
are known to have singleton structures (that is, structures with universe consisting of one element)
with big Ramsey degree one; when this holds
for every (isomorphism type of) singleton structure in $\mathcal{K}$,
we say that  $\bK$ is {\em indivisible}.

\begin{defn} \label{defn.indiv}
A \Fraisse\ structure $\bK$ is {\em indivisible} if for every singleton
substructure $\bA$ of $\bK$, $T(\bfA,\bK)$ exists and equals one.
\end{defn}

Note that when
there is only one quantifier-free 1-type over the
empty set satisfied by elements of $\bK$, so that $\bK$ has exactly one singleton substructure up to isomorphism,
indivisibility amounts to saying
that for any partition of the universe of $\bK$ into finitely many pieces, there is a subcopy of $\bK$ contained in one of the pieces.
Indivisibility
has been proved for  many structures,  including the triangle-free Henson graph in
\cite{Komjath/Rodl86},  the $k$-clique-free Henson graphs for all $k\ge 4$ in  \cite{El-Zahar/Sauer89},
more general binary relational free amalgamation structures in
\cite{Sauer03}, and for
$k$-uniform hypergraphs,
$k \ge 3$, that omit
finite substructures
in which all unordered
triples of vertices are contained in
at least one $k$-edge in
\cite{El-Zahar/Sauer94}.
For a  much broader discussion of \Fraisse\ structures and indivisibility, the reader is referred to
Nguyen Van Th\'{e}'s Habilitation \cite{NVTHabil}.
The  recent paper \cite{Sauer20} of Sauer
 characterizes  indivisibility for
\Fraisse\ limits with free amalgamation in finite relational languages: He proves that such a structure is indivisible
 if and only if it is
what he calls
 ``rank linear''.
 Our Theorem \ref{thm.indivisibility}
 overlaps his result,  not fully recovering it, but proving new results for structures without free amalgamation.

A proof that $\bK$ has finite big Ramsey degrees amounts to showing that the numbers $T(\bfA, \bK)$ exist by finding upper bounds for them.
When  a method for producing the numbers
 $T(\bfA,\bK)$
 is given,
 we will say that the {\em exact big Ramsey degrees}
 have been {\em characterized}.
In all known cases where exact big Ramsey degrees have been characterized, this has been done by finding {\em  canonical partitions}
for the finite substructures of $\bK$.

\begin{defn}[Canonical Partition]\label{defn.cp}
Let $\mathcal{K}$ be a \Fraisse\ class with \Fraisse\ limit $\bK$, and let
$\bA\in\mathcal{K}$ be given.
A partition
$\{P_i:i<n\}$
 of ${\bK\choose \bA}$ is a  {\em canonical partition}
if the following hold:
\begin{enumerate}
\item
For every subcopy
$\bJ$ of $\bK$
and each $i < n$,
 $P_i\cap {\bJ\choose\bA}$ is non-empty.
 This property is called {\em persistence}.
\item
For each finite coloring $\gamma$ of ${\bK\choose \bA}$ there is a subcopy
$\bJ$ of $\bK$
such that
for each $i<n$, all members of $P_i\cap {\bJ\choose\bA}$
are assigned the same color by $\gamma$.
 \end{enumerate}
\end{defn}

\begin{rem}\label{rem.embvscopy}
In many papers on big Ramsey degrees, including the foundational results in \cite{DevlinThesis}, \cite{Sauer06}, and \cite{Laflamme/Sauer/Vuksanovic06},
authors color {\em copies} of a given  $\bfA\in\mathcal{K}$ inside $\bK$, working with Definition \ref{defn.bRd}.
In some papers, especially those with very direct ties to topological dynamics of automorphism groups  as in \cite{Zucker19} and \cite{Zucker20},
the
authors
color
{\em embeddings}  of $\bfA$ into $\bK$.
The relationship between
these approaches
is  simple:
A structure  $\bfA\in\mathcal{K}$ has
big
Ramsey degree $T$
for copies if and only if $\bfA$ has
big
Ramsey degree $T \cdot |$Aut$(\bfA)|$
for embeddings.
Thus, one can use whichever formulation most suits the context.
Furthermore,  we show
in Theorem \ref{thm.apply}
that there is a simple way of recovering Zucker's criterion   for existence of big Ramsey structures (which uses colorings of embeddings; see Theorem 7.1 in \cite{Zucker19})
  from
 our
canonical partitions for colorings of copies of a structure.
\end{rem}

The majority of  results on  big Ramsey degrees have been proved  using some auxiliary structure, usually trees, and recently sequences of parameter words  (see \cite{Hubicka_CS20}),  to characterize
 the persistent superstructures which code
 the finite structure
 $\bfA$.
 The exception is  the recent use of category-theoretic approaches (see for instance \cite{Barbosa20},
  \cite{Masulovic18}, and \cite{Masulovic_RBS20}).
 These superstructures  fade away
 in the case of finite structures with the Ramsey property.
An example of how this works can be seen in Theorem \ref{thm.SESAPimpliesORP}, where we recover the ordered Ramsey property  for ages of  \Fraisse\
structures
with \EEAP$^+$ from
 their big Ramsey degrees.
 However, for
 big Ramsey degrees of
  \Fraisse\ limits,
 these superstructures
possess some    essential features which persist,
leading to big Ramsey degrees
greater than one.
The following notion of Zucker  deals with such superstructures via expanded languages.

Let $\mathcal{L}$ be a relational language, $\M$ a set, $\bfN$ an $\mathcal{L}$-structure, and $\iota : \M \ra \N$ an injection.
Write $\bfN \cdot \iota$ for the unique $\mathcal{L}$-structure having underlying set $\M$ such that $\iota$ is an embedding of
$\bfN \cdot \iota$ into $\bfN$.

\begin{defn}[Zucker, \cite{Zucker19}]\label{defn.bRs}
Let $\bK$ be a \Fraisse\ structure in a
relational
language  $\mathcal{L}$ with $\mathcal{K}=$ Age$(\bK)$.
We say that $\bK$ {\em admits a big Ramsey structure} if there is a
relational
language $\mathcal{L}^* \contains\mathcal{L}$ and an $\mathcal{L}^* $-structure $\bK^*$ so that the following hold:
\begin{enumerate}
\item
The reduct of $\bK^*$ to the language $\mathcal{L}$ equals $\bK$.
\item
Each $\bfA\in\mathcal{K}$ has finitely many
expansions to an $\mathcal{L}^*$-structure
$\bfA^*\in$ Age$(\bK^*)$;
denote the set of such expansions by $\bK^*(\bfA)$.
\item
For each $\bfA \in \mathcal{K}$,
$T(\bfA, \bK) \cdot |\text{Aut}(\bfA)| = |\bK^*(\bfA)|$
\item\label{witnessing}
For each $\bfA \in \mathcal{K}$, the function
$\gamma:\Emb(\bfA,\bK)\ra\bK^*(\bfA)$
given by $\gamma(\iota)=\bK^*\cdot \iota$
witnesses the fact that
$$
T(\bfA, \bK) \cdot |\text{Aut}(\bfA)| \ge |\bK^*(\bfA)|,
$$
in the following sense:
For every subcopy $\bK'$ of $\bK$, the image of the restriction of $\gamma$ to
$\text{Emb}(\bfA, \bK')$ has size  $|\bK^*(\bfA)|$.
\end{enumerate}
Such a structure $\bK^*$ is called
a {\em big Ramsey structure} for $\bK$.
\end{defn}

Note that the definition of a big Ramsey structure for $\bK$ presupposes that $\bK$ has finite big Ramsey degrees.  The big Ramsey structure $\bK^*$, when it exists, is  a device for storing information about all the big Ramsey degrees in
$\bK$ together in a uniform way.

While the study of big Ramsey degrees has been progressing for many decades,
a recent   compelling
motivation for finding big Ramsey structures  is the following theorem.

\begin{thm}[Zucker, \cite{Zucker19}]\label{thm.Zucker}
Let $\bK$ be a \Fraisse\ structure which admits a big Ramsey structure, and let $G=$ {\rm Aut}$(\bK)$.
Then the topological group $G$ has a metrizable universal completion flow, which is unique up to isomorphism.
\end{thm}

This theorem
 answered a question in \cite{Kechris/Pestov/Todorcevic05}
 which asked for an analogue, in the context of finite big Ramsey degrees, of the
 Kechris-Pestov-Todorcevic correspondence between the Ramsey property for a \Fraisse\ class and extreme amenability of
 the automorphism group of its \Fraisse\ limit;
 Zucker's theorem provides a similar connection
 between finite big Ramsey degrees and universal completion flows.
The notion of big Ramsey degree in \cite{Zucker19}
involves colorings of embeddings of structures instead of just colorings of  substructures.
As
described
in Remark \ref{rem.embvscopy},
 this poses no problem  when applying  our results  on big Ramsey degrees,
 which involve coloring copies of a structure,
 to Theorem  \ref{thm.Zucker}.


\subsection{The \EEAPnonacronym$^+$}\label{subsec.EEAP}

Recall that given
a \Fraisse\ class $\mathcal{K}$
in a finite relational language $\mathcal{L}$,
we let $\bK$ denote
an enumerated
\Fraisse\ limit of $\mathcal{K}$
with underlying set $\om$.
All results will hold regardless of which
enumeration
is chosen.
We make the following conventions and assumptions,
which will hold in the rest of this paper.

All types will be quantifier-free $1$-types, over a finite parameter set, that are realizable in $\bK$.
 With one exception, all such types will be complete; the exception is the case of ``passing types'', defined in  Section \ref{sec.sct}, which may be partial.
Complete types  will be denoted simply ``tp''.

We will assume that
for any relation symbol $R$ in $\mathcal{L}$, $R^{\bK}(\bar{a})$ can hold only for tuples
$\bar{a}$ of {\em distinct} elements of
$\om$.
In particular, we assume our structures have
 no loops.
We further assume that all relations
in $\bK$
 are  {\em non-trivial}:
 This means that
 for each relation symbol $R$ in $\mathcal{L}$, there exists a $k$-tuple $\bar{a}$ of (distinct) elements of $\om$ such that
 $R^{\bK}(\bar{a})$ holds, and a
$k$-tuple
  $\bar{b}$ of (distinct) elements of $\om$ such that $\neg R^{\bK}(\bar{b})$ holds.
 Since $\mathcal{K}$ has disjoint amalgamation by assumption, non-triviality will imply that there are
  infinitely many $k$-tuples from $\om$
 that satisfy $R^{\bK}$,
  and infinitely many that do not.
 We will further hold to  the convention that if
 $\mathcal{L}$ has any unary relation
 symbols,
 then
 letting
  $R_0,\dots, R_{n-1}$ list them,
  we have that $n\ge 2$ and for each $\bfA\in\mathcal{K}$,
 for each $a\in\mathrm{A}$,
  $R_i^{\bfA}(a)$
 holds for exactly one $i<n$.
By possibly adding new unary relation
symbols
to the language, any \Fraisse\ class with
unary relations can be assumed to meet this convention.
Finally, we assume that there is at least one non-unary relation symbol in $\mathcal{L}$.  This poses no real restriction,
as whenever a finite language has only unary relation symbols, any disjoint amalgamation class in that language will have a \Fraisse\ limit
that consists
of finitely many disjoint copies of $\om$, with vertices in a given copy all realizing the same quantifier-free 1-type over the empty set.
In this case, finitely many applications of Ramsey's Theorem will prove the existence of finite big Ramsey degrees.

We now present the Substructure Free Amalgamation Property.
This property also
provides the intuition behind the more general  amalgamation property \EEAP\  (Definition \ref{defn.EEAP_new}),
laying the foundation for the main ideas of this paper.

\begin{defn}[\SFAP]\label{defn.SFAP}
A \Fraisse\ class $\mathcal{K}$ has the
{\em  \SFAPnonacronym\ (\SFAP)} if $\mathcal{K}$ has free amalgamation,
and given  $\bfA,\bfB,\bfC,\bfD\in\mathcal{K}$, the following holds:
Suppose
\begin{enumerate}
\item[(1)]
$\bfA$  is a substructure of $\bfC$, where
 $\bfC$ extends  $\bfA$ by two vertices,
say $\mathrm{C}\setminus\mathrm{A}=\{v,w\}$;

\item[(2)]
 $\bfA$  is a substructure of $\bfB$ and
 $\sigma$ and $\tau$  are
 $1$-types over $\bfB$  with   $\sigma\re\bfA=\type(v/\bfA)$ and $\tau\re\bfA=\type(w/\bfA)$; and
\item[(3)]
$\bfB$ is a substructure of $\bfD$ which extends
 $\bfB$ by one vertex, say $v'$, such that $\type(v'/\bfB)=\sigma$.

\end{enumerate}
  Then there is
an   $\bfE\in\mathcal{K}$ extending  $\bfD$ by one vertex, say $w'$, such that
  $\type(w'/\bfB)=\tau$, $\bfE\re (
  \mathrm{A}\cup\{v',w'\})\cong \bfC$,
  and $\bfE$ adds no other relations over $\mathrm{D}$.
\end{defn}
The  definition of \SFAP\ can be stated using embeddings rather than substructures
in the standard way.
We remark that
requiring $\bfC$ in (1) to have only two more vertices than $\bfA$ is
sufficient for all our uses of the property in proofs of big Ramsey degrees,
and hence we have not formulated the property for $\bfC$ of arbitrary finite size.

\begin{rem}\label{rem.simple}
\SFAP\ is equivalent to
free amalgamation along with
 a model-theoretic property that may be termed {\em free 3-amalgamation}, a special case of the {\em disjoint 3-amalgamation} property defined
 in \cite{Kruckman19}: In the definition of disjoint $n$-amalgamation in Section 3 of \cite{Kruckman19}, take $n=3$ and impose the further condition that the ``solution'' or 3-amalgam disallows any relations (in any realization of the solution) that were not already stipulated in the initial 3-amalgamation ``problem''.
Kruckman shows
in \cite{Kruckman19} that if the age of a \Fraisse\ limit $\bK$ has disjoint amalgamation and disjoint 3-amalgamation, then $\bK$ exhibits a model-theoretic tameness property called  \emph{simplicity}.
\end{rem}

\SFAP\
ensures that  a finite substructure of a given
enumerated \Fraisse\ structure can be extended as desired without any requirements on its configuration  inside the larger structure.
\SFAP\
precludes any need for  the so-called ``witnessing properties'' which were necessary for the proofs of finite big Ramsey degrees
for
constrained binary
 free amalgamation classes,
as in
the  $k$-clique-free Henson graphs in \cite{DobrinenJML20} and \cite{DobrinenH_k19}, and the recent  more general extensions in  \cite{Zucker20}.
Free amalgamation classes with forbidden $3$-irreducible substructures satisfy \SFAP, as
shown in Proposition \ref{prop.FA}.

The next amalgamation property extends \SFAP\ to disjoint amalgamation classes.
 In  the
 definition, we again use substructures rather than embeddings.

\begin{defn}[\EEAP]\label{defn.EEAP_new}
A \Fraisse\ class $\mathcal{K}$ has the
{\em  \EEAPnonacronym\ (\EEAP)} if $\mathcal{K}$
has disjoint amalgamation,
and the following holds:
Given   $\bfA, \bfC\in\mathcal{K}$, suppose that
 $\bfA$ is a substructure of $\bfC$, where   $\bfC$ extends  $\bfA$ by two vertices, say $v$ and $w$.
Then there exist  $\bfA',\bfC'\in\mathcal{K}$, where
$\bfA'$
contains a copy of $\bfA$ as a substructure
and
$\bfC'$ is a disjoint amalgamation of $\bfA'$ and $\bfC$ over $\bfA$, such that
letting   $v',w'$ denote the two vertices in
 $\mathrm{C}'\setminus \mathrm{A}'$ and
assuming (1) and (2), the conclusion holds:
 \begin{enumerate}
 \item[(1)]
Suppose
$\bfB\in\mathcal{K}$  is any structure
 containing $\bfA'$ as a substructure,
and let
 $\sigma$ and $\tau$  be
  $1$-types over $\bfB$  satisfying    $\sigma\re\bfA'=\type(v'/\bfA')$ and $\tau\re\bfA'=\type(w'/\bfA')$,
\item[(2)]
Suppose
$\bfD\in \mathcal{K}$  extends  $\bfB$
by one vertex, say $v''$, such that $\type(v''/\bfB)=\sigma$.
\end{enumerate}
Then
  there is
an  $\bfE\in\mathcal{K}$ extending   $\bfD$ by one vertex, say $w''$, such that
  $\type(w''/\bfB)=\tau$ and  $\bfE\re (
  \mathrm{A}\cup\{v'',w''\})\cong \bfC$.
\end{defn}

\begin{rem}\label{rem.freesup}
We note that
\SFAP\ implies \EEAP,
 taking $\bfA'=\bfA$ and $\bfC'=\bfC$
and because disjoint amalgamation is implied by free amalgamation.
Further, it
follows from
their definitions that \SFAP\ and \EEAP\ are each
preserved  under free superposition.
 \end{rem}

\begin{example}
The idea behind allowing for an
extension
$\bfA'$ of $\bfA$  in the definition of \EEAP\
 is most simply  demonstrated
 for
  the \Fraisse\ class
  $\mathcal{LO}$
  of finite linear orders.
Given  $\bfA,\bfC\in \mathcal{LO}$,
 suppose
 $v,w$ are the two vertices of $\mathrm{C}\setminus\mathrm{A}$ and
 suppose that
  $v<w$ holds in $\bfC$.
  We can require $\bfA'$ to  be some extension of $\bfA$ in $\mathcal{K}$
 containing some vertex $u$
  so  that
the formula $(x<u)$ is in
$\type(v/\bfA')$ and
$(u<x)$ is in
$\type(w/\bfA')$, where $x$ is a variable.
Then given any $1$-types  $\sigma,\tau$
extending $\type(v/\bfA'), \type(w/\bfA')$, respectively,
 over some structure $\bfB$ containing $\bfA'$ as a substructure,
any  two
vertices $v',w'$ satisfying $\sigma,\tau$
will automatically satisfy $v'<w'$, thus producing a copy of $\bfC$ extending $\bfA$.

In the case of  finitely many independent  linear orders,
we can similarly  produce an
$\bfA'$ which ensures that
{\em any} vertices  $v',w'$ satisfying such  $\sigma,\tau$  as above  produce a copy of
$\bfC$ extending $\bfA$.
In more general cases,  the use of $\bfA'$ only ensures that {\em there exist} such vertices $v',w'$.
\end{example}

\begin{rem} \label{rem.weak3dap}
Ivanov \cite{Ivanov99} and independently, Kechris and Rosendal \cite{KechrisRosendal07}, have formulated a weakening of the amalgamation property which is called
{\em almost amalgamation} in \cite{Ivanov99} and {\em weak amalgamation} in \cite{KechrisRosendal07}.  This property arises in the context of generic automorphisms of countable structures.  In the presence of disjoint amalgamation, \EEAP\  may be thought of as a ternary version of weak amalgamation (one of several possible such versions), and as a ``weak'' version of the disjoint 3-amalgamation property from \cite{Kruckman19} (again, one of several possible such weakenings).
\end{rem}

\begin{rem}\label{rem.observe}
We note that we could have used the definition of the free 3-amalgama\-tion property from Remark \ref{rem.simple}, and of
an appropriately formulated version of a
``weak'' disjoint 3-amalgamation property as in Remark \ref{rem.weak3dap}, to simplify some of the proofs in the next section, which gives examples of \Fraisse\ classes with \SFAP\ and \EEAP. We have chosen to use Definitions  \ref{defn.SFAP} and \ref{defn.EEAP_new} instead, as
they are the forms used in the proof of Theorem \ref{thm.matrixHL}.
\end{rem}

In Section \ref{sec.sct}
onwards
we will be working with
so-called
{\em coding trees}
of $1$-types, which represent subcopies of a given  \Fraisse\ limit $\bK$.
For
\Fraisse\ structures in languages
with
relation symbols
of arity greater than two,
a priori, these  trees may have unbounded branching.
However, for all classes with \SFAP\ and for all classes with \EEAP\ which we have investigated, one can construct subtrees with bounded branching which still represent $\bK$.
Accordingly, we
formulate
the following
strengthened version  of \EEAP,
which imposes conditions on the branching in a coding tree for $\bK$.
The notions regarding coding trees of $1$-types require some introduction, so we refer the reader to Sections
\ref{sec.sct} and \ref{sec.FRT}
for their definitions
rather than reproducing everything here.
We mention only that
diagonal trees are skew trees which have binary splitting (see Definition \ref{def.diagskew}).
The Diagonal Coding Tree Property
is
defined
in Section \ref{sec.sct} (Definition \ref{defn.DCTP})
and the Extension Property is defined in
Section \ref{sec.FRT} (Definition \ref{defn.ExtProp}).

\begin{defn}[\EEAP$^+$]\label{defn_EEAP_newplus}
A \Fraisse\
structure $\bK$
has the
{\em  \EEAPnonacronym$^+$ (\EEAP$^+$)} if
its age
$\mathcal{K}$ satisfies \EEAP,
and $\bK$ has
the Diagonal Coding Tree Property and the Extension Property.
\end{defn}

We note that while the Diagonal Coding Tree Property and Extension Property are defined in terms of an enumerated \Fraisse\ structure, they are independent of the chosen enumeration, and hence \EEAP$^+$ is a property of a \Fraisse\ structure itself.

We have already seen that \SFAP\ implies \EEAP.
It will be shown in Theorem
\ref{thm.SFAPimpliesDCT}
that \SFAP\ in fact implies \EEAP$^+$,
recalling our assumption throughout that the language contains a non-unary relation symbol.
A coding tree version of \EEAP$^+$ is presented in Definition \ref{def.EEAPCodingTree},
and is implied by
Definition \ref{defn_EEAP_newplus}.
This
coding tree version  will be used in the proofs in this paper.

The motivation behind    \EEAP$^+$
was  to distill
the essence of those \Fraisse\ classes for which
the forcing arguments in
Theorem \ref{thm.matrixHL}
 work.
As such, it yields big Ramsey degrees which have simple characterizations, similar to those of the rationals and the Rado graph.
It is known
that \EEAP$^+$, and even \EEAP,
are not necessary for obtaining  finite  big Ramsey degrees.
For instance, generic  $k$-clique-free graphs    \cite{DobrinenH_k19} and the generic partial order \cite{Hubicka_CS20} have been shown to have  finite big Ramsey degrees,
and their ages do not have \EEAP.
A catalogue of these
and other such results will be
presented
at the end of Section \ref{sec.EEAPClasses}.
The
focus of this
paper is
on  characterizing  the
big Ramsey degrees  of
 those
 \Fraisse\ classes for which
simple  forcing methods
suffice.


\section{Examples of \Fraisse\ classes satisfying \EEAP$^+$} \label{sec.EEAPClasses}

We now  investigate  \Fraisse\  classes
which  have \Fraisse\ structures satisfying  \EEAP$^+$.
Such
classes
seem to  fall roughly into two categories:  Free amalgamation classes
of relational structures
in which
 any forbidden substructures  are $3$-irreducible
 (Definition \ref{defn.3irred})
 or which are
 unrestricted
 (Definition \ref{defn.unconst}), as well as their ordered expansions;
 and
  disjoint amalgamation classes which are
 in some sense
 ``$\bQ$-like''.

In this section, we will  be  verifying that various collections of  \Fraisse\ classes
satisfy \EEAP.
 We will  show
later
that most of these \Fraisse\ classes  in fact
 have \Fraisse\ limits satisfying \EEAP$^+$,
after we have
developed the machinery of coding trees of
$1$-types
in Section \ref{sec.sct}.
At the end of this
section,
we provide
 a  catalogue of  \Fraisse\
 structures
 which have been investigated for big Ramsey degrees.
 The list is
 non-exhaustive, as research is ongoing, but it provides a view  of many of the main results currently known.

 First, we consider free amalgamation classes.
 The following definition appears in \cite{Conant17}, and occurs implicitly in work on indivisibility in \cite{El-Zahar/Sauer94}.

 \begin{defn}\label{defn.3irred}
Let $r \ge 2$, and let $\mathcal{L}$ be a finite relational language.  An $\mathcal{L}$-structure $\bF$ is {\em $r$-irreducible} if for any $r$ distinct elements $a_0, \ldots , a_{r-1}$ in $\mathrm{F}$ there is some $R \in \mathcal{L}$ and $k$-tuple $\bar{p}$ with entries from $\mathrm{F}$, where
$k \ge r$
is the arity of $R$, such that each $a_i$, $i < r$, is among the entries of $\bar{p}$, and $R^\bF(\bar{p})$ holds.  We  say $\bF$ is {\em irreducible} when $\bF$ is 2-irreducible.
 \end{defn}

Note that for $r > \ell \ge 2$, a structure that is $r$-irreducible need not be $\ell$-irreducible.  This is because for any structure $\mathbf{F}$
such that $|\mathrm{F}|< r$, it is vacuously the case that $\mathbf{F}$ is $r$-irreducible, but if $|\mathrm{F}| \ge \ell$, then $\mathbf{F}$ may not be
$\ell$-irreducible.

 Given a set $\mathcal{F}$ of  finite $\mathcal{L}$-structures,
 let $\Forb(\mathcal{F})$ denote the
 class of finite $\mathcal{L}$-structures $\bA$ such that no member of $\mathcal{F}$ embeds into $\bA$.
 It is a standard fact that a \Fraisse\ class $\mathcal{K}$ is a free amalgamation class if and only if $\mathcal{K} = \Forb(\mathcal{F})$ for some
 set $\mathcal{F}$ of finite irreducible $\mathcal{L}$-structures.  (See  \cite{Siniora/Solecki20}  for a proof).  When
 $\mathcal{F}$ is a set of finite $\mathcal{L}$-structures such that members of $\mathcal{F}$ are both irreducible and 3-irreducible,
 $\Forb(\mathcal{F})$ furthermore has \SFAP.

\begin{prop}\label{prop.FA}
Let $\mathcal{L}$ be a
finite relational language
and $\mathcal{F}$ a (finite or infinite)  collection of finite $\mathcal{L}$-structures which are
irreducible and
$3$-irreducible.
Then
$\Forb(\mathcal{F})$ satisfies \SFAP.
Hence
the \Fraisse\ limit of $\Forb(\mathcal{F})$ has
\EEAP$^+$.
\end{prop}

\begin{proof}
Since the structures in $\mathcal{F}$ are irreducible, $\Forb(\mathcal{F})$ is a free amalgamation class.

Fix  $\bfA,\bfB,\bfC\in \Forb(\mathcal{F})$  with  $\bfA$ a substructure of both  $\bfB$ and  $\bfC$
and
$\mathrm{C} \!\setminus\! \mathrm{A} = \{v, w\}$.
Let
$\sigma,\tau$
be
realizable
$1$-types over $\bfB$ with $\sigma\re\bfA=\type(v/\bfA)$
and  $\tau\re\bfA=\type(w/\bfA)$.
Suppose
 $\bfD\in \Forb(\mathcal{F})$ is
 a $1$-vertex extension of $\bfB$ realizing $\sigma$.
 Thus, $\mathrm{D}=\mathrm{B}\cup\{v'\}$ for some $v'$
 such that
  $\type(v'/\bfB)=\sigma$.

Extend $\bfD$ to an $\mathcal{L}$-structure $\bfE$
by one vertex $w'$ satisfying  $\type(w'/\bfB)=\tau$ such that  for each relation symbol $R\in\mathcal{L}$,
letting
$k$
denote the arity of $R$,
we have the following:
\begin{enumerate}
\item[(a)]
For each $k$-tuple
$\bar{p}$ with entries from $\mathrm{A}\cup \{v', w'\}$,
let $\bar{q}$ be the $k$-tuple with entries from $\mathrm{A}\cup\{v, w\}$
such that each occurrence of $v'$, $w'$ in $\bar{p}$ (if any) is replaced by
$v$, $w$, respectively,
 and all other entries remain the same.
Then we require that $R^\bfE(\bar{p})$ holds if and only if $R^\bfC(\bar{q})$ holds.
\item[(b)]
If $k\ge 3$, then for each $b\in\mathrm{B}\setminus\mathrm{A}$
and each $k$-tuple $\bar{p}$
with entries from $\mathrm{E}$ such that $b, v', w'$ are among the entries of $\bar{p}$, we require that
$\neg R^\bfE(\bar{p})$ holds.
\end{enumerate}
It follows from (a) that $\bfE\re(\mathrm{A}\cup\{v',w'\})\cong \bfC$.
It remains to show
 that $\bfE$ is a member of $\Forb(\mathcal{F})$.
To do so,
it suffices to show that
no
 $\mathbf{F}\in\mathcal{F}$
embeds
 into  $\bfE$.

Suppose toward a contradiction that
 some  $\mathbf{F}\in\mathcal{F}$
embeds into $\bfE$.
Let $\mathbf{F}'$ denote an
  embedded copy of $\mathbf{F}$, with universe
$\mathrm{F}'\sse\mathrm{E}$.
For what follows, it helps to recall that $\mathrm{E}=\mathrm{B}\cup\{v',w'\}$.
Since $\bfD$ is in $\Forb(\mathcal{F})$,
$\mathbf{F}$ does not embed into $\bfD$, so
$\mathrm{F}'$ cannot be contained in $\mathrm{D}$.
Hence $w'$ must be in $\mathrm{F}'$.
Likewise, since
$\tau$ is a realizable $1$-type over $\bfB$,
the substructure $\bfE\re(\mathrm{B}\cup\{w'\})$  is in $\Forb(\mathcal{F})$ and hence does not contain a copy of $\mathbf{F}$.
Therefore,  $v'$ must be in
$\mathrm{F}'$.
By (a), since $\bfC$ is in $\Forb(\mathcal{F})$, the substructure $\bfE \re ( \mathrm{A} \cup \{v', w' \}) $
does not contain a copy of $\mathbf{F}$. Hence there must be some $b \in \mathrm{B}\setminus\mathrm{A}$ such that
$b$ is in $\mathrm{F}'$.
Since $\mathbf{F}$ is  $3$-irreducible,
there must be  some
relation symbol $R\in\mathcal{L}$
with arity $k \ge 3$, and some $k$-tuple $\bar{p}$ with entries from $\mathrm{F}'$ and with $b, v', w'$ among its entries, such that $R^{\bF'}(\bar{p})$ holds.
However,
 (b) implies
 $\neg R^\bfE(\bar{b})$ holds,
contradicting that $\mathbf{F}'$ is a copy of $\mathbf{F}$ in $\bfE$.
Therefore, $\mathbf{F}$ does not embed into $\bfE$.
It follows that $\bfE$ is a member of
$\Forb(\mathcal{F})$.

We have established that $\Forb(\mathcal{F})$ has \SFAP. Theorem \ref{thm.SFAPimpliesDCT} implies that the \Fraisse\ limit of
$\Forb(\mathcal{F})$ has \EEAP$^+$.
\end{proof}

We now consider a
type of \Fraisse\ class that is a generalization, to arbitrary finite relational languages, of the \Fraisse\ classes in finite binary relational
languages that were considered in
 \cite{Laflamme/Sauer/Vuksanovic06}.

 \begin{defn}\label{defn.unconst}
Given a
relational
language $\mathcal{L}$, letting $n$ denote the highest arity of any relation symbol in $\mathcal{L}$, for each $1\le i\le n$, let $\mathcal{L}_i$ denote
the sublanguage consisting of
the relation symbols in $\mathcal{L}$ of arity $i$.
Let $\mathcal{C}_i$ be a
set of
structures in the language $\mathcal{L}_i$ with domain $\{0,\dots,i-1\}$
that is closed under isomorphism.
Following \cite{Laflamme/Sauer/Vuksanovic06}, we call
$ \mathcal{C}_i$ a {\em universal constraint set}.

Let  $\mathcal{U}_{\mathcal{C}_i}$ denote the class of all finite relational structures $\bfA$ in the language
$\mathcal{L}_i$  for which the following holds:
Every induced substructure of $\bfA$ of cardinality $i$ is isomorphic to one of the structures in $\mathcal{C}_i$.
Let
$\mathcal{C}:=\bigcup_{1\le i\le n}\mathcal{C}_i$
and let
$\mathcal{U}_{\mathcal{C}}$ denote the free superposition of the
classes
$\mathcal{U}_{\mathcal{C}_i}$, $1\le i\le n$.
We call such a class $\mathcal{U}_{\mathcal{C}}$ {\em unrestricted}.
  \end{defn}

It is straightforward to check that an unrestricted class $\mathcal{U}_{\mathcal{C}}$  is  a \Fraisse\ class with disjoint amalgamation.

In \cite{Laflamme/Sauer/Vuksanovic06},
Laflamme, Sauer, and Vuksanovic characterized the exact big Ramsey degrees for
the \Fraisse\ structures in finite binary relational languages
whose ages are
unrestricted.
We now show that
arbitrary
unrestricted
\Fraisse\ classes  satisfy \EEAP.

\begin{prop}\label{prop.LSVSFAP}
Let $\mathcal{U}_{\mathcal{C}}$ be an
unrestricted
\Fraisse\ class.
Then $\mathcal{U}_{\mathcal{C}}$ satisfies \EEAP, hence also its ordered expansion $\mathcal{U}^{<}_{\mathcal{C}}$ satisfies \EEAP.
Moreover, the
 \Fraisse\ limits of
 $\mathcal{U}_{\mathcal{C}}$ and
$\mathcal{U}^{<}_{\mathcal{C}}$
have \EEAP$^+$.
\end{prop}

\begin{proof}
Let  $\mathcal{L}$ be a finite relational language  with $n$ denoting the highest arity of any relation symbol in $\mathcal{L}$,
 and let  $\mathcal{C}=\bigcup_{1\le i\le n}\mathcal{C}_i$,
where each $\mathcal{C}_i$ is a universal constraint set.
Suppose $\bfA,\bfC\in \mathcal{U}_{\mathcal{C}}$ are given
such that $\bfC$ extends $\bfA$ by two vertices $v,w$.
Here, we simply let $\bfA'=\bfA$ and $\bfC'=\bfC$.
Suppose $\bfB\in \mathcal{U}_{\mathcal{C}}$  is any structure containing $\bfA$ as a substructure, and let $\sigma,\tau$ be $1$-types over $\bfB$ satisfying $\sigma\re \bfA=\type(v/\bfA)$ and $\tau\re\bfA=\type(w/\bfA)$.
Suppose further that $\bfD\in \mathcal{U}_{\mathcal{C}}$ extends $\bfB$ by one vertex, say $v'$, such that $\type(v'/\bfB)=\sigma$.

Let $\rho$ be the $1$-type of  $w$ over $\bfC \re( \mathrm{A} \cup \{v\})$.
Take  $\bfE$  to be any $\mathcal{L}$-structure extending $\bfD$ by one vertex, say $w'$, such that
the following hold:
$\type(w'/\bfB)=\tau$ and
$\type(w'/ (\bfD\re (A \cup \{v'\}) )$ is
the $1$-type obtained by substituting $v'$ for $v$ in $\rho$.
If $\mathcal{C}_1$ is non-empty,
then we simply
take a structure $\mathbf{Z}\in \mathcal{U}_{\mathcal{C}_1}$
and declare
$w'$ to satisfy the unary relation which the vertex in $\mathbf{Z}$ satisfies.
For  each subset $G\sse E$ of cardinality at most $n$  containing  $v'$ and $w'$ and at least one vertex of $B\setminus A$,
letting $i$ denote  the cardinality of $G$,
the
$\mathcal{L}_i$-reduct of the structure $\bfE\re G$
 is isomorphic to  a member of $\mathcal{C}_i$.
Then $\bfE$ is a member of  $\mathcal{U}_{\mathcal{C}}$, and  $\bfE\re (\mathrm{A}\cup\{v',w'\})\cong \bfC$.

Thus, $\mathcal{U}_{\mathcal{C}}$ satisfies \EEAP.
By Proposition \ref{prop.LO_n} below,  the \Fraisse\ class of finite linear orders satisfies \EEAP.
As  \EEAP\ is preserved under free superpositions (see Remark \ref{rem.freesup}),
the ordered expansion $\mathcal{U}^{<}_{\mathcal{C}}$  also satisfies \EEAP.

Let $\mathbf{U}_{\mathcal{C}}$ denote an enumerated \Fraisse\ limit of $\mathcal{U}_{\mathcal{C}}$.
The coding tree $\bS(\mathbf{U}_{\mathcal{C}})$
(see Definition \ref{defn.treecodeK})
has the property that  all nodes  of the same length  have the same branching degree.
It is simple to construct a diagonal coding tree inside $\bS(\mathbf{U}_{\mathcal{C}})$,
because the universal constraint set allows each node $s$ in any subtree  $T$
of  $\bS(\mathbf{U}_{\mathcal{C}})$
to be extended independently of the substructure represented by the coding nodes in $T$ of length less or equal to that of $s$.
Thus, the Diagonal Coding Tree property trivially holds.
Further,  (1)
of
Definition \ref{defn.ExtProp}  is trivially satisfied, and hence  $\mathbf{U}_{\mathcal{C}}$ has the Extension Property.
Thus, $\mathbf{U}_{\mathcal{C}}$ satisfies \EEAP$^+$.

For an enumerated \Fraisse\ limit $\mathbf{U}_{\mathcal{C}}^{<}$ of
 $\mathcal{U}^{<}_{\mathcal{C}}$,
a diagonal coding tree can be constructed inside $\bS(\mathbf{U}^{<}_{\mathcal{C}})$
similarly  to the construction in Lemma \ref{lem.SFAPplusoderimpliesDCT}.
Again,  (1)
of
Definition \ref{defn.ExtProp}  is trivially satisfied, so the Extension Property holds.
Thus,
 $\mathbf{U}_{\mathcal{C}}^{<}$ satisfies \EEAP$^+$.
\end{proof}

By Proposition \ref{prop.LO_n} below,  the \Fraisse\ class of finite linear orders satisfies \EEAP.
As \SFAP\ implies \EEAP, and \EEAP\ is preserved under free superpositions (see Remark \ref{rem.freesup}),
the ordered expansion $\mathcal{K}^<$ of any \Fraisse\ class $\mathcal{K}$ with \SFAP\  will also have \EEAP.
 In Theorem \ref{thm.SFAPimpliesDCT} and
 Lemma \ref{lem.SFAPplusoderimpliesDCT}
 we will show that
whenever a \Fraisse\ class $\mathcal{K}$ has \SFAP, the \Fraisse\ limits of $\mathcal{K}$ and $\mathcal{K}^<$ satisfy  \EEAP$^+$.
Applying Propositions \ref{prop.FA} and \ref{prop.LSVSFAP} and Theorem \ref{thm.main}, we then obtain  the following.

\begin{thm}\label{thm.supercool}
 Let $\mathcal{L}$ be a finite relational
 language,
$\mathcal{K}$ a \Fraisse\ class in language $\mathcal{L}$, and $\mathcal{K}^<$ the ordered expansion of $\mathcal{K}$.
If $\mathcal{K}$ is an unrestricted \Fraisse\ class, then $\mathcal{K}$
and  $\mathcal{K}^<$ have \EEAP.
 If $\mathcal{K} = \Forb(\mathcal{F})$ for some set $\mathcal{F}$ of finite
irreducible and
 $3$-irreducible $\mathcal{L}$-structures, then $\mathcal{K}$
 has \SFAP\ and
 $\mathcal{K}^<$ has \EEAP.
  All such classes have \Fraisse\ limits satisfying \EEAP$^+$, and hence are
 are indivisible.
Moreover,
the \Fraisse\ limits of
all
such classes with only
unary and binary relations
 admit big Ramsey structures, and their exact big Ramsey degrees have a simple characterization.
\end{thm}

We now discuss previous
results
 recovered by Theorem \ref{thm.supercool}, as well as their  original proof methods.

 In
    \cite{Laflamme/Sauer/Vuksanovic06},
 Laflamme, Sauer, and Vuksanovic characterized  the
 exact big Ramsey degrees of the Rado graph, generic directed graph, and generic tournament.
 More generally,
 they characterized  exact big Ramsey degrees
for  the \Fraisse\ limit of  any
unrestricted \Fraisse\ class in a language
consisting of
  finitely many binary relations.
Their characterization is exactly recovered in our
Theorem \ref{thm.bounds}.

  Their proof utilized Milliken's theorem
  for strong trees \cite{Milliken79}
  and the method of envelopes, building
  on exact upper bound results for big Ramsey degrees of the Rado graph due to Sauer in \cite{Sauer06}.
Theorem \ref{thm.main}  recovers their characterization of the exact big Ramsey degrees for these  classes of structures,
and proves new results for their ordered expansions.
The indivisibility result in its full generality for unrestricted \Fraisse\ structures with relations in any arity, as well as  their ordered expansions,   is new.

 Theorem \ref{thm.supercool} also   extends a result of
El-Zahar and Sauer
 \cite{El-Zahar/Sauer94}, in which
they
proved indivisibility for
free amalgamation classes of
 $k$-uniform hypergraphs ($k\ge 3$)
with forbidden
$3$-irreducible substructures.
As these structures have only one isomorphism type of singleton substructure, their result says
that  for any $k\ge 3$ and any collection $\mathcal{F}$ of
irreducible,
$3$-irreducible $k$-uniform hypergraphs,
vertices  in  $\Forb(\mathcal{F})$ have big Ramsey degree
one.

We mention that for each $n\ge 2$, the \Fraisse\ class of finite
 $n$-partite graphs is easily seen to satisfy \SFAP.
 John Howe proved in his PhD thesis \cite{HoweThesis} that the generic  bipartite  graph has finite big Ramsey degrees; his  methods use an adjustment of  Milliken's theorem.
 Finite  big Ramsey degrees for $n$-partite graphs for  all $n\ge 2$    follow from  the more recent work of  Zucker   in \cite{Zucker20};
  his methods use a flexible version of coding trees and envelopes, but  lower bounds are not attempted in that
  paper.

Next we consider disjoint amalgamation classes which are ``$\bQ$-like''  in
 that their resemblance to
 linear orders
  makes them  in
some sense  rigid enough  to satisfy \EEAP.
Starting with
 the rationals  as a linear order $(\bQ,<)$,
 we shall show that the  \Fraisse\ class of finite   linear orders  satisfies
 \EEAP, and that $(\bQ, <)$ satisfies
 \EEAP$^+$.
 Further,
 the rational linear order with a vertex partition into finitely many dense pieces
 satisfies
 \EEAP$^+$.
We obtain a hierarchy  of
linear orders with
nested convexly ordered equivalence relations
that each satisfy
\EEAP$^+$.

Given $n\ge 1$, let $\mathcal{LO}_n$ denote  the \Fraisse\ class   of finite
structures
with  $n$-many independent linear orders. The
language
for $\mathcal{LO}_n$ is $\{<_i:i<n\}$, with each $<_i$ a binary relation symbol.
In  standard  notation,
$\mathcal{LO}$ denotes $\mathcal{LO}_1$.

\begin{prop}\label{prop.LO_n}
The \Fraisse\ limit of
$\mathcal{LO}$,
namely the rational linear order,
  satisfies \EEAP$^+$.
For each   $n\ge 2$, $\mathcal{LO}_n$  satisfies \EEAP.
\end{prop}

\begin{proof}
Fixing  $n\ge 1$,
suppose   $\bfA$ and $\bfC$   are in $\mathcal{LO}_n$ with
  $\bfA$   a substructure
of $\bfC$ and
$\mathrm{C}\setminus \mathrm{A}=\{v,w\}$.
Let $\bfC'$ be the extension of $\bfC$ by one
 vertex, $a'$, satisfying the following:
 For each $i<n$,
 if $v<_i w$ in $\bfC$, then
 $v<_i a'$ and $a'<_i w$ are in $\bfC'$;
 otherwise, $w<_i a'$ and $a'<_i v$ are  in $\bfC'$.
Define  $\bfA' $ to be the induced  substructure $\bfC'\re (\mathrm{A}\cup\{a'\})$ of $\bfC'$.

Suppose  that $\bfB$ is a finite linear order containing
 $\bfA'$ as a substructure, and let
$\sigma$ and $\tau$ be
$1$-types over $\bfB$ with the property that
$\sigma\re\bfA'=\type(v/\bfA')$ and $\tau\re\bfA'=\type(w/\bfA')$.
Suppose that $\bfD$  is a one-vertex extension of $\bfB$ by the vertex $v'$  so that $\type(v'/\bfB)=\sigma$ holds.
Now
let $\bfE$ be an extension of $\bfD$ by one vertex $w'$ satisfying  $\type(w'/\bfB)=\tau$.
For each $i<n$, $v<_i w$ holds  in $\bfC'$
 if and only if
$x<_i a'$ is in $\sigma$ and
 $a'<_i x$ is in $\tau$.
 (The opposite,
  $w<_i v$, holds  in $\bfC'$
 if and only if
$a'<_i x$ is in $\sigma$ and
 $x<_i a'$ is in $\tau$.)
 It follows that $v'<_i w'$ holds  in $\bfE$ if and only if $v<_i w$  holds in $\bfC$.
Therefore, we automatically obtain $\bfE\re(\mathrm{A}\cup\{v',w'\})\cong \bfC$.
Thus,  \EEAP\ holds.

 Lemma  \ref{lem.Q_Qbiskew} will show that
 $\mathcal{LO}$ satisfies the
 Diagonal Coding Tree Property, and the Extension Property will trivially hold.
 Hence, $\mathcal{LO}$ will satisfy
 \EEAP$^+$.
\end{proof}

Next,  we consider \Fraisse\ classes
of structures with a linear order and
a finite vertex partition.
Following the notation in \cite{Laflamme/NVT/Sauer10},
for each $n\ge 2$, let  $\mathcal{P}_n$  denote the
\Fraisse\ class
with language $\{<,P_1,\dots,P_n\}$, where  $<$ is a
binary relation symbol and each $P_i$ a unary relation symbol, such that in any structure in $\mathcal{P}_n$,
$< $ is interpreted as a linear order and the interpretations of the $P_i$ partition the vertices.
The \Fraisse\ limit of $\mathcal{P}_n$, denoted by
$\bQ_n$,  is
the
rational linear order with a partition of its underlying set into $n$ definable pieces, each of which is dense in $\bQ$.

\begin{prop}\label{bQn}
For each $n\ge 1$,
the \Fraisse\ limit $\bQ_n$ of
the \Fraisse\ class $\mathcal{P}_n$ satisfies \EEAP$^+$.
\end{prop}

\begin{proof}
The proof is almost exactly the same as that for the rationals.
Fixing  $n\ge 1$,
suppose   $\bfA$ and $\bfC$   are in $\mathcal{P}_n$ with
  $\bfA$   a substructure
of $\bfC$ and
$\mathrm{C}\setminus \mathrm{A}=\{v,w\}$.
Let $\bfC'$ be the extension of $\bfC$ by one
 vertex, $a'$,  such that
 $v<w$ in $\bfC$ if and only if
 $v< a'$ and $a'< w$ in $\bfC'$;
 (otherwise, $w<v$ and
 $w<a'$ and $a'<v$ hold  in $\bfC'$).
Let  $\bfA'=\bfC'\re (\mathrm{A}\cup\{a'\})$.

Given any $\bfB,\sigma, \tau,\bfD,v''$ as in  (2) and (3)  of Part (B) of Definition \ref{defn.EEAP_new},
any extension of $\bfD$ by one vertex $w''$ to a structure $\bfE$ with $\type(w''/\bfB)=\tau$ automatically has $v''<w''$ holding in $\bfE$ if and only if $v'<a'<w'$ holds in $\bfA'$.
Since each $P_i$ is a unary relation,
 $P_i(x)$ is in $\sigma$ if and only if $P_i(v')$ holds.
 Thus,  it follows that $P_i(v'')$ holds  in $\bfE$ for that $i$ such that $P_i(v)$ holds in $\bfA$.
 Likewise for $w''$.
 Therefore, $\bfE\re(\mathrm{A}\cup\{v'',w''\})\cong \bfC$.
 Thus,  \EEAP\ holds.

The Extension Property trivially holds for $\bQ_n$.
\EEAP$^+$ follows from the fact that the  coding tree of $1$-types for $\bQ_n$ is a skew tree with splitting degree two, from  which
the  construction of a tree satisfying \EEAP$^+$ will easily follow (see Lemma \ref{lem.DCTLO}).
\end{proof}

Next, we consider  \Fraisse\ classes with
a linear order and
finitely many
convexly ordered
equivalence relations:
An equivalence relation on a linearly ordered set is {\em convexly ordered} if each of its
equivalence classes is an interval with respect to the linear order.

Given  the
language
$\mathcal{L}=\{<,E\}$,
where $<$ and $E$ are binary relation symbols,
let $\mathcal{COE}$ denote the \Fraisse\ class
of
{\em convexly ordered equivalence relations},
$\mathcal{L}$-structures in which $<$ is interpreted as a linear order and $E$ as an equivalence relation that is convex with respect to that order.
The \Fraisse\ limit of  $\mathcal{COE}$, denoted by
$\bQ_{\bQ}$, is the dense linear order without endpoints with an equivalence relation
that has
infinitely many equivalence classes, each
an interval of order-type
$\bQ$, and with an induced order on the set of equivalence classes that is also of order-type $\bQ$.
One can  think of $\bQ_\bQ$ as $\bQ$  copies of $\bQ$ with the lexicographic order.
This  structure was
described
by Kechris, Pestov, and Todorcevic in  \cite{Kechris/Pestov/Todorcevic05}, where they proved that its automorphism group is extremely amenable;
 from
 the main result of
  \cite{Kechris/Pestov/Todorcevic05},
  it then follows that
$\mathcal{COE}$ has the Ramsey property.
This generated interest in  the question of
whether $\bQ_{\bQ}$ has finite big Ramsey degrees or big Ramsey structures.

Let $\mathcal{COE}_2$ denote the \Fraisse\ class in language $\{<,E_0,E_1\}$, where $<$,
$E_0$ and $E_1$
are binary relation symbols,
such that in any structure in $\mathcal{COE}_2$, $<$ is interpreted as a linear order, $E_0$ and $E_1$ as convexly ordered equivalence
relations, and with the additional property that the interpretation of $E_1$ is a coarsening of that of $E_0$; that is, for any
$\bfA$ in $\mathcal{COE}_2$, $a\, E_0^\bfA\, b$ implies $a\, E_1^\bfA\, b$.
Then Flim$(\mathcal{COE}_2)$ is
 $\bQ_{\bQ_{\bQ}}$, that is $\bQ$ copies of $\bQ_{\bQ}$; we shall  denote this as
  $(\bQ_{\bQ})_2$.
  One can see that this recursive construction gives rise to a hierarchy of dense linear orders without endpoints with finitely many convexly ordered  equivalence relations, where each successive equivalence relation coarsens the previous one.
In general,   let $\mathcal{COE}_{n}$ denote the \Fraisse\ class
in the  language $\{<,E_0,\dots, E_{n-1}\}$
where $<$ is interpreted as a linear order and
each $E_i$  $(i<n)$ is interpreted as a
convexly ordered
equivalence relation, and such that
 for each $i<n-2$, the interpretation of $E_{i+1}$ coarsens that of $E_i$.
Let
$(\bQ_{\bQ})_{n}$ denote
the \Fraisse\ limit of  $\mathcal{COE}_{n}$.

More generally, we may consider \Fraisse\ classes
that are a
blend
of the $\mathcal{COE}_n$ and $\mathcal{P}_p$, having
finitely many linear orders,
finitely many convexly ordered equivalence relations, and a partition into finitely many pieces (each of which, in the \Fraisse\ limit, will be dense).
 Let $\mathcal{L}_{m,n,p}$ denote the  language consisting of
 finitely many binary relation symbols, $<_0,\dots, <_{m-1}$, finitely many binary relation symbols $E_0,\dots, E_{n-1}$,
and finitely many unary relation symbols $P_0,\dots,P_{p-1}$.
 A \Fraisse\ class $\mathcal{K}$ in language
 $\mathcal{L}_{m,n,p}$  is a member of
 $\mathcal{LOE}_{m,n,p}$
 if each $<_i$, $i<m$, is interpreted as a linear order,
each $E_j$, $j<n$, is interpreted as a
convexly ordered
equivalence relation with respect to exactly one of the linear orders $<_{i_j}$, for some $i_j<\ell$,
and the interpretations of the $P_k$, $k<p$, induce a vertex partition into at most $p$ pieces.
Let $\mathcal{LOE}$ be the union over all triples $(m,n,p)$
of $\mathcal{LOE}_{m,n,p}$.
Let $\mathcal{COE}_{n,p}$
be
the  \Fraisse\ class
in
$\mathcal{LOE}_{1,n,p}$ for which  the reduct to the language $\{<_0,E_0,\dots, E_{n-1}\}$ is a member of $\mathcal{COE}_n$.

\begin{prop}\label{prop.loe}
Each
\Fraisse\ class  in  $\mathcal{LOE}$ satisfies  \EEAP.
Moreover, for any $n,p$,
the \Fraisse\ limit of
 $\mathcal{COE}_{n,p}$
satisfies \EEAP$^+$.
\end{prop}

\begin{proof}
Suppose   $\bfA$ and $\bfC$   are in $\mathcal{K}$ with
  $\bfA$   a substructure
of $\bfC$ and
$\mathrm{C}\setminus \mathrm{A}=\{v,w\}$.
The unary relations are handled exactly as they were in Proposition \ref{bQn}, so we  need to check that \EEAP\ holds for the binary relations.

Let $\bfC'$ be an  extension of $\bfC$
by   vertices $a'_k$ ($k<m+n$) satisfying the following:
For each $i<m$,
$v<_i w$ if and only if $v<_i a'_i$ and $a'_i<_i w$ in $\bfC'$.
Given  $j<n$,  if
$v\, E_j\, w$  holds  in $\bfC$, then require that $a'_{m+j}$ satisfies  $v\, E_j\, a'_{m+j}$ and $w\, E_j\, a'_{m+j}$   in $\bfC'$.
If $v \not \hskip-.06in  E_j\, w$  holds in $\bfC$,
then require that
$a'_{m+j}$ satisfies
 $v\, E_j\, a'_{m+j}$ and $w \not \hskip-.06in E_j\, a'_{m+j}$ in $\bfC'$.
Let $\bfA'=\bfC'\re(\mathrm{A}\cup\{a'_k:k<m+n\})$.

Suppose  that $\bfB\in\mathcal{K}$ contains
 $\bfA'$ as a substructure, and let
$\sigma$ and $\tau$ be  consistent realizable $1$-types over $\bfB$ with the property that   $\sigma\re\bfA'=\type(v/\bfA')$ and $\tau\re\bfA'=\type(w/\bfA')$.
Suppose that $\bfD$  is a one-vertex extension of $\bfB$ by the vertex $v'$  satisfying $\type(v'/\bfB)=\sigma$.
Now
let $\bfE$ be an extension of $\bfD$ by one vertex $w'$ satisfying  $\type(w'/\bfB)=\tau$.
The same argument as in the proof of Proposition \ref{prop.LO_n}
ensures that  for each $i<m$, $v'<_i w'$ in $\bfE$ if and only if $v<_i w$ in $\bfC$.

Fix $j<n$.
 If
$v\, E_j\, w$  in $\bfC$,  then  as $v\, E_j\, a'_{m+j}$ and $w\, E_j\, a'_{m+j}$  hold in $\bfC'$,  the formula
 $x\, E_j\, a'_{m+j}$  is in both $\sigma$ and $\tau$.
Since $v'$ satisfies $\sigma$ and $w'$ satisfies $\tau$,
it follows that $v\, E_j\, w$  in $\bfE$.
On the other hand, if
 $v \not \hskip-.06in E_j\, w$ holds  in $\bfC$,
 then  the formula
  $x\, E_j\, a'_{m+j}$ is in $\sigma$  and $x \not \hskip-.06in E_j\, a'_{m+j}$ is in $\tau$.
  Again, since
 $v'$ satisfies $\sigma$ and $w'$ satisfies $\tau$,
it follows that $v  \not \hskip-.06in E_j\, w$  in $\bfE$.
Thus,
$\bfE\re(\mathrm{A}\cup\{v',w'\})\cong \bfC$.
Hence \EEAP\ holds.

Given any $n,p$,
we will show
that the \Fraisse\ limit of $\mathcal{COE}_{n,p}$ satisfies
 the Diagonal Coding Tree Property
 in
Lemma \ref{lem.Q_Qbiskew}
and that
the Extension Property holds in
Lemma \ref{lem.EPCOE},
ensuring  that  \EEAP$^+$ holds.
\end{proof}

This brings us to our second collection of  big Ramsey structures.

\begin{thm}\label{thm.LOEqRels}
The following \Fraisse\ structures
 satisfy \EEAP$^+$.
Hence
they
admit big Ramsey structures.
\begin{enumerate}
\item
The rationals, $\bQ$.
\item
$\bQ_n$, for each $n\ge 1$.
\item
$\bQ_{\bQ}$, and more generally,
$(\bQ_{\bQ})_n$ for each $n\ge 2$.
\item
The \Fraisse\ limit of any \Fraisse\ class
in $\mathcal{COE}_{n,p}$, for any $n,p\ge 1$.
\end{enumerate}
\end{thm}

\begin{proof}
This follows from Propositions \ref{prop.LO_n},
\ref{bQn} and
\ref{prop.loe},
and Theorem \ref{thm.main}.
\end{proof}

We now discuss previous results which are recovered in Theorem \ref{thm.LOEqRels}, and results which  are new.

Part (1) of Theorem \ref{thm.LOEqRels} recovers
the following previously known results:
Upper bounds for  finite big Ramsey degrees of the rationals  were  found  by Laver \cite{LavUnp}
using Milliken's theorem.
The big Ramsey degrees were characterized and computed by Devlin in \cite{DevlinThesis}.
Zucker  interpreted
Devlin's characterization into a big Ramsey structure, from which he then constructed the universal completion flow  of the rationals
 in \cite{Zucker19}.

Exact big Ramsey degrees of the structures $\bQ_n$ were characterized  and calculated by Laflamme, Nguyen Van Th\'{e}, and Sauer in \cite{Laflamme/NVT/Sauer10}, using a colored level set version Milliken Theorem which they proved specifically for their application.
The work in this paper using coding trees of $1$-types   provides a new way to view and recover  their characterization of the big Ramsey degrees.
From their work on $\bQ_2$,
Laflamme, Nguyen Van Th\'{e}, and Sauer further
calculated the big Ramsey degrees of the circular
directed graph $\mathbf{S}(2)$
in \cite{Laflamme/NVT/Sauer10}.
Exact Ramsey degrees of $\mathbf{S}(n)$
for all $n\ge 3$  were recently calculated by  Barbosa in \cite{Barbosa20}
using category theory methods.
These structures $\mathbf{S}(n)$ have ages which do not satisfy \EEAP.

Part (3) of Theorem \ref{thm.LOEqRels} answers a question
 posed
 by Zucker during the open problem session
at the 2018 BIRS Workshop on {\em Unifying Themes in Ramsey Theory}:
He asked whether $\bQ_{\bQ}$ has finite big Ramsey degrees and whether it  admits a big Ramsey structure.
At that meeting,
proofs that $\bQ_{\bQ}$
 has finite big Ramsey degrees were found
  by  \Hubicka\ using unary functions and strong trees,
by Zucker using similar methods,
 and by Dobrinen
 using an approach that involved
developing a topological Ramsey space  with strong trees as bases, where each node in the given base is replaced with a strong tree.
None  of these proofs have been   published, nor were those upper bounds shown to be exact.
Independently, Howe also proved upper bounds for big Ramsey degrees in $\bQ_{\bQ}$ \cite{HoweThesis}.
The result  in this paper via \EEAP$^+$ and coding trees of $1$-types
characterizes exact big Ramsey degrees and
proves  that $\bQ_{\bQ}$ admits a big Ramsey structure, and
moreover, shows how it
 fits into a  broader scheme of structures which have easily described big Ramsey degrees.

Part (4) of Theorem \ref{thm.LOEqRels}  in its full generality is new.

We now turn to \Fraisse\ classes that do not fall within the purview of this paper.
The sorts of   \Fraisse\ classes which we know do not satisfy  \EEAP\ are those
with some forbidden irreducible  substructure
which is not $3$-irreducible.
 For instance, the  ages of the $k$-clique-free Henson graphs, most metric spaces, and the generic partial order  do not satisfy \EEAP.
We present two concrete examples of \Fraisse\ classes failing \EEAP\ to give an idea of how failure arises from interplays of two vertices.

\begin{example}[\SFAP\ fails for triangle-free graphs]\label{ex.trianglefree}
  Let $\mathcal{G}_3$ denote
  the \Fraisse\ class of finite triangle-free graphs.
  Let $\bfA$ be the graph with  two  vertices $\{a_0,a_1\}$ forming a  non-edge,
   and let
  $\bfC$ be the graph with   vertices $\{a_0,a_1,v,w\}$ with   exactly  one  edge,
   $v\, E\, w$.
 Suppose
  $\bfB$ has vertices  $\{a_0,a_1,b\}$, where $b\not\in\{v,w\}$.
  Let $\sigma=\{\neg E(x,a_0)\wedge \neg E(x,a_1) \wedge E(x,b)\}$
  and
  $\tau=\{\neg E(x,a)\wedge E(x,a_1) \wedge E(x,b)\}$.
  Then $\sigma\re\bfA=\type(v/\bfA)$,
 $\tau\re\bfA=\type(w/\bfA)$, and $\sigma\ne\tau$.

Suppose $\bfE\in\mathcal{G}_3$ is a graph
satisfying
the conclusion of  Definition \ref{defn.SFAP}.
To simplify notation, suppose that $\bfE$ has
universe
$\mathrm{E}=\{a_0,a_1,b,v,w\}$, with the obvious inclusion maps being the amalgamation maps.
Then
$\type(v/\bfB)=\sigma$,  $\type(w/\bfB)=\tau$,
and $\bfE\re\{a_0,a_1,v,w\}\cong \bfC$,
 so  each pair in
$\{b,v,w\}$ has an edge in $\bfE$.
But this implies that  $\bfE$ has a triangle, contradicting
$\bfE\in\mathcal{G}_3$ .
Therefore, \SFAP\ fails for $\mathcal{G}_3$.
\end{example}

The failure of  \EEAP\ for partial orders can be proved similarly, by taking $\bfC$ to have two vertices not in $\bfA$ which are unrelated to each other, and constructing $\bfB$, $\sigma$, $\tau$ so that any extension $\bfE$ satisfying $\sigma$ and $\tau$ induces a relation between any $v'$, $w'$ satisfying $\sigma$, $\tau$ respectively,
in such a way
that transitivity forces there to be a relation between $v'$ and $w'$.

  We now give an example where \SFAP\ fails in a structure
  with  a relation of arity higher than two.

\begin{example}[\SFAP\ fails for
$3$-hypergraphs forbidding
the irreducible 3-hypergraph
on four vertices with three hyper-edges]\label{ex.hyperI}
Suppose our language has one ternary relation symbol $R$.
Let
$\mathbf{I}$  denote a ``pyramid'', the structure on four vertices with exactly three hyper-edges;
that is, say $\mathrm{I}=\{i,j,k,\ell\}$ and
$\mathbf{I}$ consists of the relation
$
\{R^\mathbf{I}(i,j,k), R^\mathbf{I}(i,j,\ell),R^\mathbf{I}(i,k,\ell)\}.
$
Then every two vertices in $\mathrm{I}$ are in some relation in $\mathbf{I}$, so $\mathbf{I}$ is irreducible.
However, the triple $\{j,k,\ell\}$ is not contained in any relation in $\mathbf{I}$.

The free amalgamation class Forb$(\{\mathbf{I}\})$ does not satisfy
\SFAP:
Let $\bfA$ be the singleton $\{a\}$,
with $R^\bfA = \emptyset$,
and let $\bfC$ have universe $\{a,c_0,c_1\}$ with
$R^\bfC = \{(a,c_0,c_1)\}$.
Let $\bfB$ have universe $\{a,b\}$, and let $\sigma$ and $\tau$ both be the $1$-types $\{R(x,a,b)\}$ over $\bfB$.
Suppose that $\bfE\in$ Forb$(\{\mathbf{I}\})$
satisfies the conclusion of
Definition \ref{defn.EEAP_new}.
Then $\bfE$ has universe $\{a,b,c_0,c_1\}$ and
 $R^\bfE = \{(a,c_0,c_1), (c_0,a,b), (c_1,a,b)\}$.
 Hence $\bfE$ contains a copy of $\mathbf{I}$, contradicting that $\bfE\in$ Forb$(\{\mathbf{I}\})$.
\end{example}

\begin{rem}
The same argument shows  that
\SFAP\ fails for
any  free amalgamation class Forb$(\mathcal{F})$
 where some $\mathbf{F}\in\mathcal{F}$
 is not 3-irreducible.
\end{rem}

We now present a catalogue of many (though not all) of the known results regarding indivisibility,
 finite  big Ramsey degrees  (upper bounds), and characterizations  of exact big Ramsey degrees
(canonical partitions).
A blank box means the property  has not yet been proved or disproved.
All previously known results for  \Fraisse\ classes
in languages with relations of arity at most two
with \Fraisse\ limits satisfying \EEAP$^+$ are recovered by our Theorem \ref{thm.bounds}.
New results in this paper are indicated by the number of the theorem from which they follow.

In all cases where exact big Ramsey degrees have been characterized, this has been achieved via finding canonical partitions.
Moreover,
for structures in languages with relations of arity at most two,
these canonical partitions have been found in terms of similarity types of antichains
in
trees of $1$-types, either explicitly or implicitly.
Once one has such canonical partitions, the existence of a big Ramsey structure follows from
Theorem \ref{thm.apply} in conjunction with Zucker's Theorem 7.1 in \cite{Zucker19}.
Thus, we do not include a column for existence of big Ramsey structures.
\vfill\eject

\textbf{Key}
\vspace{2mm}

    \begin{tabular}{|l|}
    \hline
     $\bullet$
     DA: Disjoint Amalgamation\\
     $\bullet$ FA: Free Amalgamation\\
     $\bullet$ SDAP:
     strongest of \SFAP, \EEAP$^+$, or \EEAP\ known to hold\\
$\bullet$ IND: Indivisibility\\
     $\bullet$ FBRD: Finite big Ramsey degrees\\
     $\bullet$ CP: Exact big Ramsey degrees characterized via Canonical Partitions\\ \\
   \cmark \ Yes \qquad \xmark \ No \quad $\bigstar$ \ In some cases, not in all \\
    \hline
    \end{tabular}



\begin{enumerate}[leftmargin=-.01cm]
    \item $\mathbb{Q}$-like structures
    \begin{center}
        \begin{tabular}{|p{3.45cm}|c|c|c|c|c|c|c|}
        \hline
        {\bf \Fraisse\ limit} & {\bf DA} & {\bf FA}  & {\bf SDAP}&
\bf{IND} &\bf{FBRD} & \bf{CP}\\
        \hline
        $\mathbb{Q}$ with no relations & \cmark & \cmark  &
         SDAP$$
& Pigeonhole
& \cite{Ramsey30} & \cite{Ramsey30}\\
        \hline
        $(\mathbb{Q},<)$ & \cmark & \xmark  & SDAP$^+$
& Folklore
& \cite{LavUnp} & \cite{DevlinThesis}\\
        \hline
        $\mathbb{Q}_n$ & \cmark & \xmark & SDAP$^+$
&  Folklore
& \cite{Laflamme/NVT/Sauer10} & \cite{Laflamme/NVT/Sauer10}\\
        \hline
        $\mathbf{S}(2)$& \cmark & \xmark & \xmark
&  \xmark
& \cite{Laflamme/NVT/Sauer10}& \cite{Laflamme/NVT/Sauer10}\\
        \hline
        $\mathbf{S}(3)$,$\mathbf{S}(4),\cdots$& \cmark & \xmark & \xmark
& \xmark
& \cite{Barbosa20} & \cite{Barbosa20}\\
        \hline
        $\mathbb{Q}_{\mathbb{Q}}$, $\mathbb{Q}_{\mathbb{Q}_{\mathbb{Q}}}, \cdots$ & \cmark & \xmark  & SDAP$^+$
& [Thm \ref{thm.indivisibility}]
& [Thm \ref{thm.onecolorpertype}] & [Thm \ref{thm.bounds}]\\

        \hline
 \Fraisse\ limit of $\mathcal{COE}_{n,p}$
 & \cmark & \xmark  & SDAP$^+$
&  [Thm \ref{thm.indivisibility}]
& [Thm \ref{thm.onecolorpertype}] & [Thm \ref{thm.bounds}]\\
        \hline

       Main reducts of $(\bQ,<)$ & \cmark & \xmark &  SDAP$^+$
& Folklore
& \cite{Masulovic18} &  \\
      \hline
        Generic structures with two or more independent  linear relations & \cmark & \xmark & \EEAP\
& Folklore
&\cite{Hubicka_CS20} & \\
        \hline
        \end{tabular}
    \end{center}
   \vspace{4mm}

    \item Unconstrained  relational structures and their ordered expansions
        \begin{center}
        \begin{tabular}{|p{3.6cm}|c|c|c|c|c|c|c|}
                \hline
        {\bf \Fraisse\ limit} & {\bf DA} & {\bf FA} & {\bf SDAP}
&{\bf IND} & {\bf FBRD} & {\bf CP} \\
        \hline
            Rado graph & \cmark & \cmark  & SFAP
& Folklore
& \cite{Sauer06} & \cite{Sauer06}\\

            \hline
       Generic directed graph & \cmark & \cmark & SFAP
&   \cite{El-Zahar/Sauer93}
 & \cite{Laflamme/Sauer/Vuksanovic06} & \cite{Laflamme/Sauer/Vuksanovic06}\\
             \hline
             Generic tournament & \cmark & \xmark  & SDAP$^+$
&   \cite{El-Zahar/Sauer93}
& \cite{Laflamme/Sauer/Vuksanovic06} & \cite{Laflamme/Sauer/Vuksanovic06}\\
             \hline
             Generic unrestricted structures in a finite binary relational language
             & \cmark & 
	$\bigstar$                        
             & SDAP$^+$
&  \cite{Laflamme/Sauer/Vuksanovic06}
 & \cite{Laflamme/Sauer/Vuksanovic06} & \cite{Laflamme/Sauer/Vuksanovic06} \\
             \hline
Ordered expansions of any of the above structures
             & \cmark & \xmark  & SDAP$^+$
& [Thm \ref{thm.indivisibility}]
& [Thm \ref{thm.onecolorpertype}] & [Thm \ref{thm.bounds}]  \\

 \hline

            Generic $3$-uniform hypergraph & \cmark & \cmark & SFAP
&  \cite{El-Zahar/Sauer94}
& \cite{Hubicka_et4_19withproofs} &  \\
            \hline
            Generic $k$-uniform hypergraph for $k > 3$ & \cmark & \cmark  & SFAP
&  \cite{El-Zahar/Sauer94}
&  \cite{Hubicka_et4_20} &
           \\
             \hline

Generic unrestricted structures  with relations in any arity, and their ordered expansions
             & \cmark & 
$\bigstar$                 
             & SDAP$^+$
& [Thm \ref{thm.indivisibility}]
& &  \\
             \hline
        \end{tabular}
    \end{center}
    \vspace{4mm}

\eject

    \item Constrained  structures with relations of arity at most two
        \begin{center}
        \begin{tabular}{|p{4.2cm}|c|c|c|c|c|c|c|}
                \hline
        {\bf \Fraisse\ limit} & {\bf DA} & {\bf FA} & {\bf SDAP} &  {\bf IND} &{\bf FBRD} & \bf{ CP} \\
        \hline
      Generic bipartite & \cmark & \cmark  & SFAP
&  Folklore
& \cite{HoweThesis} & [Thm \ref{thm.bounds}] \\
            \hline
             Generic $n$-partite for $n\geq 3$ & \cmark & \cmark & SFAP
&  Folklore
& \cite{Zucker20} & [Thm \ref{thm.bounds}]\\
            \hline
        Generic $K_3$-free graphs & \cmark & \cmark & \xmark
& \cite{Komjath/Rodl86}
& \cite{DobrinenJML20} & \cite{Balko7}\\
        \hline
        Generic $K_n$-free graphs for finite $n>3$ & \cmark & \cmark & \xmark
& \cite{El-Zahar/Sauer89}
& \cite{DobrinenH_k19} & \cite{Balko7}\\
        \hline
 \Fraisse\ limits with free amalgamation that are ``rank linear''
  & \cmark & \cmark & 
  $\bigstar$
& \cite{Sauer03}
 & \cite{Zucker20} & \cite{Balko7}\\
        \hline
        
\Fraisse\ limit of Forb$(\mathcal{F})$,
where
$\mathcal{F}$ is a finite  set of finite irreducible
structures

& \cmark & \cmark & 
$\bigstar$
& 
$\bigstar$
 & \cite{Zucker20} & \cite{Balko7}\\
        \hline
         Generic poset & \cmark & \xmark & \xmark
&  Folklore
& \cite{Hubicka_CS20} &  \\
        \hline
        \end{tabular}
    \end{center}
    \vspace{4mm}

   \item Constrained 
arbitrary  
   arity relational structures
   \vspace{4mm}
      \begin{center}
        \begin{tabular}{|p{4.3cm}|c|c|c|c|c|c|c|}
                \hline
        {\bf \Fraisse\ limit} & {\bf DA} & {\bf FA} & {\bf SDAP} & {\bf IND} & \bf{FBRD} & \bf{CP} \\
       \hline
        Generic $k$-hypergraph omitting a finite set of  finite $3$-irreducible  $k$-hypergraphs
for $k \ge 3$
         & \cmark & \cmark & 
SFAP                 
& \cite{El-Zahar/Sauer94}  & &  \\
        \hline
       \Fraisse\ limit of        $\Forb(\mathcal{F})$    where   all $F \in \mathcal{F}$ are irreducible and $3$-irreducible
      & \cmark & \cmark & SFAP
& 
\cite{El-Zahar/Sauer94}, [Thm \ref{thm.indivisibility}]
& & \\
             \hline
            \Fraisse\ limit of
            $\Forb(\mathcal{F})^<$,
   where all $F \in \mathcal{F}$ are
   irreducible and
   $3$-irreducible
   & \cmark & \xmark & SDAP$^+$
& [Thm \ref{thm.indivisibility}]
& &\\
             \hline

\Fraisse\ limits with free amalgamation that are ``rank linear''
   & \cmark & \cmark  & 
 $\bigstar$  
& \cite{Sauer20}
& &  \\
             \hline

        \end{tabular}
    \end{center}
\end{enumerate}
\ \vskip.1in

\begin{rem}
Sauer's result in \cite{Sauer20} is the state of the art for  indivisibility  of
 free amalgamation classes:  He proved that a
\Fraisse\ structure with free amalgamation in a finite relational language
is indivisible if and only if it is
what he calls ``rank linear''.
For the definition of rank linear, see Definition 1.3 in \cite{Sauer20}.
\end{rem}

\begin{rem}
Results on indivisibility and big Ramsey degrees of metric spaces appear in
\cite{DLPS07},
\cite{DLPS08},
 \cite{Hubicka_CS20},
\cite{Masulovic18}, \cite{Masulovic_RBS20},
\cite{NVTThesis},
\cite{NVT08}, \cite{NVTMem10}, and \cite{NVTSauer09}.
In his PhD thesis \cite{NVTThesis},
Nguyen Van Th\'{e} proved
results on indivisibility of Urysohn spaces
which were later published in \cite{NVTMem10}, including
that all Urysohn spaces with distance set $S$ of size four are indivisible (except for $S=\{1,2,3,4\}$).
A characterization of those
countable ultrametric spaces which are homogeneous and indivisible
was proved by
Delhomm\'{e}, Laflamme, Pouzet, and Sauer in
 \cite{DLPS08}.
Nguyen Van Th\'{e} showed finite big Ramsey degrees for finite $S$-submetric spaces of ultrametric $S$-spaces
in
\cite{NVT08}, with $S$ finite and nonnegative.
In
\cite{NVTSauer09},
Nguyen Van Th\'{e} and Sauer   proved that for each integer $m\ge 1$, the countable homogenous metric space with distances in $\{1,\dots, m\}$ is indivisible.
Sauer established indivisibility
 of Urysohn $S$-metric spaces with S finite  in \cite{Sauer12}.
Ma\v{s}ulovi\'{c}
proved finite big Ramsey degrees for
 Urysohn $S$-metric spaces,
 where $S$ is a finite distance set with no internal jumps and a property called ``compactness'' in that paper, meaning that the distances are not too far apart.
Recently, Hubi\v{c}ka  extended this to all  Urysohn $S$-metric spaces where
$S$ is tight in addition to finite and nonnegative \cite{Hubicka_CS20}.
As \EEAP\ fails for non-trivial metric spaces (for the same reason it fails for the triangle-free graphs and partial orders),  we mention no details here.
\end{rem}


\section{Coding trees of $1$-types for \Fraisse\ structures}\label{sec.sct}

Fix throughout a \Fraisse\ class $\mathcal{K}$ in a finite  relational language $\mathcal{L}$.
Recall that  $\bK$  denotes
an {\em enumerated \Fraisse\ limit}
for $\mathcal{K}$,  meaning that
 $\bK$ has universe $\om$.
In order to  avoid confusion, we shall usually use $v_n$ instead of just $n$
to denote the $n$-th member of the universe  of $\bK$, and we shall call this  the {\em $n$-th  vertex} of $\bK$.
For $n < \om$,  we write $\bK_n$, and sometimes $\bK\re n$, to denote the
substructure of $\bK$ on the set of vertices $\{v_i:i<n\}$.
We call $\bK_n$ an {\em initial segment} of $\bK$.
Note that $\bK_0$ is the empty structure.

In Subsection \ref{subsec.3.1},
we present
a general  construction of
 trees of complete $1$-types over initial segments of $\bK$,
which we call {\em coding trees}.
Graphics of coding trees
are then presented for various  prototypical \Fraisse\ classes satisfying \EEAP$^+$.
In Subsection \ref{subsec.3.2},
we define
{\em passing types}, extending  the notion
of  {\em passing number}
due to Laflamme, Sauer, and Vuksanovic in \cite{Laflamme/Sauer/Vuksanovic06},  which has been central to all prior results on big Ramsey degrees   for binary relational structures.
Then we
extend the notion
from \cite{Laflamme/Sauer/Vuksanovic06}
of {\em similarity type} for binary relational structures
to
structures with relations of any arity.
In Subsection \ref{subsec.3.3}, we introduce
 {\em diagonal}  coding trees.
These  will be
 key
to obtaining
 precise big Ramsey degree results without appeal to
 the method of
 envelopes.
We
define the \emph{Diagonal Coding Tree Property}, one of the conditions for \EEAP$^+$ to hold, and
then
present the coding tree version of \EEAP$^+$ (Definition \ref{def.EEAPCodingTree}),
 from which the formulation of \EEAP\ in Definition \ref{defn.EEAP_new} was extracted.
 We show how to construct diagonal  subtrees    representing copies of various \Fraisse\ structures,
 completing  the proofs  of
 the Diagonal Coding Tree Property
  in Theorem \ref{thm.supercool} and Propositions
 \ref{prop.LO_n},
\ref{bQn}, and
\ref{prop.loe}.


\subsection{Coding trees of $1$-types}\label{subsec.3.1}

All types will be quantifier-free 1-types, with variable $x$, over some finite initial segment of $\bK$.
For $n \ge 1$, a type over $\bK_n$ must contain the formula
$\neg (x = v_i ) $ for each $i < n$.
Given a type $s$ over $\bK_n$, for any $i < n$, $s \re \bK_i$ denotes the restriction of $s$ to parameters from $\bK_i$.
Recall that the notation ``tp'' denotes a complete quantifer-free 1-type.

\begin{defn}[The  Coding  Tree of $1$-Types, $\bS(\bK)$]\label{defn.treecodeK}
The {\em coding  tree of $1$-types}
$\bS(\bK)$
for an enumerated \Fraisse\ structure $\bK$
 is the set of all complete
  $1$-types over initial segments of $\bK$
along with a function $c:\om\ra \bS(\bK)$ such that
$c(n)$ is the
$1$-type
of $v_n$
over $\bK_n$.
The tree-ordering is simply inclusion.
\end{defn}

We shall usually simply write $\bS$, rather than $\bS(\bK)$.
Note that we make no requirement at this point on $\bK$;
an enumerated \Fraisse\ limit of  any \Fraisse\ class (with no reference to its amalgamation or Ramsey properties)
naturally induces a  coding tree of $1$-types as above.
We say that $c(n)$ {\em represents} or {\em codes}  the vertex  $v_n$.
Instead of writing $c(n)$, we shall usually write $c_n$ for the $n$-th coding node  in $\bS$.

We let  $\bS(n)$
denote
 the collection
 of all   $1$-types  $\type(v_i/\bK_n)$, where
 $i\ge n$.
 Note that each $c(n)$ is a node in $\bS(n)$.
The set  $\bS(0)$ consists of
 the
 $1$-types over the empty structure $\bK_0$.
For $s\in \bS(n)$, the immediate successors of $s$   are  exactly those $t\in \bS(n+1)$ such that $s\sse t$.
For each $n<\om$, the set  $\bS(n)$ is finite, since the language
$\mathcal{L}$  consists of finitely many finitary relation symbols.

We say that each  node $s\in\bS(n)$ has {\em length}  $n+1$, and denote the length of $s$ by $|s|$.
Thus, all nodes in $\bS$ have length at least one.
While it is slightly unconventional to consider the roots
of $\bS$ as having length one,  this approach lines up with the  natural correspondence between  nodes in $\bS$ and certain sequences of partial $1$-types
that  we define
in the next paragraph.
The reader wishing for a  tree starting with a node of length zero  may consider adding the
empty set  to $\bS$, as this will have no
effect
on the results in this paper.
A {\em level set} is a subset $X\sse\bS$ such that all nodes in $X$ have the same length.

Let  $n<\om$ and   $s\in\bS(n)$ be given.
We let  $s(0)$  denote the set of formulas in $s$ involving
no parameters;
$s(0)$ is the unique member of $\bS(0)$ such that $s(0)\sse s$.
For  $1\le i\le n$,
we let
 $s(i)$ denote  the set of those formulas in $s\re \bK_i$
in which $v_{i-1}$ appears; in other words, the formulas in $s \re \bK_i$ that are not in  $s \re \bK_{i-1}$.
In this manner,  each $s\in \bS$ determines a unique sequence
$\lgl s(i):i< |s|\rgl$, where $\{s(i):i< |s|\}$ forms a partition of $s$.
For $j< |s|$, $\bigcup_{i\le j}s(i)$  is the  node in $\bS(j)$
such that $\bigcup_{i\le  j}s(i)\sse s$.
For $\ell\le |s|$,
we shall usually write $s \re \ell$ to denote $\bigcup_{i<\ell}s(i)$.

Given $s,t\in \bS$,
we define the {\em meet} of $s$ and $t$, denoted $s\wedge t$, to be
$s \re \bK_m$ for the maximum  $m\le \min(|s|,|t|)$ such that $s \re \bK_m=t \re \bK_m$.
It can be useful to think of  $s\in \bS$ as the sequence $\lgl s(0),\dots, s(|s|-1)\rgl$;
then $s\wedge t$ can be interpreted in the  usual way for trees of sequences.

It will be useful later  to have  specific notation for unary relations.
We will let $\Gamma$
denote $\bS(0)$,  the set of  complete
$1$-types
over the empty set
that are realized in $\bK$.
If $\mathcal{L}$ has no unary relation symbols, then $\Gamma$ will consist exactly of the ``trivial'' 1-type
which is satisfied by every element of $\bK$.
For $\gamma\in\Gamma$, we write
``$\gamma(v_n)$ holds in $\bK$'' when $\gamma$ is the 1-type of $v_n$ over the empty set; in practice,
it will be the unary relation symbols in $\gamma$ (if there are any) that will be of interest to us.

 \begin{rem}
 Our definition of $s(i)$
sets up   for the   definition of passing type
in Subsection \ref{subsec.3.2},
which  directly abstracts the notion of passing number used in  \cite{Sauer06} and
  \cite{Laflamme/Sauer/Vuksanovic06}, and in subsequent papers building on their ideas.
\end{rem}

\begin{rem}
In the case where all relation symbols in the language $\mathcal{L}$ have arity at most two,
  the  coding  tree  of $1$-types $\bS$  has bounded branching.
  If $\mathcal{L}$ has any relation symbol of arity three or greater, then  $\bS$ may
  have branching which increases as the levels increase.
If such a \Fraisse\ class satisfies \EEAP,  sometimes more work  still  must be  done  in order to guarantee
that its \Fraisse\ limit has
\EEAP$^+$.
\end{rem}


We now provide graphics for   coding trees of $1$-types which are prototypical for the \Fraisse\ classes which we proved in Section \ref{sec.EEAPClasses} to  have \EEAP.
We start with the
rational linear order,
since its coding tree of $1$-types is the simplest, and  also because  the rationals were the first \Fraisse\ structure for which big Ramsey degrees were characterized (Devlin, \cite{DevlinThesis}).

\begin{example}[The coding tree of $1$-types  $\bS(\bQ)$]\label{ex.Q}
Figure \ref{fig.Qtree}
shows the coding tree of $1$-types for  $(\bQ,<)$, the rationals as a linear order.
This is the \Fraisse\ limit of $\mathcal{LO}$, the class of finite linear orders.
We assume that the universe  of $\bQ$  is linearly ordered in order-type $\om$ as $\lgl v_n:n<\om\rgl$.
For each $n$, the coding node $c_n$
is the $1$-type of
vertex $v_n$  over the initial segment $\{v_i:i<n\}$ of $\bQ$.
(Recall that $x$ is the variable in all  of our $1$-types.)
Thus, the coding node $c_0$ is the empty $1$-type, and
 $c_1$   is the  $1$-type
$\{v_0<x\}$.
Thus,  the coding nodes $\{c_0,c_1\}$ represent the linear order $v_0<v_1$.
Likewise, the coding node $c_2$ is  the $1$-type
$\{x<v_0,x<v_1\}$
 over the linear order $v_0<v_1$.
Hence, $c_2$ represents the vertex $v_2$ satisfying
$v_2<v_0<v_1$.
The coding node $c_3$ is the $1$-type $\{v_0<x,x<v_1,v_2<x\}$, so  $c_3$ represents the vertex $v_3$ satisfying
$v_2<v_0<v_3<v_1$.
Below the  tree, we  picture the linear order on the vertices $v_0,\dots,v_5$ induced by the coding nodes.
As the tree grows in height, the linear order represented by the coding nodes  grows into  the countable dense linear order with no endpoints.

Notice that only the coding nodes branch.
This is because
of the rigidity of the rationals:
Given  a non-coding node $s$ on the same level as a coding node $c_{n}$ (say $n\ge 1$),
$s$ is a  $1$-type which is satisfied by any vertex
which lies in some interval determined by the vertices $\{v_i:i< n\}$, and  $v_n$ is not in that interval.
Thus, the order between $v_n$ and any vertex satisfying $s$ is
predetermined,
so $s$ does not split.
Said another way,
letting $m$ denote the length of the meet of  $c_n$ and $s$,
$c_n$  and $s$ must disagree on the formula $x<v_{m}$;
hence,  $x<v_m$ is in $c_n$   if and only if  $v_m<x$ is in
  $s$.
In the case that the formula $x<v_m$ is in $c_n$,
then
it follows that  $v_n<v_m$.
On the other hand, any realization $v_i$ of the $1$-type $s$ must satisfy $v_m<v_i$.
Hence  every realization of $s$ by some vertex  $v_i$ must satisfy $v_n<v_i$.
Thus, there is only one immediate successor of $s$ in the tree of $1$-types.
The tree of $1$-types for $\bQ$
 eradicates the extraneous structure
which appears in the more traditional approach of using the full binary branching tree and Milliken's Theorem  to approach big Ramsey degrees of the rationals.
\end{example}


\begin{figure}
\begin{tikzpicture}[grow'=up,scale=.6]
\tikzstyle{level 1}=[sibling distance=4in]
\tikzstyle{level 2}=[sibling distance=2in]
\tikzstyle{level 3}=[sibling distance=1in]
\tikzstyle{level 4}=[sibling distance=0.5in]
\tikzstyle{level 5}=[sibling distance=0.2in]
\tikzstyle{level 6}=[sibling distance=0.1in]
\tikzstyle{level 7}=[sibling distance=0.07in]
\node [label=$c_0$] {} coordinate (t9)
child{ coordinate (t0) edge from parent[color=black,thick]
child{ coordinate (t00) edge from parent[color=black,thick]
child{ coordinate (t000) edge from parent[color=black,thick]
child{coordinate (t0000) edge from parent[color=black,thick]
child{coordinate (t00000) edge from parent[color=black,thick]
child{coordinate (t000000) edge from parent[color=black,thick]}
}
}
}
child{ coordinate (t001) edge from parent[color=black,thick]
child{ coordinate (t0010) edge from parent[color=black,thick]
child{coordinate (t00100) edge from parent[color=black,thick]
child{coordinate (t001000) edge from parent[color=black,thick]}
child{coordinate (t001001) edge from parent[color=black,thick]}
}
}
}
}
}
child{ coordinate (t1) edge from parent[color=black,thick]
child{ coordinate (t10) edge from parent[color=black,thick]
child{ coordinate (t100)
child{coordinate (t1000)
child{coordinate (t10000)
child{coordinate (t100000)}
}
}
child{coordinate (t1001)
child{coordinate (t10010)
child{coordinate (t100100)}
}
}
}
}
child{ coordinate (t11) [label={\small 1/16}]
child{coordinate (t110)
child{coordinate (t1100)
child{coordinate (t11000)
child{coordinate (t110000)}
}
child{coordinate (t11001)
child{coordinate (t110010)}
}
}
}
}
}
;

\node[circle, fill=blue,inner sep=0pt, minimum size=5pt] at (t9) {};
\node[circle, fill=blue,inner sep=0pt, minimum size=5pt,label=0:$c_2$,label=250:$\scriptstyle{x<v_1}$] at (t00) {};
\node[circle, fill=blue,inner sep=0pt, minimum size=5pt,label=0:$c_5$,label=280:${\scriptstyle x<v_4}$] at (t00100) {};
\node[circle, fill=blue,inner sep=0pt, minimum size=5pt, label=$c_1$,label=250:$\scriptstyle{v_0<x}$] at (t1) {};
\node[circle, fill=blue,inner sep=0pt, minimum size=5pt, label = 0:$c_3$,label=240:$\scriptstyle{v_2<x}$] at (t100) {};
\node[circle, fill=blue,inner sep=0pt, minimum size=5pt,label = 0:$c_4$,label=300:$\scriptstyle{v_3<x}$] at (t1100) {};
\node[label=280:$\scriptstyle{x<v_0}$] at (t0) {};
\node[label=270:$\scriptstyle{v_1 < x}$] at (t11) {};
\node[label=270:$\scriptstyle{x<v_1}$] at (t10) {};
\node[label=300:$\scriptstyle{v_2 < x}$] at (t110) {};
\node[label=240:${\scriptstyle x<v_3}$] at (t1000) {};
\node[label=280:${\scriptstyle v_3 < x}$] at (t1001) {};
\node[label=240:${\scriptstyle x<v_4}$] at (t10000) {};
\node[label=300:${\scriptstyle x < v_4}$] at (t10010) {};
\node[label=240:${\scriptstyle v_5<x}$] at (t100000) {};
\node[label=300:${\scriptstyle v_5<x}$] at (t100100) {};
\node[label=260:${\scriptstyle x<v_4}$] at (t11000) {};
\node[label=280:${\scriptstyle v_4<x}$] at (t11001) {};
\node[label=260:${\scriptstyle v_5<x}$] at (t110000) {};
\node[label=280:${\scriptstyle v_5<x}$] at (t110010) {};
\node[label=260:${\scriptstyle x<v_2}$] at (t000) {};
\node[label=280:${\scriptstyle v_2<x}$] at (t001) {};
\node[label=260:${\scriptstyle x<v_3}$] at (t0000) {};
\node[label=280:${\scriptstyle x<v_3}$] at (t0010) {};
\node[label=260:${\scriptstyle x<v_4}$] at (t00000) {};
\node[label=260:${\scriptstyle x<v_5}$] at (t000000) {};
\node[label=260:${\scriptstyle x<v_5}$] at (t001000) {};
\node[label=280:${\scriptstyle v_5<x}$] at (t001001) {};
\node[label=280:${\scriptstyle v_5<x}$] at (t001001) {};

\node[circle,inner sep=0pt, minimum size=5pt,label=90:$\vdots$] at (t000000) {};
\node[circle,inner sep=0pt, minimum size=5pt,label=90:$\vdots$] at (t001000) {};
\node[circle,inner sep=0pt, minimum size=5pt,label=90:$\vdots$] at (t001001) {};
\node[circle,inner sep=0pt, minimum size=5pt,label=90:$\vdots$] at (t100000) {};
\node[circle,inner sep=0pt, minimum size=5pt,label=90:$\vdots$] at (t100100) {};
\node[circle,inner sep=0pt, minimum size=5pt,label=90:$\vdots$] at (t100100) {};
\node[circle,inner sep=0pt, minimum size=5pt,label=90:$\vdots$] at (t110000) {};
\node[circle,inner sep=0pt, minimum size=5pt,label=90:$\vdots$] at (t110010) {};

\node[circle, fill=white,draw,inner sep=0pt, minimum size=4pt] at (t0) {};
\node[circle, fill=white,draw,inner sep=0pt, minimum size=4pt] at (t000) {};
\node[circle, fill=white,draw,inner sep=0pt, minimum size=4pt] at (t0000) {};
\node[circle, fill=white,draw,inner sep=0pt, minimum size=4pt] at (t00000) {};
\node[circle, fill=white,draw,inner sep=0pt, minimum size=4pt] at (t000000) {};
\node[circle, fill=white,draw,inner sep=0pt, minimum size=4pt] at (t001) {};
\node[circle, fill=white,draw,inner sep=0pt, minimum size=4pt] at (t0010) {};
\node[circle, fill=white,draw,inner sep=0pt, minimum size=4pt] at (t001000) {};
\node[circle, fill=white,draw,inner sep=0pt, minimum size=4pt] at (t001001) {};
\node[circle, fill=white,draw,inner sep=0pt, minimum size=4pt] at (t10) {};
\node[circle, fill=white,draw,inner sep=0pt, minimum size=4pt] at (t1000) {};
\node[circle, fill=white,draw,inner sep=0pt, minimum size=4pt] at (t10000) {};
\node[circle, fill=white,draw,inner sep=0pt, minimum size=4pt] at (t100000) {};
\node[circle, fill=white,draw,inner sep=0pt, minimum size=4pt] at (t1001) {};
\node[circle, fill=white,draw,inner sep=0pt, minimum size=4pt] at (t10010) {};
\node[circle, fill=white,draw,inner sep=0pt, minimum size=4pt] at (t100100) {};
\node[circle, fill=white,draw,inner sep=0pt, minimum size=4pt] at (t11) {};
\node[circle, fill=white,draw,inner sep=0pt, minimum size=4pt] at (t110) {};
\node[circle, fill=white,draw,inner sep=0pt, minimum size=4pt] at (t11000) {};
\node[circle, fill=white,draw,inner sep=0pt, minimum size=4pt] at (t110000) {};
\node[circle, fill=white,draw,inner sep=0pt, minimum size=4pt] at (t11001) {};
\node[circle, fill=white,draw,inner sep=0pt, minimum size=4pt] at (t110010) {};

\node[circle, fill=blue,inner sep=0pt, minimum size=5pt, label=$v_2$,below=3cm of t00] (v2) {};
\node[circle, fill=blue,inner sep=0pt, minimum size=5pt,label=$v_5$,below =3.9cm of t001] (v5) {};
\node[circle, fill=blue,inner sep=0pt, minimum size=5pt,label=$v_0$,below =1.25cm of t9] (v0) {};
\node[circle, fill=blue,inner sep=0pt, minimum size=5pt,label=$v_3$,below =3.95cm of t100] (v3) {};
\node[circle, fill=blue,inner sep=0pt, minimum size=5pt,label=$v_1$,below =2.1cm of t1] (v1) {};
\node[circle, fill=blue,inner sep=0pt, minimum size=5pt,label=$v_4$,below =3cm of t11] (v4) {};
\end{tikzpicture}
\caption{Coding tree of $1$-types for $(\bQ,<)$ and the  linear order represented by its coding nodes.}\label{fig.Qtree}
\end{figure}


\begin{example}[The coding tree of $1$-types $\bS(\bQ_2)$]\label{ex.ctbQ_2}
Next, we consider coding trees of $1$-types for linear orders with equivalence relations with finitely many equivalence classes, each of which is dense in the linear order.
Figure \ref{fig.Q2tree} provides a graphic for   the coding tree of $1$-types for the structure
$\bQ_2$,  the rationals with an equivalence relation with two equivalence classes which are each dense in  the linear order.
This  is the \Fraisse\ limit of $\mathcal{P}_2$
discussed just before  Lemma \ref{bQn}.
We point out that
$c_0$ is the $1$-type $\{U_1(x)\}$,
$c_1$ is the $1$-type $\{U_0(x), x<v_0\}$,
$c_2$ is the $1$-type $\{U_0(x), v_0<x, v_1<x\}$, etc.

Note  that $\bS(\bQ_2)$ looks like two identical disjoint copies of a coding tree for $\bQ$.
This is because each of the two unary relations, representing the two equivalence classes, appears densely in the linear order.
The ordered structure $\bQ_2$ appears below the two trees as the vertices $v_0,v_1,\dots$.
Unlike Figure \ref{fig.Qtree} for $\bQ$, the  vertices in $\bQ_2$ do not line up below the coding nodes in the trees representing them, since $\bS(\bQ_2)$ has two roots.
However, if we modify our definition of coding tree of $1$-types to  have individual  coding nodes  $c_n$ represent the unary relations satisfied by $v_n$
(rather than $\bS$  having   $|\Gamma|$ many roots),
this has the effect of producing a one-rooted tree with
``$\gamma$-colored'' coding nodes appearing cofinally in the tree, for each $\gamma\in\Gamma$.
This approach then shows the linear order $\bQ_2$ lining up below the coding nodes,  recovers the characterization of the big Ramsey degrees in \cite{Laflamme/NVT/Sauer10}, and
 will aid us in proving \EEAP$^+$ for $\mathcal{P}_2$.
 (See Definition \ref{defn.ctU}
 for this variation of tree of $1$-types, which reproduces the approach in \cite{Laflamme/NVT/Sauer10}.)
The formulation given in Definition \ref{defn.treecodeK} makes clear the distinction  between  exact big Ramsey degrees for \Fraisse\ classes with \SFAP\  and those with \SFAP\ plus a linear order, and sets up correctly for proving  lower bounds in general.

Similarly, for any $n\ge 2$,  $\bS(\bQ_n)$ will have  $n$ roots, and  above each root, the $n$ trees will  be copies of each other.
\end{example}

\begin{figure}
\begin{center}
\begin{minipage}{.2\textwidth}
\hspace*{-4cm}
\begin{tikzpicture}[grow'=up,scale=.3]
\tikzstyle{level 1}=[sibling distance=4in]
\tikzstyle{level 2}=[sibling distance=2in]
\tikzstyle{level 3}=[sibling distance=1in]
\tikzstyle{level 4}=[sibling distance=0.5in]
\tikzstyle{level 5}=[sibling distance=0.2in]
\tikzstyle{level 6}=[sibling distance=0.1in]
\tikzstyle{level 7}=[sibling distance=0.07in]

\node {} coordinate (t9)
child{ coordinate (t0) edge from parent[color=black,thick]
child{coordinate (t00) edge from parent[color=black,thick]
child{coordinate (t000) edge from parent[color=black,thick]
child{coordinate (t0000) edge from parent[color=black,thick]
child{coordinate (t00000) edge from parent[color=black,thick]}
}
child{coordinate (t0001) edge from parent[color=black,thick]
child{coordinate (t00010) edge from parent[color=black,thick]}
}
}
}
child{coordinate (t01) edge from parent[color=black,thick]
child{coordinate (t010) edge from parent[color=black,thick]
child{coordinate (t0100) edge from parent[color=black,thick]
child{coordinate (t01000) edge from parent[color=black,thick]}
}
}
}
}
child{ coordinate (t1) edge from parent[color=black,thick]
child{coordinate (t10) edge from parent[color=black,thick]
child{coordinate (t100) edge from parent[color=black,thick]
child{coordinate (t1000) edge from parent[color=black,thick]
child{coordinate (t10000) edge from parent[color=black,thick]}
child{coordinate (t10001) edge from parent[color=black,thick]}
}
}
child{coordinate (t101) edge from parent[color=black,thick]
child{coordinate (t1010) edge from parent[color=black,thick]
child{coordinate (t10100) edge from parent[color=black,thick]}
}
}
}
};
%
\node[label=90:$\scriptstyle{c_1}$,label=10:\hspace*{3mm}$\scriptstyle{v_1<x}$,label=170:$\scriptstyle{x<v_1}$\hspace*{3mm},circle, fill=red,inner sep=0pt, minimum size=5pt] at (t0) {};
\node[label=90:$\scriptstyle{c_2}$,label=150:$\scriptstyle{x<v_2}\hspace{2mm}$,label=20:$\hspace{2mm}\scriptstyle{v_2<x}$,circle, fill=red,inner sep=0pt, minimum size=5pt] at (t10) {};
\node[label=0:$\scriptstyle{c_5}$,circle, fill=red,inner sep=0pt, minimum size=5pt] at (t00010) {};

\node[circle,inner sep=0pt, minimum size=5pt,label=90:$\vdots$] at (t00000) {};
\node[circle,inner sep=0pt, minimum size=5pt,label=90:$\vdots$] at (t00010) {};
\node[circle,inner sep=0pt, minimum size=5pt,label=90:$\vdots$] at (t01000) {};
\node[circle,inner sep=0pt, minimum size=5pt,label=90:$\vdots$] at (t10000) {};
\node[circle,inner sep=0pt, minimum size=5pt,label=90:$\vdots$] at (t10001) {};
\node[circle,inner sep=0pt, minimum size=5pt,label=90:$\vdots$] at (t10100) {};

\node[label=180:$\scriptstyle{x<v_0}\hspace{6mm}$,label=270:$\scriptstyle{U_0(x)}$,label=0:$\hspace{6mm}\scriptstyle{v_0<x}$,circle, fill=white,draw,inner sep=0pt, minimum size=4pt] at (t9) {};
\node[label=90:$\scriptstyle{x<v_2}\hspace{14mm}$,circle, fill=white,draw,inner sep=0pt, minimum size=4pt] at (t00) {};
\node[circle, fill=white,draw,inner sep=0pt, minimum size=4pt] at (t000) {};
\node[circle, fill=white,draw,inner sep=0pt, minimum size=4pt] at (t0000) {};
\node[circle, fill=white,draw,inner sep=0pt, minimum size=4pt] at (t00000) {};
\node[circle, fill=white,draw,inner sep=0pt, minimum size=4pt] at (t0001) {};
\node[label=90:$\hspace{11mm}\scriptstyle{x<v_2}$,circle, fill=white,draw,inner sep=0pt, minimum size=4pt] at (t01) {};
\node[circle, fill=white,draw,inner sep=0pt, minimum size=4pt] at (t010) {};
\node[circle, fill=white,draw,inner sep=0pt, minimum size=4pt] at (t0100) {};
\node[circle, fill=white,draw,inner sep=0pt, minimum size=4pt] at (t01000) {};
\node[label=20:$\scriptstyle{v_1<x}$,circle, fill=white,draw,inner sep=0pt, minimum size=4pt] at (t1) {};
\node[circle, fill=white,draw,inner sep=0pt, minimum size=4pt] at (t100) {};
\node[circle, fill=white,draw,inner sep=0pt, minimum size=4pt] at (t1000) {};
\node[circle, fill=white,draw,inner sep=0pt, minimum size=4pt] at (t10000) {};
\node[circle, fill=white,draw,inner sep=0pt, minimum size=4pt] at (t10001) {};
\node[circle, fill=white,draw,inner sep=0pt, minimum size=4pt] at (t101) {};
\node[circle, fill=white,draw,inner sep=0pt, minimum size=4pt] at (t1010) {};
\node[circle, fill=white,draw,inner sep=0pt, minimum size=4pt] at (t10100) {};
\end{tikzpicture}
\end{minipage}
\begin{minipage}{0.2\textwidth}
\hspace*{-0.5cm}
\begin{tikzpicture}[grow'=up,scale=.3]
\tikzstyle{level 1}=[sibling distance=4in]
\tikzstyle{level 2}=[sibling distance=2in]
\tikzstyle{level 3}=[sibling distance=1in]
\tikzstyle{level 4}=[sibling distance=0.5in]
\tikzstyle{level 5}=[sibling distance=0.2in]
\tikzstyle{level 6}=[sibling distance=0.1in]
\tikzstyle{level 7}=[sibling distance=0.07in]

\node {} coordinate (u9)
child{coordinate (u0) edge from parent[color=black,thick]
child{coordinate (u00) edge from parent[color=black,thick]
child{coordinate (u000) edge from parent[color=black,thick]
child{coordinate (u0000) edge from parent[color=black,thick]
child{coordinate (u00000) edge from parent[color=black,thick]}
}
child{coordinate (u0001) edge from parent[color=black,thick]
child{coordinate (u00010) edge from parent[color=black,thick]}
}
}
}
child{coordinate (u01) edge from parent[color=black,thick]
child{coordinate (u010) edge from parent[color=black,thick]
child{coordinate (u0100) edge from parent[color=black,thick]
child{coordinate (u01000) edge from parent[color=black,thick]}
}
}
}
}
child{coordinate (u1) edge from parent[color=black,thick]
child{coordinate (u10) edge from parent[color=black,thick]
child{coordinate (u100) edge from parent[color=black,thick]
child{coordinate (u1000) edge from parent[color=black,thick]
child{coordinate (u10000) edge from parent[color=black,thick]}
child{coordinate (u10001) edge from parent[color=black,thick]}
}
}
child{coordinate (u101) edge from parent[color=black,thick]
child{coordinate (u1010) edge from parent[color=black,thick]
child{coordinate (u10100) edge from parent[color=black,thick]}
}
}
}
};

\node[label=90:$\scriptstyle{c_0}$, label=180:$\scriptstyle{x<v_0}\hspace{6mm}$,label=0:$\hspace{6mm}\scriptstyle{v_0<x}$,label=270:$\scriptstyle{U_1(x)}$,circle, fill=blue,inner sep=0pt, minimum size=5pt] at (u9) {};
\node[label=0:$\scriptstyle{c_3}$,circle, fill=blue,inner sep=0pt, minimum size=5pt] at (u000) {};
\node[label=180:$\scriptstyle{c_4}$,circle, fill=blue,inner sep=0pt, minimum size=5pt] at (u1000) {};

\node[circle,inner sep=0pt, minimum size=5pt,label=90:$\vdots$] at (u00000) {};
\node[circle,inner sep=0pt, minimum size=5pt,label=90:$\vdots$] at (u00010) {};
\node[circle,inner sep=0pt, minimum size=5pt,label=90:$\vdots$] at (u01000) {};
\node[circle,inner sep=0pt, minimum size=5pt,label=90:$\vdots$] at (u10000) {};
\node[circle,inner sep=0pt, minimum size=5pt,label=90:$\vdots$] at (u10001) {};
\node[circle,inner sep=0pt, minimum size=5pt,label=90:$\vdots$] at (u10100) {};

\node[label=150:$\scriptstyle{x<v_1}\hspace{5mm}$,label=30:$\hspace{5mm}\scriptstyle{v_1<x}$,circle, fill=white,draw,inner sep=0pt, minimum size=4pt] at (u0) {};
\node[label=120:$\scriptstyle{x<v_2}$,circle, fill=white,draw,inner sep=0pt, minimum size=4pt] at (u00) {};
\node[circle, fill=white,draw,inner sep=0pt, minimum size=4pt] at (u0000) {};
\node[circle, fill=white,draw,inner sep=0pt, minimum size=4pt] at (u00000) {};
\node[circle, fill=white,draw,inner sep=0pt, minimum size=4pt] at (u0001) {};
\node[circle, fill=white,draw,inner sep=0pt, minimum size=4pt] at (u00010) {};
\node[label=30:$\scriptstyle{x<v_2}$,circle, fill=white,draw,inner sep=0pt, minimum size=4pt] at (u01) {};
\node[circle, fill=white,draw,inner sep=0pt, minimum size=4pt] at (u010) {};
\node[circle, fill=white,draw,inner sep=0pt, minimum size=4pt] at (u0100) {};
\node[circle, fill=white,draw,inner sep=0pt, minimum size=4pt] at (u01000) {};
\node[label=30:$\scriptstyle{v_1<x}$,circle, fill=white,draw,inner sep=0pt, minimum size=4pt] at (u1) {};
\node[label=150:$\scriptstyle{x<v_2}\hspace{2mm}$,label=30:$\hspace{2mm}\scriptstyle{v_2<x}$,circle, fill=white,draw,inner sep=0pt, minimum size=4pt] at (u10) {};
\node[circle, fill=white,draw,inner sep=0pt, minimum size=4pt] at (u100) {};
\node[circle, fill=white,draw,inner sep=0pt, minimum size=4pt] at (u10000) {};
\node[circle, fill=white,draw,inner sep=0pt, minimum size=4pt] at (u10001) {};
\node[circle, fill=white,draw,inner sep=0pt, minimum size=4pt] at (u101) {};
\node[circle, fill=white,draw,inner sep=0pt, minimum size=4pt] at (u1010) {};
\node[circle, fill=white,draw,inner sep=0pt, minimum size=4pt] at (u10100) {};
\end{tikzpicture}
\end{minipage}

\end{center}
\begin{tikzpicture}
\node[circle, fill=blue,inner sep=0pt, minimum size=5pt, label=$v_3$] (v3) {};
\node[circle, fill=red,inner sep=0pt, minimum size=5pt,label=$v_5$,right=1.6cm of v3] (v5) {};
\node[circle, fill=red,inner sep=0pt, minimum size=5pt,label=$v_1$,right=1.6cm of v5] (v1) {};
\node[circle, fill=blue,inner sep=0pt, minimum size=5pt, label=$v_0$,right=1.6cm of v1] (v0) {};
\node[circle, fill=blue,inner sep=0pt, minimum size=5pt,label=$v_4$,right=1.6cm of v0] (v4) {};
\node[circle, fill=red,inner sep=0pt, minimum size=5pt,label=$v_2$,right=1.6cm of v4] (v2) {};
\end{tikzpicture}
\caption{Coding tree of $1$-types for $(\mathbb{Q}_2,<)$ and the linear order represented by its coding nodes.}
\label{fig.Q2tree}
\end{figure}
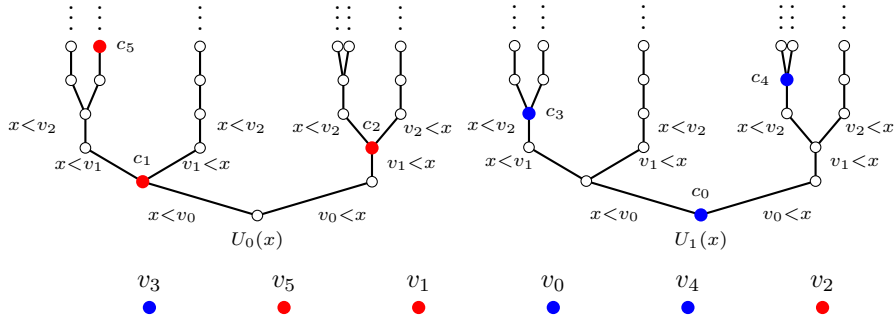


\begin{example}[The coding tree of $1$-types  $\bS(\bQ_{\bQ})$]\label{ex.Q_Q}
Next, we present the tree of $1$-types  for the  \Fraisse\ structure $\bQ_{\bQ}$.
Recall that  this  is the \Fraisse\ limit of the class $\mathcal{CO}$ in the language $\mathcal{L}=\{<,E\}$, where $<$ is a linear order and $E$ is a convexly ordered
equivalence relation, meaning that all equivalence classes are intervals.

Figure \ref{fig.QQtree} shows  the first six levels of a coding tree of $1$-types, $\bS(\bQ_{\bQ})$.
The formulas which are in the $1$-types can be read from the graphic.  For instance,
$c_0$ is the empty type.
$c_1$ is the $1$-type $\{v_0<x, E(x,v_0)\}$,
so since the vertex $v_1$ satisfies this $1$-type, we have
$v_0<v_1$ and $E(v_0,v_1)$ holding.
Similarly,
$c_2$ is the $1$-type
$\{ x<v_0, \neg E(x,v_0), x<v_1, \neg E(x,v_1)\}$,
so $v_2$ satisfies  $v_2<v_0$, and $v_2$ is not equivalent to either of $v_0$ or $v_1$.
 $c_3$ is the $1$-type
$\{v_0<x,E(x,v_0), v_1<x, E(x,v_1), v_2<x,\neg E(x,v_2)\}$, and hence, we see that $v_1<v_3$ and $v_3$ is equivalent to $v_1$ and hence also to $v_0$.
Note that only coding nodes branch.
Moreover,  $c_n$ has splitting  degree two if $c_n$ represents a vertex $v_n$ which is equivalent to $v_i$ for some $i<n$;
 otherwise
$c_n$ has  splitting  degree four.
 For each non-coding node $s$ on the level of a coding node $c_n$, there is only one possible $1$-type extending $s$ over the initial structure on the first $n+1$ vertices of  $\bQ_{\bQ}$.
We will show later in Lemma \ref{lem.Q_Qbiskew} that there is a coding subtree representing $\bQ_{\bQ}$  ensuring that   \EEAP$^+$ holds.
The structures $\bQ_{\bQ_{\bQ}}$, etc.\ have coding trees of $1$-types which behave similarly.

In Figure \ref{fig.QQtree},
below the tree  $\bS(\bQ_{\bQ})$ is the linear order  on the vertices  $v_0,\dots, v_5$ represented  by the coding nodes $c_0,\dots,c_5$;  the lines between the vertices represent that they are in the same equivalence class.  Thus, $v_0,v_1,v_3$ are all in the same equivalence class, $v_2, v_4$ are in a different equivalence class, and $v_5$ is in yet another equivalence class.
\end{example}


\begin{sidewaysfigure}
\vspace{13cm}
\begin{tikzpicture}[grow'=up,scale=.60]
\tikzstyle{level 1}=[sibling distance=3.3in]
\tikzstyle{level 2}=[sibling distance=2in]
\tikzstyle{level 3}=[sibling distance=1in]
\tikzstyle{level 4}=[sibling distance=0.5in]
\tikzstyle{level 5}=[sibling distance=0.2in]
\tikzstyle{level 6}=[sibling distance=0.1in]
\tikzstyle{level 7}=[sibling distance=0.07in]
\node [label=$c_0$] {} coordinate (t9)
child{ coordinate (t0) edge from parent[color=black,thick]
child{ coordinate (t00) edge from parent[color=red,thick]
child{coordinate (t000) edge from parent[color=black,thick]
child{coordinate (t0000) edge from parent[color=purple,thick]
child{coordinate (t00000) edge from parent[color=magenta,thick]
child{coordinate (t000000) edge from parent[color=black,thick]}
}
}
}
child{coordinate (t001) edge from parent[color=purple,thick]
child{coordinate (t0010) edge from parent[color=purple,thick]
child{coordinate (t00100) edge from parent[color=magenta,thick]
child{coordinate (t001000) edge from parent[color=black,thick]}
}
child{coordinate (t00101) edge from parent[color=magenta,thick]
child{coordinate (t001010) edge from parent[color=black,thick]}
}
}
}
child{coordinate (t002) edge from parent[color=olive,thick]
child{coordinate (t0020) edge from parent[color=purple,thick]
child{coordinate (t00200)
edge from parent[color=magenta,thick]
child{coordinate (t002000)
edge from parent[color=black,thick]}
}
}
}
child{coordinate (t003) edge from parent[color=teal,thick]
child{coordinate (t0030) edge from parent[color=purple,thick]
child{coordinate (t00300) edge from parent[color=magenta,thick]
child{coordinate (t003000) edge from parent[color=black,thick]}
child{coordinate (t003001) edge from parent[color=black,thick]}
child{coordinate (t003002) edge from parent[color=black,thick]}
child{coordinate (t003003) edge from parent[color=black,thick]}
}
}
}
}
}
child{ coordinate (t1) edge from parent[color=black,thick]
child{ coordinate (t10) edge from parent[color=cyan,thick]
child{coordinate (t100) edge from parent[color=black,thick]
child{coordinate (t1000) edge from parent[color=purple,thick]
child{coordinate (t10000) edge from parent[color=magenta,thick]
child{coordinate (t100000) edge from parent[color=black,thick]}
}
}
}
}
}
child{ coordinate (t2) edge from parent[color=black,thick]
child{ coordinate (t20) edge from parent[color=purple,thick]
child{coordinate (t200) edge from parent[color=red,thick]
child{coordinate (t2000) edge from parent[color=purple,thick]
child{coordinate (t20000) edge from parent[color=magenta,thick]
child{coordinate (t200000) edge from parent[color=black,thick]}
}
}
}
}
child{coordinate (t21) edge from parent[color=teal,thick]
child{coordinate (t210) edge from parent[color=olive,thick]
child{coordinate (t2100) edge from parent[color=magenta,thick]
child{coordinate (t21000) edge from parent[color=magenta,thick]
child{coordinate (t210000) edge from parent[color=black,thick]}
}
}
child{coordinate (t2101) edge from parent[color=purple,thick]
child{coordinate (t21010) edge from parent[color=magenta,thick]
child{coordinate (t210100) edge from parent[color=black,thick]}
}
}
}
}
}
child{ coordinate (t3) edge from parent[color=black,thick]
child{ coordinate (t30) edge from parent[color=orange,thick]
child{coordinate (t300) edge from parent[color=red,thick]
child{coordinate (t3000) edge from parent[color=blue,thick]
child{coordinate (t30000) edge from parent[color=magenta,thick]
child{coordinate (t300000) edge from parent[color=black,thick]}
}
}
}
}
}
;

\node[circle, fill=blue,inner sep=0pt, minimum size=5pt] at (t9) {};
\node[label=280:$\scriptstyle{\{x<v_0,x \cancel{E} v_0\}}$] at (t0) {};
\node[circle, fill=blue,inner sep=0pt, minimum size=5pt,label=0:$c_1$] at (t2) {};
\node[label=250:$\scriptstyle{\{v_0<x,x \cancel{E} v_0\}}$] at (t3) {};
\node[circle, fill=blue,inner sep=0pt, minimum size=5pt,label=0:$c_2$] at (t00) {};
\node[circle, fill=blue,inner sep=0pt, minimum size=5pt,label=0:$c_3$] at (t210) {};
\node[circle, fill=blue,inner sep=0pt, minimum size=5pt,label=0:$c_4$] at (t0010) {};
\node[circle, fill=blue,inner sep=0pt, minimum size=5pt,label=0:$c_5$] at (t00300) {};
\node[label=0:$\hspace{-1mm}\scriptstyle{\{x<v_0, x E v_0\}}$] at (t1) {};
\node[label=180:$\hspace{1mm}\scriptstyle{\{v_0<x, x E v_0\}}$] at (t2) {};
\node[label=200:${\color{red}\scriptstyle{\{x<v_1, x \cancel{E} v_1\}}}$] at (t00) {};
\node[label=200:${\color{cyan}\scriptstyle{\{x<v_1, x E v_1\}}}$] at (t10) {};
\node[label=250:${\color{purple}\scriptstyle{\{x<v_1, x E v_1\}}}\hspace{-5mm}$] at (t20) {};
\node[label=280:${\color{teal}\scriptstyle{\{v_1<x, x E v_1\}}}$] at (t21) {};
\node[label=280:${\color{orange}\scriptstyle{\{v_1<x, x \cancel{E} v_1\}}}$] at (t30) {};
\node[label=270:$\scriptstyle{\{x<v_2, x \cancel{E} v_2\}}$] at (t000) {};
\node[label=270:${\color{teal}\scriptscriptstyle{\{v_2<x, x \cancel{E} v_2\}}} \hspace{3mm}$] at (t003) {};
\node[label=250:$\scriptscriptstyle{\{v_2<x, x \cancel{E} v_2\}} \hspace{-2mm}$] at (t100) {};
\node[label=250:${\color{red}\scriptstyle{\{v_2<x, x \cancel{E} v_2\}}}$] at (t200) {};
\node[label=280:${\color{olive}\scriptstyle{\{v_2<x, x \cancel{E} v_2\}}}$] at (t210) {};
\node[label=280:${\color{red}\scriptstyle{\{v_2<x, x \cancel{E} v_2\}}}$] at (t300) {};
\node[label=280:${\color{blue}\scriptstyle{\{v_3<x, x \cancel{E} v_3\}}}$] at (t3000) {};
\node[label=280:${\color{purple}\scriptstyle{\{v_3<x, x E v_3\}}}$] at (t2101) {};
\node[label=250:${\color{red}\scriptstyle{\{x<v_3, x E v_3\}}}$] at (t2100) {};
\node[label=280:${\color{magenta}\scriptstyle{\{v_4<x, x \cancel{E} v_4\}}}$] at (t00300) {};

\node[circle,inner sep=0pt, minimum size=5pt,label=90:$\vdots$] at (t000000) {};
\node[circle,inner sep=0pt, minimum size=5pt,label=90:$\vdots$] at (t001000) {};
\node[circle,inner sep=0pt, minimum size=5pt,label=90:$\vdots$] at (t001010) {};
\node[circle,inner sep=0pt, minimum size=5pt,label=90:$\vdots$] at (t002000) {};
\node[circle,inner sep=0pt, minimum size=5pt,label=90:$\vdots$] at (t003000) {};
\node[circle,inner sep=0pt, minimum size=5pt,label=90:$\vdots$] at (t003001) {};
\node[circle,inner sep=0pt, minimum size=5pt,label=90:$\vdots$] at (t003002) {};
\node[circle,inner sep=0pt, minimum size=5pt,label=90:$\vdots$] at (t003003) {};
\node[circle,inner sep=0pt, minimum size=5pt,label=90:$\vdots$] at (t100000) {};
\node[circle,inner sep=0pt, minimum size=5pt,label=90:$\vdots$] at (t200000) {};
\node[circle,inner sep=0pt, minimum size=5pt,label=90:$\vdots$] at (t210000) {};
\node[circle,inner sep=0pt, minimum size=5pt,label=90:$\vdots$] at (t210100) {};
\node[circle,inner sep=0pt, minimum size=5pt,label=90:$\vdots$] at (t300000) {};

\node[circle, fill=blue,inner sep=0pt, minimum size=5pt, label=$v_4$,below=3.8cm of t001] (v4) {};
\node[circle, fill=blue,inner sep=0pt, minimum size=5pt, label=$v_2$,below=2.9cm of t00] (v2) {};
\node[circle, fill=blue,inner sep=0pt, minimum size=5pt, label=$v_5$,below=3.799cm of t003] (v5) {};
\node[circle, fill=blue,inner sep=0pt, minimum size=5pt, label=$v_0$,below=1.0cm of t9] (v0) {};
\node[circle, fill=blue,inner sep=0pt, minimum size=5pt, label=$v_1$,below=1.92cm of t2] (v1) {};
\node[circle, fill=blue,inner sep=0pt, minimum size=5pt, label=$v_3$,below=2.83cm of t21] (v3) {};

\node[circle, fill=white,draw,inner sep=0pt, minimum size=4pt] at (t0) {};
\node[circle, fill=white,draw,inner sep=0pt, minimum size=4pt] at (t000) {};
\node[circle, fill=white,draw,inner sep=0pt, minimum size=4pt] at (t0000) {};
\node[circle, fill=white,draw,inner sep=0pt, minimum size=4pt] at (t00000) {};
\node[circle, fill=white,draw,inner sep=0pt, minimum size=4pt] at (t000000) {};
\node[circle, fill=white,draw,inner sep=0pt, minimum size=4pt] at (t001) {};
\node[circle, fill=white,draw,inner sep=0pt, minimum size=4pt] at (t00100) {};
\node[circle, fill=white,draw,inner sep=0pt, minimum size=4pt] at (t001000) {};
\node[circle, fill=white,draw,inner sep=0pt, minimum size=4pt] at (t00101) {};
\node[circle, fill=white,draw,inner sep=0pt, minimum size=4pt] at (t001010) {};
\node[circle, fill=white,draw,inner sep=0pt, minimum size=4pt] at (t002) {};
\node[circle, fill=white,draw,inner sep=0pt, minimum size=4pt] at (t0020) {};
\node[circle, fill=white,draw,inner sep=0pt, minimum size=4pt] at (t00200) {};
\node[circle, fill=white,draw,inner sep=0pt, minimum size=4pt] at (t002000) {};
\node[circle, fill=white,draw,inner sep=0pt, minimum size=4pt] at (t003) {};
\node[circle, fill=white,draw,inner sep=0pt, minimum size=4pt] at (t0030) {};
\node[circle, fill=white,draw,inner sep=0pt, minimum size=4pt] at (t003000) {};
\node[circle, fill=white,draw,inner sep=0pt, minimum size=4pt] at (t003001) {};
\node[circle, fill=white,draw,inner sep=0pt, minimum size=4pt] at (t003002) {};
\node[circle, fill=white,draw,inner sep=0pt, minimum size=4pt] at (t003003) {};
\node[circle, fill=white,draw,inner sep=0pt, minimum size=4pt] at (t1) {};
\node[circle, fill=white,draw,inner sep=0pt, minimum size=4pt] at (t10) {};
\node[circle, fill=white,draw,inner sep=0pt, minimum size=4pt] at (t100) {};
\node[circle, fill=white,draw,inner sep=0pt, minimum size=4pt] at (t1000) {};
\node[circle, fill=white,draw,inner sep=0pt, minimum size=4pt] at (t10000) {};
\node[circle, fill=white,draw,inner sep=0pt, minimum size=4pt] at (t100000) {};
\node[circle, fill=white,draw,inner sep=0pt, minimum size=4pt] at (t20) {};
\node[circle, fill=white,draw,inner sep=0pt, minimum size=4pt] at (t200) {};
\node[circle, fill=white,draw,inner sep=0pt, minimum size=4pt] at (t2000) {};
\node[circle, fill=white,draw,inner sep=0pt, minimum size=4pt] at (t20000) {};
\node[circle, fill=white,draw,inner sep=0pt, minimum size=4pt] at (t200000) {};
\node[circle, fill=white,draw,inner sep=0pt, minimum size=4pt] at (t21) {};
\node[circle, fill=white,draw,inner sep=0pt, minimum size=4pt] at (t2100) {};
\node[circle, fill=white,draw,inner sep=0pt, minimum size=4pt] at (t21000) {};
\node[circle, fill=white,draw,inner sep=0pt, minimum size=4pt] at (t210000) {};
\node[circle, fill=white,draw,inner sep=0pt, minimum size=4pt] at (t2101) {};
\node[circle, fill=white,draw,inner sep=0pt, minimum size=4pt] at (t21010) {};
\node[circle, fill=white,draw,inner sep=0pt, minimum size=4pt] at (t210100) {};
\node[circle, fill=white,draw,inner sep=0pt, minimum size=4pt] at (t3) {};
\node[circle, fill=white,draw,inner sep=0pt, minimum size=4pt] at (t30) {};
\node[circle, fill=white,draw,inner sep=0pt, minimum size=4pt] at (t300) {};
\node[circle, fill=white,draw,inner sep=0pt, minimum size=4pt] at (t3000) {};
\node[circle, fill=white,draw,inner sep=0pt, minimum size=4pt] at (t30000) {};
\node[circle, fill=white,draw,inner sep=0pt, minimum size=4pt] at (t300000) {};

\draw[magenta,thick] (v4)--(v2);
\draw[magenta,thick] (v0)--(v1);
\draw[magenta,thick] (v1)--(v3);
\end{tikzpicture}
\caption{Coding tree of $1$-types for $\bQ_{\bQ}$ and linear order with convex equivalence relations represented by its coding nodes.}
\label{fig.QQtree}
\end{sidewaysfigure}
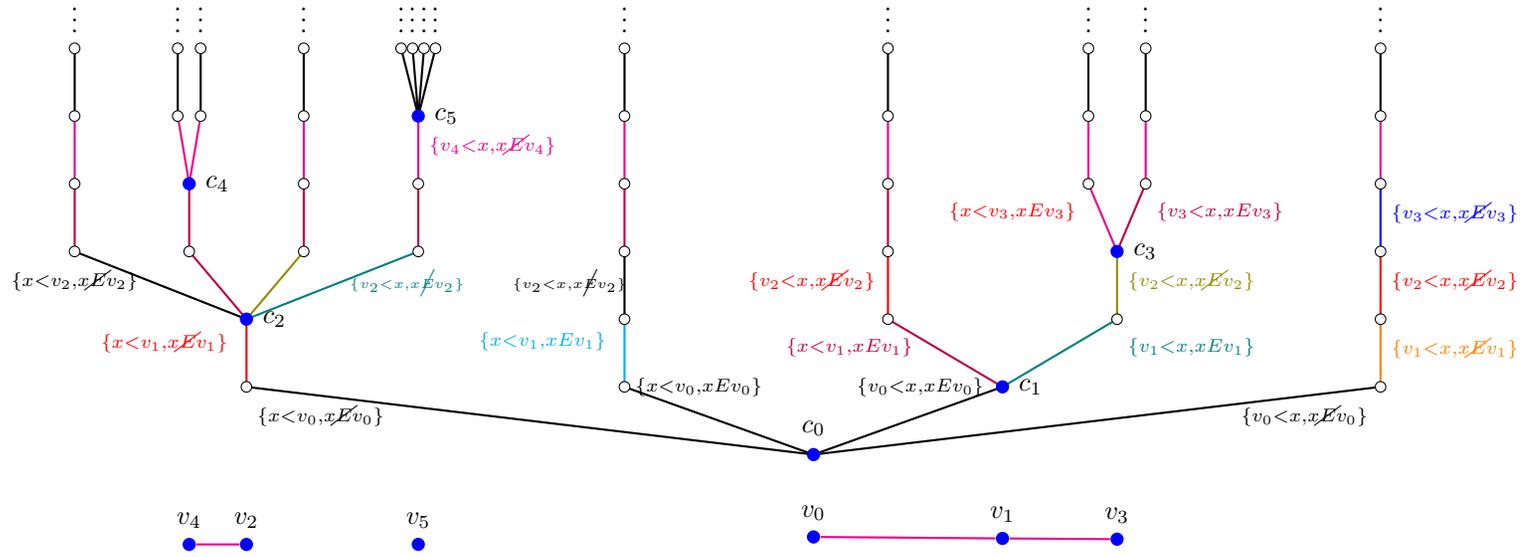


Next, we present graphics for coding trees of $1$-types for some free amalgamation classes.
The tree of $1$-types for the Rado graph is simply a binary tree in which the coding nodes are dense and every node $s$  at the level of the $n$-th coding node splits into two immediate successors, representing the two possible extensions of $s$ to
the $1$-types
$s\cup\{E(x,v_n)\}$ and $s\cup\{\neg E(x,v_n)\}$.
This follows immediately from the Extension Property for the Rado graph.
As this is simple to visualize, and as a graphic has already appeared in \cite{DobrinenIfCoLog20}, we move on to bipartite graphs.

\begin{example}[The coding  tree of $1$-types
for the generic bipartite graph]\label{ex.BPG}
Figure \ref{fig.genbiptree} presents a  coding tree of $1$-types for the  generic
bipartite graph.
The unary relations $U_0$ and $U_1$, which keep track of which partition each vertex is in, are represented by ``red'' and ``blue'', respectively.
We have chosen to enumerate  this  structure so that
 odd indexed vertices are in one of the partitions, and even indexed vertices  are in the other,
for purely aesthetic reasons.
The edge relation is represented as extension to the right, and non-edge is represented by extending left.
On the left is the bipartite graph being represented by the coding nodes in the two-rooted tree of $1$-types.
For instance, $c_0$ is the $1$-type $\{U_0(x)\}$, so $v_0$ is a vertex in the  collection of ``red''  vertices.
For another example,
$c_3$ is the $1$-type $\{U_1(x), E(x,v_0), \neg E(x,v_1),E(x,v_2)\}$.
Thus, $v_3$ is in the collection of ``blue''  vertices and has edges exactly with $v_0$ and $v_2$.
By Proposition \ref{prop.FA},
for each $n\ge 2$,
the \Fraisse\ class of finite $n$-partite graphs satisfies \SFAP.
\end{example}


\begin{figure}
\begin{center}
\begin{minipage}{.2\textwidth}
\hspace*{-3.5cm}
\begin{tikzpicture}[grow'=up,scale=.4]
\tikzstyle{level 1}=[sibling distance=3in]
\tikzstyle{level 2}=[sibling distance=1.8in]
\tikzstyle{level 3}=[sibling distance=1in]
\tikzstyle{level 4}=[sibling distance=0.5in]
\tikzstyle{level 5}=[sibling distance=0.2in]
\tikzstyle{level 6}=[sibling distance=0.1in]
\tikzstyle{level 7}=[sibling distance=0.07in]

\node {} coordinate (t9)
child{ coordinate (t0) edge from parent[color=black,thick]
child{coordinate (t00) edge from parent[color=black,thick]
child{coordinate (t000) edge from parent[color=black,thick]
child{coordinate (t0000) edge from parent[color=black,thick]
child{coordinate (t00000) edge from parent[color=black,thick]}
child{coordinate (t00001) edge from parent[color=white,thick]}
}
child{coordinate (t0001) edge from parent[color=black,thick]
child{coordinate (t00010) edge from parent[color=black,thick]}
child{coordinate (t00011) edge from parent[color=white,thick]}
}
}
child{coordinate (t001) edge from parent[color=white,thick]}
}
child{coordinate (t01) edge from parent[color=black,thick]
child{coordinate (t010) edge from parent[color=black,thick]
child{coordinate (t0100) edge from parent[color=black,thick]
child{coordinate (t01000) edge from parent[color=black,thick]}
child{coordinate (t01001) edge from parent[color=white,thick]}
}
child{coordinate (t0101) edge from parent[color=black,thick]
child{coordinate (t01010) edge from parent[color=black,thick]}
child{coordinate (t01011) edge from parent[color=white,thick]}
}
}
child{coordinate (t011) edge from parent[color=white,thick]}
}
}
child{coordinate (t1) edge from parent[color=white,thick]}
;

\node[label=180:$\scriptstyle{\cancel{E}(x,v_0)}\hspace{0.5cm}$,label=270:$\scriptstyle{U_0(x)}$,label=0:$c_0$,circle, fill=red,inner sep=0pt, minimum size=5pt] at (t9) {};
\node[label=0:$c_2$,circle, fill=red,inner sep=0pt, minimum size=5pt] at (t01) {};
\node[label=0:$c_4$,circle, fill=red,inner sep=0pt, minimum size=5pt] at (t0001) {};

\node[circle,inner sep=0pt, minimum size=5pt,label=90:$\vdots$] at (t00000) {};
\node[circle,inner sep=0pt, minimum size=5pt,label=90:$\vdots$] at (t00010) {};
\node[circle,inner sep=0pt, minimum size=5pt,label=90:$\vdots$] at (t01000) {};
\node[circle,inner sep=0pt, minimum size=5pt,label=90:$\vdots$] at (t01010) {};

\node[label=20:$\hspace{0.5cm}\scriptstyle{E(x,v_1)}$,label=160:$\scriptstyle{\cancel{E}(x,v_1)}\hspace{0.5cm}$,circle, fill=white,draw,inner sep=0pt, minimum size=4pt] at (t0) {};
\node[circle, fill=white,draw,inner sep=0pt, minimum size=4pt] at (t00) {};
\node[circle, fill=white,draw,inner sep=0pt, minimum size=4pt] at (t000) {};
\node[circle, fill=white,draw,inner sep=0pt, minimum size=4pt] at (t0000) {};
\node[circle, fill=white,draw,inner sep=0pt, minimum size=4pt] at (t00000) {};
\node[circle, fill=white,draw,inner sep=0pt, minimum size=4pt] at (t00010) {};
\node[circle, fill=white,draw,inner sep=0pt, minimum size=4pt] at (t010) {};
\node[circle, fill=white,draw,inner sep=0pt, minimum size=4pt] at (t0100) {};
\node[circle, fill=white,draw,inner sep=0pt, minimum size=4pt] at (t01000) {};
\node[circle, fill=white,draw,inner sep=0pt, minimum size=4pt] at (t0101) {};
\node[circle, fill=white,draw,inner sep=0pt, minimum size=4pt] at (t01010) {};

\node[label=0:$v_0$,circle, fill=red,inner sep=0pt, minimum size=5pt,left=5.8cm of t9] (v0) {};
\node[label=0:$v_1$,circle, fill=blue,inner sep=0pt, minimum size=5pt,left=2.8cm of t0] (v1) {};
\node[label=180:$v_2$,circle, fill=red,inner sep=0pt, minimum size=5pt,above=0.95cm of v0] (v2) {};
\node[label=0:$v_3$,circle, fill=blue,inner sep=0pt, minimum size=5pt,above=0.95cm of v1] (v3) {};
\node[label=180:$v_4$,circle, fill=red,inner sep=0pt, minimum size=5pt,above=1.0cm of v2] (v4) {};
\node[label=0:$v_5$,circle, fill=blue,inner sep=0pt, minimum size=5pt,above=1.05cm of v3] (v5) {};

\draw[black,thick] (v0) to (v1);
\draw[black,thick] (v0) to (v3);
\draw[black,thick] (v1) to (v2);
\draw[black,thick] (v2) to (v3);
\draw[black,thick] (v2) to (v5);
\draw[black,thick] (v3) to (v4);
\end{tikzpicture}
\end{minipage}
\begin{minipage}{.2\textwidth}
\hspace*{0.5cm}
\vspace*{0.1cm}
\begin{tikzpicture}[grow'=up,scale=.4]
\tikzstyle{level 1}=[sibling distance=3in]
\tikzstyle{level 2}=[sibling distance=1.8in]
\tikzstyle{level 3}=[sibling distance=1in]
\tikzstyle{level 4}=[sibling distance=0.5in]
\tikzstyle{level 5}=[sibling distance=0.2in]
\tikzstyle{level 6}=[sibling distance=0.1in]
\tikzstyle{level 7}=[sibling distance=0.07in]

\node {} coordinate (u9)
child{ coordinate (u0) edge from parent[color=black,thick]
child{coordinate (u00) edge from parent[color=black,thick]
child{coordinate (u000) edge from parent[color=black,thick]
child{coordinate (u0000) edge from parent[color=black,thick]
child{coordinate (u00000) edge from parent[color=black,thick]}
child{coordinate (u00001) edge from parent[color=black,thick]}
}
child{coordinate (u0001) edge from parent[color=white,thick]}
}
child{coordinate (u001) edge from parent[color=black,thick]
child{coordinate (u0010) edge from parent[color=black,thick]
child{coordinate (u00100) edge from parent[color=black,thick]}
child{coordinate (u00101) edge from parent[color=black,thick]}
}
child{coordinate (u0011) edge from parent[color=white,thick]}
}
}
child{coordinate (u01) edge from parent[color=white,thick]}
}
child{coordinate (u1) edge from parent[color=black,thick]
child{coordinate (u10) edge from parent[color=black,thick]
child{coordinate (u100) edge from parent[color=black,thick]
child{coordinate (u1000) edge from parent[color=black,thick]
child{coordinate (u10000) edge from parent[color=black,thick]}
child{coordinate (u10001) edge from parent[color=black,thick]}
}
child{coordinate (u1001) edge from parent[color=white,thick]}
}
child{coordinate (u101) edge from parent[color=black,thick]
child{coordinate (u1010) edge from parent[color=black,thick]
child{coordinate (u10100) edge from parent[color=black,thick]}
child{coordinate (u10101) edge from parent[color=black,thick]}
}
child{coordinate (u1011) edge from parent[color=white,thick]}
}
}
child{coordinate (u11) edge from parent[color=white,thick]}
}
;

\node[label=0:$c_1$,circle, fill=blue,inner sep=0pt, minimum size=5pt] at (u1) {};
\node[label=0:$c_3$,circle, fill=blue,inner sep=0pt, minimum size=5pt] at (u101) {};
\node[label=180:$c_5$,circle, fill=blue,inner sep=0pt, minimum size=5pt] at (u00100) {};

\node[circle,inner sep=0pt, minimum size=5pt,label=90:$\vdots$] at (u00000) {};
\node[circle,inner sep=0pt, minimum size=5pt,label=90:$\vdots$] at (u00001) {};
\node[circle,inner sep=0pt, minimum size=5pt,label=90:$\vdots$] at (u00100) {};
\node[circle,inner sep=0pt, minimum size=5pt,label=90:$\vdots$] at (u00101) {};
\node[circle,inner sep=0pt, minimum size=5pt,label=90:$\vdots$] at (u10000) {};
\node[circle,inner sep=0pt, minimum size=5pt,label=90:$\vdots$] at (u10001) {};
\node[circle,inner sep=0pt, minimum size=5pt,label=90:$\vdots$] at (u10100) {};
\node[circle,inner sep=0pt, minimum size=5pt,label=90:$\vdots$] at (u10101) {};

\node[label=270:$\scriptstyle{U_1(x)}$,label=20:$\hspace{1.3cm}\scriptstyle{E(x,v_0)}$,label=160:$\scriptstyle{\cancel{E}(x,v_0)}\hspace{1.3cm}$,circle, fill=white,draw,inner sep=0pt, minimum size=4pt] at (u9) {};
\node[label=160:$\scriptstyle{\cancel{E}(x,v_1)}\hspace{0.5cm}$,circle, fill=white,draw,inner sep=0pt, minimum size=4pt] at (u0) {};
\node[circle, fill=white,draw,inner sep=0pt, minimum size=4pt] at (u00) {};
\node[circle, fill=white,draw,inner sep=0pt, minimum size=4pt] at (u000) {};
\node[circle, fill=white,draw,inner sep=0pt, minimum size=4pt] at (u001) {};
\node[circle, fill=white,draw,inner sep=0pt, minimum size=4pt] at (u0000) {};
\node[circle, fill=white,draw,inner sep=0pt, minimum size=4pt] at (u00000) {};
\node[circle, fill=white,draw,inner sep=0pt, minimum size=4pt] at (u0010) {};
\node[circle, fill=white,draw,inner sep=0pt, minimum size=4pt] at (u00001) {};
\node[circle, fill=white,draw,inner sep=0pt, minimum size=4pt] at (u00101) {};
\node[label=-20:$\hspace{5mm}\scriptstyle{\cancel{E}(x,v_1)}$,circle, fill=white,draw,inner sep=0pt, minimum size=4pt] at (u10) {};
\node[circle, fill=white,draw,inner sep=0pt, minimum size=4pt] at (u100) {};
\node[circle, fill=white,draw,inner sep=0pt, minimum size=4pt] at (u1000) {};
\node[circle, fill=white,draw,inner sep=0pt, minimum size=4pt] at (u10000) {};
\node[circle, fill=white,draw,inner sep=0pt, minimum size=4pt] at (u10001) {};
\node[circle, fill=white,draw,inner sep=0pt, minimum size=4pt] at (u1010) {};
\node[circle, fill=white,draw,inner sep=0pt, minimum size=4pt] at (u10100) {};
\node[circle, fill=white,draw,inner sep=0pt, minimum size=4pt] at (u10101) {};
\end{tikzpicture}
\end{minipage}
\end{center}
\caption{Coding tree of $1$-types for the generic bipartite graph}\label{fig.genbiptree}
\end{figure}


Lastly, we consider free amalgamation classes with relations of higher arity.
The prototypical example of this is the generic $3$-uniform hypergraph, and discussing it should provide the reader with reasonable intuition about coding trees for higher arities.

\begin{example}[The coding tree of $1$-types for the generic $3$-uniform hypergraph]\label{ex.3reghypergraph}
Figure \ref{fig.3hyptree} presents
the coding tree of $1$-types for the generic $3$-uniform hypergraph.
This tree has the property that every node at the same level branches into the same number of immediate successors, as there are no forbidden substructures.
On the left of Figure \ref{fig.3hyptree} is a picture of the hypergraph being built, where $v_n$ is the vertex satisfying the $1$-type of the coding node $c_n$ over the initial segment of the structure restricted to  $\{v_i:i<n\}$.

Since hyperedges involve three vertices, $c_0$ and $c_1$ are both the empty $1$-types.
Technically, these nodes are the same, but we draw them distinctly in Figure 4.\ to aid the drawing of the hypergraph on the left.
Letting $R$ denote the $3$-hyperedge relation,
$c_1$ branches into  two  $1$-types  over $\{v_0,v_1\}$:
$\{\neg R(x,v_0,v_1)\}$ and $\{R(x,v_0,v_1)\}$.
Since $c_2=\{R(x,v_0,v_1)\}$, it follows that $R(v_0,v_1,v_2)$ holds in the hypergraph represented on the left of the tree; this hyperedge is represented by the oval containing these three vertices.

Both nodes on the level of $c_2$ branch into four immediate successors.
This is because
for each node $s$ at the level of $c_2$, the immediate successors of $s$ range over the possibilities of adding a new formula $R(x,\cdot,\cdot)$ or $\neg R(x,\cdot,\cdot)$ containing  the parameter $v_2$
and a choice of either $v_0$ or $v_1$ as the second parameter.
In particular,  the immediate successors of $c_2$ are the $1$-types consisting of  $\{R(x,v_0,v_1)\}$  unioned with  one of the following:
\begin{enumerate}
\item
$\{\neg R(x,v_0,v_2),\neg R(x,v_1,v_2)\}$;
\item
$\{\neg R(x,v_0,v_2),R(x,v_1,v_2)\}$;
\item
$\{ R(x,v_0,v_2),\neg R(x,v_1,v_2)\}$;
\item
$\{ R(x,v_0,v_2),R(x,v_1,v_2)\}$.
\end{enumerate}
Likewise, the immediate successors of the other node $s=\{\neg R(x,v_0,v_1)\}$ in level two of the tree consists of the extensions of $s$ by one of the four above cases.
In general, each node on the level of $c_n$ branches into $2^n$ many immediate successors.
This is because the new formulas in any immediate successor have the choice of $R(x,p,v_n)$ or its negation, where $p\in\{v_i:i<n\}$.
However,
the \Fraisse\ class of  finite $3$-uniform hypergraphs  satisfies \SFAP\ (by
Proposition \ref{prop.FA}), and
  Theorem \ref{thm.SFAPimpliesDCT} will  provide  a  skew subtree coding the generic $3$-hypergraph in which the branching degree is two (that is, a diagonal subtree).

The coding node $c_3$ is the $1$-type $\{\neg R(x,v_0,v_1),  R(x,v_0,v_2),\neg R(x,v_1,v_2)\}$.
Thus, the hypergraph being built on the left has the hyperedge $R(v_0,v_2,v_3)$.
The coding node
$c_4$ is the $1$-type  consisting of $R(x,v_0,v_1),
R(x,v_1,v_2),
R(x,v_2,v_3)$ along with
$\neg R(x,p_0,p_1)$ where $p_0,p_1$ are parameters in $\{v_0,\dots,v_3\}$.
This codes the new hyperedges
$R(v_0,v_1,v_4)$,  $R(v_1,v_2,v_4)$ and $R(v_2,v_3,v_4)$.
\end{example}


\begin{figure}
\hspace{2cm}
\begin{tikzpicture}[grow'=up,scale=.8]
\tikzstyle{level 1}=[sibling distance=3.3in]
\tikzstyle{level 2}=[sibling distance=2in]
\tikzstyle{level 3}=[sibling distance=.4in]
\tikzstyle{level 4}=[sibling distance=0.2in]
\tikzstyle{level 5}=[sibling distance=0.125in]
\node [label=0:$c_0$] {} coordinate (t9)
child{ coordinate (t0) edge from parent[color=black,thick]
child{ coordinate (t00) edge from parent[color=black,thick]
child{ coordinate (t000) edge from parent[color=magenta,thick]
child{coordinate (t0000) edge from parent[color=black,thick]}
child{coordinate (t0001) edge from parent[color=black,thick]}
child{coordinate (t0002) edge from parent[color=black,thick]}
child{coordinate (t0003) edge from parent[color=black,thick]}
child{coordinate (t0004) edge from parent[color=black,thick]}
child{coordinate (t0005) edge from parent[color=black,thick]}
child{coordinate (t0006) edge from parent[color=black,thick]}
child{coordinate (t0007) edge from parent[color=black,thick]}
}
child{coordinate (t001) edge from parent[color=purple,thick]}
child{coordinate (t002) edge from parent[color=cyan,thick]}
child{coordinate (t003) edge from parent[color=orange,thick]}
}
child{ coordinate (t01) edge from parent[color=black,thick]
child{ coordinate (t010) edge from parent[color=magenta,thick]}
child{ coordinate (t011) edge from parent[color=purple,thick]
child{ coordinate (t0110) edge from parent[color=black,thick]}
child{ coordinate (t0111) edge from parent[color=black,thick]}
child{ coordinate (t0112) edge from parent[color=black,thick]}
child{ coordinate (t0113) edge from parent[color=black,thick]}
child{ coordinate (t0114) edge from parent[color=black,thick]}
child{ coordinate (t0115) edge from parent[color=black,thick]}
child{ coordinate (t0116) edge from parent[color=black,thick]}
child{ coordinate (t0117) edge from parent[color=black,thick]}
}
child{ coordinate (t012) edge from parent[color=cyan,thick]}
child{ coordinate (t013) edge from parent[color=orange,thick]}
}
}
;

\node[circle, fill=blue,inner sep=0pt, minimum size=5pt] at (t9) {};
\node[circle, fill=blue,inner sep=0pt, minimum size=5pt,label=90:$c_1$,label=280:$\scriptstyle{\emptyset}$] at (t0) {};
\node[circle, fill=blue,inner sep=0pt, minimum size=5pt,label=0:$c_2$] at (t01) {};
\node[circle, fill=blue,inner sep=0pt, minimum size=5pt,label=0:$c_3$] at (t002) {};
\node[circle, fill=blue,inner sep=0pt, minimum size=5pt,label=60:$c_4$] at (t0111) {};
\node[label=0:$\scriptstyle{\{\neg R(x,v_0,v_1)\}}$] at (t00) {};
\node[label=180:$\scriptstyle{\{R(x,v_0,v_1)\}}$] at (t01) {};
\node[label=280:${\color{orange}\scriptstyle{R(x,v_1,v_2)\}}}$,label=0:${\color{orange}\scriptstyle{\{R(x,v_0,v_2),}}$] at (t013) {};

\node[circle, fill=white,draw,inner sep=0pt, minimum size=4pt] at (t00) {};
\node[circle, fill=white,draw,inner sep=0pt, minimum size=4pt] at (t000) {};
\node[circle, fill=white,draw,inner sep=0pt, minimum size=4pt] at (t001) {};
\node[circle, fill=white,draw,inner sep=0pt, minimum size=4pt] at (t003) {};
\node[circle, fill=white,draw,inner sep=0pt, minimum size=4pt] at (t0000) {};
\node[circle, fill=white,draw,inner sep=0pt, minimum size=4pt] at (t0001) {};
\node[circle, fill=white,draw,inner sep=0pt, minimum size=4pt] at (t0002) {};
\node[circle, fill=white,draw,inner sep=0pt, minimum size=4pt] at (t0003) {};
\node[circle, fill=white,draw,inner sep=0pt, minimum size=4pt] at (t0004) {};
\node[circle, fill=white,draw,inner sep=0pt, minimum size=4pt] at (t0005) {};
\node[circle, fill=white,draw,inner sep=0pt, minimum size=4pt] at (t0006) {};
\node[circle, fill=white,draw,inner sep=0pt, minimum size=4pt] at (t0007) {};
\node[circle, fill=white,draw,inner sep=0pt, minimum size=4pt] at (t010) {};
\node[circle, fill=white,draw,inner sep=0pt, minimum size=4pt] at (t011) {};
\node[circle, fill=white,draw,inner sep=0pt, minimum size=4pt] at (t012) {};
\node[circle, fill=white,draw,inner sep=0pt, minimum size=4pt] at (t013) {};
\node[circle, fill=white,draw,inner sep=0pt, minimum size=4pt] at (t0110) {};
\node[circle, fill=white,draw,inner sep=0pt, minimum size=4pt] at (t0112) {};
\node[circle, fill=white,draw,inner sep=0pt, minimum size=4pt] at (t0113) {};
\node[circle, fill=white,draw,inner sep=0pt, minimum size=4pt] at (t0114) {};
\node[circle, fill=white,draw,inner sep=0pt, minimum size=4pt] at (t0115) {};
\node[circle, fill=white,draw,inner sep=0pt, minimum size=4pt] at (t0116) {};
\node[circle, fill=white,draw,inner sep=0pt, minimum size=4pt] at (t0117) {};

\node[circle, fill=blue,inner sep=0pt, minimum size=5pt, label=$v_0$,left=7.2cm of t9] (v0) {};
\node[circle, fill=blue,inner sep=0pt, minimum size=5pt, label=$v_1$,left=7.2cm of t0] (v1) {};
\node[circle, fill=blue,inner sep=0pt, minimum size=5pt, label=$v_2$,left=9.25cm of t01] (v2) {};
\node[circle, fill=blue,inner sep=0pt, minimum size=5pt, label=$v_3$,left=3.66cm of t000] (v3) {};
\node[circle, fill=blue,inner sep=0pt, minimum size=5pt, label=$v_4$,left=7.98cm of t0112] (v4) {};

\node[circle,inner sep=0pt, minimum size=5pt,label=90:$\vdots$] at (t0000) {};
\node[circle,inner sep=0pt, minimum size=5pt,label=90:$\vdots$] at (t0001) {};
\node[circle,inner sep=0pt, minimum size=5pt,label=90:$\vdots$] at (t0002) {};
\node[circle,inner sep=0pt, minimum size=5pt,label=90:$\vdots$] at (t0003) {};
\node[circle,inner sep=0pt, minimum size=5pt,label=90:$\vdots$] at (t0004) {};
\node[circle,inner sep=0pt, minimum size=5pt,label=90:$\vdots$] at (t0005) {};
\node[circle,inner sep=0pt, minimum size=5pt,label=90:$\vdots$] at (t0006) {};
\node[circle,inner sep=0pt, minimum size=5pt,label=90:$\vdots$] at (t0007) {};
\node[circle,inner sep=0pt, minimum size=5pt,label=90:$\vdots$] at (t001) {};
\node[circle,inner sep=0pt, minimum size=5pt,label=90:$\vdots$] at (t002) {};
\node[circle,inner sep=0pt, minimum size=5pt,label=90:$\vdots$] at (t003) {};
\node[circle,inner sep=0pt, minimum size=5pt,label=90:$\vdots$] at (t010) {};
\node[circle,inner sep=0pt, minimum size=5pt,label=90:$\vdots$] at (t012) {};
\node[circle,inner sep=0pt, minimum size=5pt,label=90:$\vdots$] at (t013) {};
\node[circle,inner sep=0pt, minimum size=5pt,label=90:$\vdots$] at (t0110) {};
\node[circle,inner sep=0pt, minimum size=5pt,label=90:$\vdots$] at (t0111) {};
\node[circle,inner sep=0pt, minimum size=5pt,label=90:$\vdots$] at (t0112) {};
{};
\node[circle,inner sep=0pt, minimum size=5pt,label=90:$\vdots$] at (t0113) {};
{};
\node[circle,inner sep=0pt, minimum size=5pt,label=90:$\vdots$] at (t0114) {};
{};
\node[circle,inner sep=0pt, minimum size=5pt,label=90:$\vdots$] at (t0115) {};
{};
\node[circle,inner sep=0pt, minimum size=5pt,label=90:$\vdots$] at (t0116) {};
{};
\node[circle,inner sep=0pt, minimum size=5pt,label=90:$\vdots$] at (t0117) {};

\node[circle,fill=black,inner sep=0pt, minimum size=1pt,below=.1cm of v0] (u0) {};
\node[circle,fill=purple,inner sep=0pt, minimum size=1pt,below=.1cm of v1] (u1) {};
\node[circle,fill=magenta,inner sep=0pt, minimum size=1pt,above=.1cm of v2] (a2) {};
\node[circle,fill=black,inner sep=0pt, minimum size=1pt,above=.1cm of v1] (a1) {};
\node[circle,fill=blue,inner sep=0pt, minimum size=1pt,below=.1cm of v2] (u2) {};
\node[circle,fill=purple,inner sep=0pt, minimum size=1pt,right=.1cm of v2] (r2) {};
\node[circle,fill=cyan,inner sep=0pt, minimum size=1pt,above=.1cm of v3] (a3) {};
\node[circle,fill=purple,inner sep=0pt, minimum size=1pt,left=.1cm of v3] (l3) {};
\node[circle,fill=white,inner sep=0pt, minimum size=1pt,left=.1cm of v1] (l1) {};
\node[circle,fill=white,inner sep=0pt, minimum size=1pt,right=.1cm of v0] (r0) {};
\node[circle,fill=cyan,inner sep=0pt, minimum size=1pt,above=.1cm of r0] (ar0) {};
\node[circle,fill=black,inner sep=0pt, minimum size=1pt,above=.1cm of v4] (a4) {};
\node[circle,fill=magenta,inner sep=0pt, minimum size=1pt,left=.1cm of v4] (l4) {};
\node[circle,fill=black,inner sep=0pt, minimum size=1pt,below=.1cm of l4] (ul4) {};

\draw[magenta,thick] (u0) to [out=180,in=180] (a2);
\draw[magenta,thick] (u0) to [in=0,out=0] (a2);

\draw[cyan,thick] (u0) to [in=180,out=180] (a3);
\draw[cyan,thick] (a3) to [in=180,out=0] (ar0);
\draw[cyan,thick] (ar0) to [in=-30,out=-30] (u0);

\draw[black,thick] (u0) to [in=180,out=180] (a1);
\draw[black,thick] (a4) to [in=0,out=0] (u0);
\draw[black,thick] (a1) to [in=300,out=0] (ul4);
\draw[black,thick] (ul4) to [in=180,out=100] (a4);

\draw[blue,thick] (u2) to [out=180,in=180] (a4);
\draw[blue,thick] (a4) to [in=0,out=0] (u2);

\draw[purple,thick] (a4) to [in=180,out=180] (u1);
\draw[purple,thick] (u1) to [in=100,out=20] (r2);
\draw[purple,thick] (r2) to [in=260,out=90] (l3);
\draw[purple,thick] (l3) to [in=0,out=90] (a4);
\end{tikzpicture}
\caption{Coding tree of $1$-types for the generic $3$-uniform hypergraph.}\label{fig.3hyptree}
\end{figure}


\subsection{Passing types and  similarity}\label{subsec.3.2}

As before, let
$\bK$ be an enumerated \Fraisse\ structure and $\bS:=\bS(\bK)$ be the corresponding coding  tree of $1$-types.
We begin by defining the notion of a subtree of $\bS$.
As is standard in Ramsey theory on infinite trees (see Chapter 6 of \cite{TodorcevicBK10}),
a subtree is not necessarily closed under initial segments,
but rather it is closed under
those portions of initial segments that have certain prescribed lengths.

\begin{defn}[Subtree]\label{defn.subtree}
Let $T$ be a subset of $\bS$, and let $L$ be the set of lengths of
coding nodes in $T$
and lengths of meets
of two incomparable nodes (not necessarily coding nodes) in $T$. Then $T$
is a {\em subtree} of $\bS$
if $T$ is closed under meets and
closed under initial segments with lengths in
$L$, by which we mean
that  whenever   $\ell\in L$ and $t\in T$ with $\ell\le |t|$, then $t\re \ell$ is also a member of $T$.
\end{defn}

We
now describe the
 natural correspondence from subtrees of $\bS$ to  substructures of $\bK$.
 The following notation will aid in the translation.

 \begin{notation}\label{notn.KreA}
Given a subtree $A\sse\bS$, let
 $\lgl c^A_n:n< N\rgl$ denote the  enumeration of
 the coding nodes of $A$  in order of increasing length, where $N\le \om$ is the number of coding nodes in $A$.
 Let
 \begin{equation}
 \mathrm{N}^A:=\{i\in \om:   \exists m\, (c_i=c^A_m)\},
 \end{equation}
 the set of indices $i$ such that $c_i$ is a coding node in $A$.
 For $n<N$,
 let
\begin{equation}
 \mathrm{N}^A_n:=\{i\in   \mathrm{N}^A:   \exists m<n\, (c_i=c^A_m)\},
\end{equation}
the set of indices  of the first  $n$
coding nodes  in $A$.
  Recall that $\om$ is the set of vertices for $\bK$,
 and that we  often  use   $v_i$ to denote $i$,  the $i$-th vertex of $\bK$.
Thus,  $\mathrm{N}^A$ is precisely the set of vertices of $\bK$
 represented by the coding nodes in $A$.
  Let $\bK\re A$ denote the
 substructure of $\bK$
on  universe  $ \mathrm{N}^A$.
 We call this the
 {\em substructure of $\bK$
 represented by the coding nodes in $A$}, or simply {\em  the substructure represented by $A$}.
\end{notation}

The next definition extends the notion   of {\em passing number}  developed in \cite{Sauer06} and \cite{Laflamme/Sauer/Vuksanovic06} to
code
binary relations
using
 pairs of nodes
in regular splitting trees.
 Here, we extend this notion to relations of any arity.

Recall from
the discussion after
Definition \ref{defn.treecodeK}
that for $s\in \bS$,
 $s(0)$  denotes the set of formulas in $s$
 without parameters; and for
for  $1\le i<|s|$,
 $s(i)$ denotes  the set of those formulas in $s\re \bK_i$
 in which $v_{i - 1}$ appears.

\begin{defn}[Passing Type]\label{defn.passingtype}
Given $s,t\in \bS$  with $|s|<|t|$,
we call
  $t(|s|)$
the {\em passing type of $t$  at $s$}.
We also call $t(|s|)$
the {\em passing type of $t$ at $c_n$},
where $n+1 = |s|$,
as $|c_n|=n+1$.

Let $A$ be a subtree of $\bS$, $t$ be a node in $\bS$,
and
 $c_n$ be a coding node in $\bS$ such that  $|c_n|<|t|$.
We write $t(c_n;A)$ to denote the set of those formulas in
$t(|c_n|)$ in which all  parameters  are  from among
$\{v_i: i\in  \mathrm{N}^A_m\cup\{n\}\}$,
where $m$ is least such that $|c^A_m|\ge |c_n|$.
We call $t(c_n;A)$ the {\em passing type of $t$ at $c_n$ over $A$}.

Given  a  coding node $c_n^A$ in $A$, we write
 $t(n;A)$ to denote  $t(c^A_n;A)$,
  and call this
the {\em passing type of $t$ at $n$ over $A$}.
\end{defn}

Note that passing types are  partial types which   do  not include any unary relation symbols.
Thus,  one can have realizations of the same passing type by elements which differ on the unary relations.
Further, note that the passing type of $t$ at $s$  only takes into consideration the length of $s$, not $s$ itself.
Writing the ``passing type of $t$ at $s$'' rather than ``passing type of $t$ at $|s|$'' continues the convention set forth in \cite{Sauer06}, \cite{Laflamme/Sauer/Vuksanovic06}, and continued in  all papers following  on these two.

\begin{rem}\label{rem.pnspecial case}
In the  case where the language $\mathcal{L}$ only has binary relation symbols,
passing type  reduces to  the concept of passing number,  first defined and  used in
\cite{Sauer06} and \cite{Laflamme/Sauer/Vuksanovic06} and later used  in \cite{DobrinenJML20}, \cite{DobrinenH_k19}, \cite{DobrinenRado19}, \cite{Zucker20}.
This is  because
for binary relational structures,
the tree $\bS$ has a bounded degree of  branching.
In the special case
of the Rado graph,
where the language has exactly one binary relation, say $E$,  the tree $\bS$ is regular $2$-branching and may be correlated with the tree of
finite
sequences of $0$'s and $1$'s;
then
the passing number  $0$ of $t$ at  $s$  corresponds to the passing type
generated by
$\{\neg R(x,v_{|s|})\}$, and
 the passing number  $1$ of $t$ at  $s$  corresponds to the passing type
 generated by
 $\{R(x,v_{|s|})\}$.

In the case of the rationals,
the  coding tree  of $1$-types $\bS$ for $\bQ$  provides a minimalistic  way to view the work of Devlin in
 \cite{DevlinThesis},
 as $\bS$ branches exactly at coding nodes and nowhere else.
 In our set-up, any antichain of coding nodes is automatically a so-called diagonal antichain,
as defined in Subsection \ref{subsec.3.3}.
 This differs from the  previous approaches to big Ramsey degrees of $\bQ$
 in \cite{LavUnp} and  \cite{DevlinThesis} (see also \cite{TodorcevicBK10}),
  which  use  the binary branching tree, Milliken's theorem, and the method of  envelopes.
\end{rem}

We will need to be able to compare structures represented by  different sets of coding nodes in $\bS$.
The next notion provides a way to do so.

Recall that $x$ is  the variable used in all $1$-types in $\bS$. Given subsets $X$ and $Y$ of $\om$ and map $f: X \to Y$, let
$f^*: X \cup \{x\} \to Y \cup \{x\}$ be the extension of $f$ given by $f^*(x) = x$.

\begin{defn}[Similarity of Passing Types over Subsets]\label{defn.prespt}
Let $A$ and $B$ be  subsets of $\bS$,  and let $m,n\in\om$
 be such that
 $\mathrm{N}^A\cap m$ has the same number of elements
  as $\mathrm{N}^B\cap n$, say $p$.
  Let $f$ be the increasing  bijection  from $\mathrm{N}^A_p$ to $\mathrm{N}^B_p$.
Suppose  $s,t\in\bS$  are such that $|c_m|<|s|$ and $|c_n|<|t|$.
We write
\begin{equation}
s(c_m;A)\sim t(c_n;B)
\end{equation}
when,
given any
relation symbol $R\in \mathcal{L}$ of arity $k$
and  $k$-tuple $(z_0,\dots, z_{k - 1})$, where all $z_i$ are from among $\{v_i:i\in \mathrm{N}^A_p\}\cup\{x\}$ and at least one $z_i$ is the variable $x$,
we have that
  $R(z_0,\dots, z_{k-1})$
   is in $s(c_m;A)$ if and only if
  $R(f^*(z_0),\dots,f^*(z_{k-1}))$ is in $t(c_n;B)$.
   When $s(c_m;A)\sim t(c_n;B)$ holds,
 we say that  the passing type of  $s$ at $c_m$ over $A$ is {\em similar} to the passing type of $t$ at $c_n$ over $B$.

If $A$ and $B$ each have  at least $n+1$ coding nodes,
then
for  $s,t\in\bS$ with
 $|c^A_n|<|s|$  and
 $|c^B_n|<|t|$,
define
\begin{equation}
s(n;A)\sim t(n;B)
\end{equation}
  to mean that $s(c^A_n;A)\sim t(c^B_n;B)$.
 When $s(n;A)\sim t(n;B)$,
 we say that {\em $s$ over $A$ and $t$ over $B$ have similar passing types at the $n$-th coding node},
 or that {\em the passing type of  $s$ at $n$ over $A$ is similar to the passing type of $t$ at $n$ over $B$}.
\end{defn}

It is clear that for fixed $n$, $\sim$ is an equivalence relation
on passing types over subsets of $\bS$.

The following fact is the essence of why we are interested in similarity of passing types:
They tell us exactly when two structures represented by coding nodes are isomorphic
 as substructures of the enumerated structure $\bK$; that is, when there exists an
 $\mathcal{L}$-isomorphism between the structures that preserves the order relation on their underlying sets
 inherited from $\om$.

\begin{fact}\label{fact.simsamestructure}
Let $A$ and $B$ be subsets of $\bS$  and $n<\om$ such that  $A$ and $B$ each have  $n+1$ many coding nodes.
Then
the substructures $\bK\re A$ and $\bK\re B$ are isomorphic, as ordered substructures of $\bK$,
if and only if
\begin{enumerate}
\item
For each $i\le n$, the $1$-types
$c^A_i$   and $c^B_i$ contain
the same parameter-free formulas;
and
\item
For all $i<j\le n$, $c^A_{j}(i;A)\sim c^B_j(i;B)$.
\end{enumerate}
\end{fact}

We now extend the similarity relation on passing types over subsets of $\bS$ to a relation on subtrees of $\bS$ that preserves
tree structure. For this, we first define a
(strict)
linear order $\prec$ on $\bS$:
We may assume there is a linear ordering on the relation symbols
and negated relation symbols in $\mathcal{L}$,
with the convention that
all the
negated relation symbols  appear in the linear order before the  relation symbols.
(We make this convention to support the intuition that ``moving left'' from a node in a tree indicates that a relation does not hold, while ``moving right'' suggests that it does; the convention is not necessary for our results.)
 Extend the usual linear order $<$ on $\om$, the underlying set of $\bK$,
to the set $\{x\} \cup \om$ by setting $x<n$ for each $n \in \om$.
Let $(\{x\} \cup \om)^{<\om}$,  the set of finite sequences from $\{x\} \cup \om$, have the induced lexicographic order.
Then the induced lexicographic order on the set
$$
(\{ R : R\in\mathcal{L}\} \cup \{\neg R : R \in \mathcal{L}\}) \times (\{x\} \cup \om)^{<\om}
$$
is a linear order on the set of atomic and negated atomic formulas of $\mathcal{L}$ that have one free variable $x$ and parameters
from $\omega$.  Since any node of $\bS$ is completely determined by such atomic and negated atomic formulas, this
lexicographic order gives rise to a linear order on $\bS$, which we denote $\prec$.
 Observe that by definition of the lexicographic ordering, we have:
 If
 $s \subsetneq t$,
then $s\prec t$; and
for any incomparable   $s,t \in\bS$, if
 $|s\wedge t|=n$,
 then
$s\prec t$ if and only if $s\re(n+1)\prec t\re (n+1)$.
This order $\prec$
generalizes
the lexicographic order  for the case of binary relational structures in \cite{Sauer06}, \cite{Laflamme/Sauer/Vuksanovic06}, \cite{DobrinenJML20}, \cite{DobrinenH_k19}, and \cite{Zucker20}.

\begin{defn}[Similarity Map]\label{def.ssmap}
Let  $S$ and $T$  be meet-closed subsets
of $\bS$.
A function $f:S\ra T$ is a {\em similarity map} of $S$ to $T$ if for all nodes
$s, t \in \bS$,
the following hold:
\begin{enumerate}
\item
$f$ is a bijection which preserves $\prec$:
$s\prec t$ if and only if $f(s)\prec f(t)$.

\item
$f$ preserves meets, and hence splitting nodes:
$f(s\wedge t)=f(s)\wedge f(t)$.

\item
$f$ preserves relative lengths:
$|s|<|t|$ if and only if
$|f(s)|<|f(t)|$.

\item
$f$ preserves initial segments:
$s\sse t$ if and only if $f(s)\sse f(t)$.

\item
$f$ preserves  coding  nodes and their
parameter-free formulas:
  Given a coding node $c_n^S\in S$, $f(c_n^S)=c^T_n$;
moreover,
for  $\gamma\in\Gamma$,
$\gamma(v^S_n)$ holds in $\bK$
  if and only if
 $\gamma(v^T_n)$ holds in $\bK$,
 where $v^S_n$ and $v^T_n$ are the vertices of $\bK$ represented by coding nodes
 $c^S_n$ and $c^T_n$, respectively.

\item
$f$ {\em preserves relative passing types} at coding nodes:
$s(n;S)\sim f(s)(n;T)$, for each $n$ such that $|c^S_n|<|s|$.
\end{enumerate}

When there is a similarity map between $S$ and $T$, we say that $S$ and $T$  are {\em similar} and  we write $S\sim T$.
Given a subtree $S$ of $\bS$,
we let $\Sim(S)$ denote the collection of all subtrees  $T$of $\bS$ which are similar to $S$.
If $T'\sse T$ and $f$ is a  similarity map of $S$ to $T'$, then we say that  $f$ is a {\em similarity embedding} of $S$ into $T$.
\end{defn}

\begin{rem}
It follows from  (2)  that
 $s$ is a splitting node in $S$ if and only if  $f(s)$ is a splitting node in $T$.
 Moreover, if $s$ is a splitting node in $S$,
 then $s$ has the same number of immediate successors in $S$ as $f(s)$ has in $T$.
Similarity is an equivalence relation
on the subtrees of $\bS$,
since the
identity map is a similarity map, the
inverse of a  similarity map is a similarity map, and the composition of two similarity maps is a  similarity map.

Our notion of {\em similarity}   extends  the notion of {\em strong similarity} in \cite{Sauer06} and  \cite{Laflamme/Sauer/Vuksanovic06} for trees without coding nodes, and in \cite{DobrinenJML20} and \cite{DobrinenH_k19} for trees with coding nodes.
We drop the word {\em strong} to make the terminology more efficient, since there is only one notion of similarity  used in this paper.
\end{rem}

Given two
substructures  $\bF,\bG$ of $\bK$, we write $\bF\cong^{\om} \bG$
when there exists an $\mathcal{L}$-isomorphism between $\bF$ and $\bG$ that preserves
 the linear order on
their universes inherited from $\om$.
Note that for any subtrees
$S,T$ of $\bS$,
$S\sim T$ implies that $\bK\re S\cong^{\om}\bK\re T$.


\subsection{Diagonal coding trees and \EEAP$^+$}\label{subsec.3.3}

Our approach to finding exact big Ramsey degrees for structures with unary and  binary relations
 starts with the kinds of trees  that will
actually produce the  exact degrees, upon taking a subcopy of $\bK$ represented by an antichain of coding nodes in such trees.
Namely, we will work with {\em diagonal coding trees}.
This
will  enable us to characterize the  big Ramsey degrees without the need to develop a  notion  of  envelope and pass through that intermediate step.
As this can be done  with  very little additional  work in the forcing arguments, this is our approach:
We will  work with  skew  subtrees of $\bS$   which have two-branching from the outset.
This leads to the first direct proof of exact big Ramsey degrees, without any appeal to envelopes.
The same methods  also enable us to prove that any \Fraisse\ structure satisfying  \EEAP$^+$, with  relations of any arity,  is indivisible.

The following modification of
 Definition \ref{defn.treecodeK}  of $\bS(\bK)$  will be useful especially for \Fraisse\ classes which have both non-trivial unary relations and a linear order or some  similar
relation, such as the betweenness relation.
Recall    that $\Gamma$ denotes the set of complete
$1$-types having only parameter-free formulas;
in particular, the only relation symbols that can occur in any $\gamma \in \Gamma$ will be unary.

\begin{defn}[The Unary-Colored  Coding Tree of $1$-Types,
$\bU(\bK)$]\label{defn.ctU}
Let $\mathcal{K}$ be a \Fraisse\ class in language $\mathcal{L}$ and $\bK$ an enumerated \Fraisse\ structure for $\mathcal{K}$.
For $n < \om$,
let $c_n$ denote the $1$-type
of $v_n$
over $\bK_n$
(exactly as in the definition of $\bS(\bK)$).
Let $\mathcal{L}^-$ denote
the collection of all relation symbols in $\mathcal{L}$ of arity greater than one,
and let $\bK^-$ denote the reduct of $\bK$ to $\mathcal{L}^-$
and $\bK_n^-$ the reduct of $\bK_n$ to $\mathcal{L}^-$.

For $n<\om$, define the {\em $n$-th level, $\bU(n)$}, to  be
the collection
of all $1$-types
$s$ over $\bK^-_n$
in the language $\mathcal{L}^-$
such that
for some $i\ge n$,
$v_i$ satisfies $s$.
Define $\bU$
 to be
$\bigcup_{n<\om}\bU(n)$.
The tree-ordering on $\bU$ is simply inclusion.
The {\em unary-colored coding tree of $1$-types}
is
the tree $\bU$ along with the  function $c:\om\ra \bU$ such that $c(n)=c_n$.
Thus,
$c_n$ is the $1$-type
(in the language $\mathcal{L}^-$) of $v_n$
 in $\bU(n)$  along with the additional ``unary color''
$\gamma\in\Gamma$ such that
$\gamma(v_n)$ holds in $\bK$.
 \end{defn}

Note that $\bK^-$ is not necessarily
a \Fraisse\ structure, as the collection of reducts of members of $\mathcal{K}$ to $\mathcal{L}^-$ need not be a
\Fraisse\ class.
This poses no problem to our uses of $\bU$ or to the results.

\begin{rem}\label{rem.bU}
 In the case that $\mathcal{K}$ has no  unary relations, $\bU$ is  the same as $\bS$.
Otherwise, the difference between $\bU$ and $\bS$
is that all non-coding nodes in $\bU$  are
complete  $1$-types over initial segments of $\bK^-$
in the language $\mathcal{L}^-$,
while all
nodes in $\bS$, coding or non-coding,
are complete $1$-types over initial segments of $\bK$
in the language $\mathcal{L}$.
In particular,  $\bS(0)$ equals $\Gamma$, while
$\bU(0)$
 has  exactly one node, $c_0$.

Definition
\ref{defn.passingtype} of passing type applies  to   $\bU$, as the notion of  passing type involves no unary relations.
Definition
 \ref{defn.prespt}
 of  similarity of passing types
 and Definition \ref{def.ssmap} of similarity maps both apply
to  $\bU$,
 since the notion of coding nodes is the same in both $\bS$ and $\bU$.
Working inside  $\bU$ instead of $\bS$ makes the upper bound arguments   for \Fraisse\ classes with both
a linear order
and unary relations simpler, lining up with the previous approach for big Ramsey degrees of $\bQ_n$ in \cite{Laflamme/NVT/Sauer10}.
This set-up will allow us to do one
uniform
forcing proof in the next section for all classes satisfying \EEAP$^+$.
For  classes with \SFAP, the exact  bound proofs will return to the $\bS$ setting.

Lastly, we point out that the tree $\bU$ extends the approach used by Zucker in \cite{Zucker20} for  certain  free amalgamation  classes
with binary and unary relations.
\end{rem}

The following definition of {\em diagonal}, motivated by Definition 3.2 in \cite{Laflamme/Sauer/Vuksanovic06},
can be found in \cite{DobrinenJML20} and \cite{DobrinenH_k19}.

\begin{defn}[Diagonal tree]\label{def.diagskew}
We call a subtree $T\sse \bS$
or $T\sse\bU$
{\em diagonal}
if each level of $T$ has at most one splitting node,
each splitting node in $T$ has degree two (exactly two immediate successors), and
coding node levels in $T$ have no splitting nodes.
\end{defn}

In  most currently known cases
with finite
big Ramsey degrees,
any {\em persistent} similarity type,
in the sense of Definition \ref{defn.cp},
 is  diagonal (though additional requirements are necessary for
 constrained
  free amalgamation classes
 as in \cite{DobrinenH_k19}, \cite{DobrinenJML20}, and \cite{Zucker20}).
This idea is seen for
unrestricted
structures with finitely many binary relations in \cite{Laflamme/Sauer/Vuksanovic06}, where
Laflamme, Sauer, and Vuksanovic prove, using Milliken's theorem and envelopes, that the persistent similarity types are  diagonal.
We will show below that almost all the structures considered in Section
\ref{sec.EEAPClasses}
have diagonal coding trees.

\begin{notation}\label{notn.cong<}
Given a diagonal subtree $T$
(of $\bS$ or $\bU$)
with coding nodes,
we let
$\lgl c^T_n: n<N\rgl$,
where $N\le\om$,
  denote the enumeration of the coding nodes in $T$ in order of increasing length.
Let $\ell^T_n$ denote  $|c^T_n|$, the {\em length} of $c^T_n$.
We shall call a node in $T$ a {\em critical node} if it is either a splitting node or a coding node in $T$.
Let
\begin{equation}
\widehat{T}=\{t\re n:t\in T\mathrm{\ and \ } n\le |t|\}.
\end{equation}
Given
$s\in T$ that is not a splitting node in $T$,
we let $s^+$ denote the immediate successor of $s$ in $\widehat{T}$.
Given any
$\ell$,
we let $T\re\ell$ denote the set of those nodes in $\widehat{T}$ with length $\ell$,
and we let
  $T\rl \ell$
  denote the
  union of the
  set of nodes in $T$ of length less than  $\ell$
  with the set $T\re\ell$.
\end{notation}

Extending
Notation \ref{notn.KreA} to subtrees $T$ of either $\bS$ or $\bU$, we write $\bK\re T$ to denote the substructure of
$\bK$ on $\mathrm{N}^T$, the set of vertices of $\bK$ represented by the coding nodes in $T$.

\begin{defn}[Diagonal Coding Subtree]\label{defn.sct}
A subtree $T\sse\bU$ is called a {\em diagonal coding subtree} if $T$ is diagonal and  satisfies the following properties:
  \begin{enumerate}
 \item
   $\bK\re T\cong\bK$.
  \item
  For each $n<\om$, the collection of $1$-types
  in
  $T\re (\ell^T_n+1)$ over $\bK\re (T\rl \ell^T_n)$
   is in
    one-to-one correspondence with the collection of
   $1$-types in $\bU(n+1)$.
   \item[(3)]
Given $m<n$
and letting
  $A:=T\rl(\ell^T_m-1)$,
if $c^T_n\contains c^T_m$
 then
 $$
  (c^{T}_n)^+(c^{T}_n; A)
  \sim
  (c^{T}_m)^+(c^{T}_m; A).
$$
\end{enumerate}
Likewise, a subtree $T\sse\bS$ is a {\em diagonal coding subtree} if the above hold with $\bU$ replaced by $\bS$.
\end{defn}

\begin{rem}
Requirement (3) aids in the proofs in the next section and can be  met by
the \Fraisse\ limit of
any \Fraisse\ class satisfying \EEAP.
Note that if $T\sse\bU$
(or $T\sse\bS$)
satisfies (3),
then any subtree $S$ of $T$ satisfying $S\sim T$ automatically satisfies (3).
\end{rem}

Now we are prepared to define the Diagonal Coding Tree Property, which is an assumption in Definition \ref{defn_EEAP_newplus}
of \EEAP$^+$.
We say that a tree  $T$ is  {\em perfect}  if $T$ has  no terminal nodes, and  each node in  $T$ has  at least  two  incomparable extensions in $T$.

Recall  our assumption  that any \Fraisse\ class $\mathcal{K}$ that we consider has
 at least one non-unary relation symbol in its language.
We make this assumption because
if $\mathcal{K}$  has  only unary relation symbols in its language, then $\bS$ is a disjoint union of finitely many infinite branches.
In this case,  finitely many applications of  Ramsey's Theorem  will yield finite big Ramsey degrees.

We point out that whenever $\mathcal{K}$ satisfies \SFAP, every node in  $\bS$ (and also in $\bU$) has at least two immediate successors.
However, there  are \Fraisse\ classes
in binary relational languages
 that satisfy \EEAP, and yet for which the trees $\bS$ and  $\bU$ are not perfect; for example, certain \Fraisse\ classes of ultrametric spaces.
In such cases,
Theorem  \ref{thm.matrixHL}
does not apply, as
the forcing posets
used in its  proof
are atomic.
Thus, one of the requirements for \EEAP$^+$ is that there is a perfect subtree of $\bU$ which codes a copy of $\bK$, whenever $\mathcal{L}$  has relation symbols of arity greater than one.
This is an ingredient in the next property.

\begin{defn}[Diagonal Coding Tree Property]\label{defn.DCTP}
A \Fraisse\ class $\mathcal{K}$ in language $\mathcal{L}$  satisfies the {\em Diagonal Coding Tree Property}
if  given any enumerated \Fraisse\ structure $\bK$
for $\mathcal{K}$,
there is a diagonal coding
subtree $T$ of either $\bS$ or $\bU$
such that $T$ is perfect.
\end{defn}

From here
through most of Section
\ref{sec.FRT},
we will simply work in $\bU$ to avoid duplicating arguments, noting that
for \Fraisse\ classes with \SFAP, or
without \SFAP\ but with \Fraisse\ limits having
\EEAP$^+$ and
in a language with no unary relation symbols,
the following can all be done inside $\bS$.

We now define the  space of coding subtrees of $\bU$ with which we shall be working.

\begin{defn}[The Space of  Diagonal Coding Trees of $1$-Types, $\mathcal{T}$]\label{def.subtree}
Let $\bK$ be any enumerated \Fraisse\  structure
and let $\bT$ be a fixed diagonal coding subtree of $\bU$.
Then the space of coding trees
 $\mathcal{T}(\bT)$ consists of all   subtrees $T$ of $\bT$ such that
 $T\sim\bT$.
Members of $\mathcal{T}(\bT)$ are called simply {\em coding trees}, where diagonal is understood to be  implied.
We shall usually simply write $\mathcal{T}$ when $\bT$ is clear
from context.
For $T\in \mathcal{T}$,  we write
$S\le T$ to mean that  $S$ is a  subtree of $T$ and $S$ is a member of $\mathcal{T}$.
\end{defn}

\begin{rem}
Given  $\bT$ satisfying   (1)--(3) in Definition \ref{defn.sct}, if $T\sse \bT$  satisfies $T\sim \bT$,  then $T$ also satisfies (1)--(3).
Any tree $T$ satisfying  (1) and (2)  has no terminal nodes and has coding nodes dense in $T$.
Condition (2) implies that
the \Fraisse\ structure $\bJ:=\bK\re T$  represented by $T$ has the following  property:
For
any $i-1<j<k$ in $\mathrm{J}$ satisfying
$\bJ\re (i\cup\{j\}) \cong \bJ\re (i\cup\{k\})$,
it holds
 that
$\type(j/\bK_i)=\type(k /\bK_i)$;
equivalently,  that whenever
two vertices in $\mathrm{J}$ are in the same orbit  over
$\bJ_i$ in $\bJ$, they
are in the same orbit over $\bK_i$
in $\bK$.
\end{rem}

The
first use
of diagonal subtrees of the infinite binary tree in characterizing exact big Ramsey degrees
 was
for
 the rationals in \cite{DevlinThesis}.
Diagonal subtrees of the infinite binary tree  turned out to be at the heart of
  characterizing the exact big Ramsey degrees of
the Rado graph  as well as of the generic directed graph and the  generic tournament in
\cite{Sauer06} and
\cite{Laflamme/Sauer/Vuksanovic06}.
More generally,
diagonal subtrees of  boundedly branching trees turned out to be central to the characterization of
big Ramsey degrees of unconstrained
 structures with finitely many binary relations
 in
\cite{Sauer06} and
\cite{Laflamme/Sauer/Vuksanovic06}.
More recently,
  characterizations of the  big Ramsey degrees
for    triangle-free graphs were found to
   involve
 diagonal  subtrees (\cite{DobrinenJML20},\cite{DobrinenH_3ExactDegrees20}), and similarly, for free amalgamation classes with finitely
 many binary relations and finitely many
 finite
 forbidden irreducible substructures  on three or more vertices (\cite{Balko7},\cite{DobrinenH_k19},\cite{Zucker20}).
 However, in these cases,
  properties additional to being diagonal are essential to characterizing their big Ramsey degrees;
  hence, their big Ramsey degrees do not have a ``simple'' characterization solely in terms of similarity types of antichains of coding nodes in diagonal coding trees.
We will prove that, similarly to the rationals and the Rado graph,
 all
unary and binary relational \Fraisse\  classes  with \Fraisse\ structure satisfying  \EEAP$^+$ have big Ramsey degrees which are  characterized  simply by similarity types of
antichains of coding nodes in diagonal  coding trees, along with the passing types of their coding nodes.

Recalling from Notation \ref{notn.cong<}
that
  $t\in T$ is called  a
{\em critical node}
if $t$ is either a splitting node or a coding node in $T$,
any two critical nodes in a diagonal coding tree
 have different lengths, and thus,
the levels of $T$ are designated by the lengths of the critical  nodes in $T$.
(This follows from the definition of {\em diagonal}.)
 If $\lgl d^T_m:m<\om\rgl$  enumerates  the critical nodes in $T$ in order of strictly increasing length,
then we let
$T(m)$ denote the
collection of those nodes in $T$ with length $|d^T_m|$, which we call the {\em $m$-th level} of $T$.

Given a substructure $\bJ$ of $\bK$,
we let $\bU\re\bJ$ denote the subtree of $\bU$
induced by the meet-closure of the coding nodes $\{c_n:n\in \mathrm{J}\}$.
 We call $\bU\re \bJ$ the {\em subtree of $\bU$ induced by $\bJ$}.
 If $\bJ=\bK\re T$ for some  $T\in\mathcal{T}$,
 then $\bU\re \bJ=T$, as $T$ being  diagonal ensures that the  coding nodes in $\bU\re \bJ$ are exactly those in  $T$.

We now  state the property truly  at the heart of this paper.
This  is   the
property  which makes  the forcing arguments in the next section simpler than the arguments for binary relational structures omitting some finite set of
irreducible substructures on three or more vertices.
As seen in Section
\ref{sec.EEAPClasses},
this property
 unifies a seemingly disparate collection of \Fraisse\ classes
 with finite big Ramsey degrees.

\begin{defn}[\EEAP$^+$, Coding Tree Version]\label{def.EEAPCodingTree}
A \Fraisse\ class $\mathcal{K}$ satisfies the
 {\em
Coding Tree Version of  \EEAP$^+$}
 if and only if   $\mathcal{K}$ satisfies
 the disjoint amalgamation property
 and,
 letting $\bK$ be any enumerated \Fraisse\ limit of $\mathcal{K}$, $\bK$ satisfies
 the Diagonal Coding Tree Property,
the Extension Property, and
the following
condition:

Let $T$ be any  diagonal coding  subtree   of $\bU(\bK)$  (or of $\bS(\bK)$), and let $\ell<\om$ be given.
Let
$i,j$ be any distinct integers such that
$\ell<\min(|c^T_i|,|c^T_j|)$,
and let
$\bfC$ denote  the substructure of $\bK$  represented by the coding nodes in $T\rl \ell$ along with
 $\{c^T_i,c^T_j\}$.
Then there are $m \ge\ell$
and
$s',t'\in T\re m$ such that
$s'\contains s$ and $t'\contains t$
 and,
assuming (1) and (2), the conclusion holds:
 \begin{enumerate}
\item[(1)]
Suppose $n\ge m$ and
$s'',t''\in T\re n$
 with $s''\contains s'$ and
$t''\contains t'$.

\item[(2)]
Suppose $c^T_{i'}\in T$ is any coding node extending $s''$.
\end{enumerate}
Then
there is a coding node $c^T_{j'}\in T$, with $j'>i'$,
such that
$c_{j'}\contains t''$ and
the substructure  of $\bK$  represented by the coding nodes in $T\rl \ell$ along with $\{c^T_{i'},c^T_{j'}\}$ is isomorphic to
$\bfC$.
\end{defn}

\begin{rem}
The structures
$\bK\re (T\rl \ell)$, $\bK\re (T\rl m)$, and $\bK\re (T\rl n)$ above
 play the roles  of $\bfA$, $\bfA'$, and $\bfB$, respectively, in Definition \ref{defn.EEAP_new}.
 The Extension Property  will be defined in
the next section, in Definition \ref{defn.ExtProp}.
 Suffice it to say here that the Extension Property is easily  satisfied by
 \Fraisse\ limits of
 all classes  satisfying  \SFAP, and
 of all classes
 satisfying  \EEAP\ that are considered in this paper. It may even turn out to follow from \EEAP, but we include it to be precise about what is being assumed in \EEAP$^+$.
\end{rem}

The final work in this subsection is to prove that most of the  \Fraisse\ classes with \EEAP\ discussed in Section
\ref{sec.EEAPClasses}
have
\Fraisse\ limits satisfying the Diagonal Coding Tree Property,
thus completing
another part of
the proofs
that these classes have \Fraisse\ limits
satisfying
\EEAP$^+$.
Verification that these \Fraisse\ limits satisfy the Extension Property will be carried out in  Section \ref{sec.FRT}.

We begin with  free amalgamation classes satisfying \SFAP.
For
\Fraisse\ limits of
these classes, we can always construct a diagonal coding subtree of $\bS$.
The following notation will be used in  the rest of this subsection.
Given  $j<\om$, sets  vertices $\{v_{m_i}:i<j\}$ and
$\{v_{n_i}:i<j\}$, and
$1$-types $s,t\in\bS$ such that $|s|>m_{j-1}$ and $|t|>n_{j-1}$, we will  write
\begin{equation}
s\re (\bK\re\{v_{m_i}:i<j\})\sim t\re (\bK\re\{v_{n_i}:i<j\})
\end{equation}
exactly when, for each
$i<j$,
$s(c_{m_i}; \{c_{m_k}:k<i\})\sim
t(c_{n_i}; \{c_{n_k}:k<i\})$.

\begin{thm}\label{thm.SFAPimpliesDCT}
\SFAP\ implies
 \EEAP$^+$.
\end{thm}

\begin{proof}
Suppose $\mathcal{K}$ is a \Fraisse\ class satisfying \SFAP.
Then  $\mathcal{K}$ automatically also satisfies \EEAP.

Let  $\bK$ be any enumerated \Fraisse\ structure for $\mathcal{K}$, and let $\bS$ be the coding tree of $1$-types over finite initial segments of $\bK$.
Recall that  $c_n$ denotes the $n$-th coding node of $\bS$, that is, the  $1$-type of
the $n$-th vertex of $\bK$ over $\bK_n$.
If there are any  unary relations in the language $\mathcal{L}$ for $\mathcal{K}$, then $\bS(0)$ will have more than one node.
Recall our convention that the ``leftmost'' or $\prec$-least
node in $\bS(n)$  is the $1$-type over $\bK_n$ in which no relations of arity greater than one are satisfied.

We start constructing
a diagonal coding subtree
$\bT$ by letting the minimal level of $\bT$ equal $\bS(0)$.
Take a level set $X$ of $\bS$ satisfying
(a)
for each $t\in \bS(0)$,
the number of nodes in $X$ extending $t$
is the same as the number of nodes in $\bS(1)$ extending  $t$, and
(b)
 the subtree $U_0$ generated by the meet-closure  of $X$ is diagonal.
 We may assume, for convenience, that the $\prec$-order of the  splitting nodes in $U_0$   is the same as the ordering by their lengths.

Let $x_*$ denote the $\prec$-least member of $X$ extending $c_0$.
(If there are no  unary relation symbols in the language, then $x_*$ is the
``leftmost'' or $\prec$-least
node in $X$.)
Let $c^{\bT}_0$ denote the coding node of least length extending
 $x_*$.
Extend the rest of the nodes in $X$  to the length of $c^{\bT}_0$ and call this set of nodes, along with $c^{\bT}_0$,  $Y$;  define $\bT\re |c^{\bT}_0|=Y$.
Then  take one immediate successor in $\bS$  of each member of $Y$ so that  there is  a one-to-one correspondence between the $1$-types in $Y$ over   $\bK\re\{v^T_0\}$,
  where $v_0^{\bT}$ is the vertex in $\bK$ represented by
 $c^{\bT}_0$, and the $1$-types in $\bS(1)$:
Letting $p=|\bS(1)|$,
list the nodes in $\bS(1)$ and $Y$ in $\prec$-increasing order as
$\lgl s_i:i<p\rgl$ and $\lgl y_i:i<p\rgl$, respectively.
Take $z_i$  to be an
 immediate successor of $y_i$ in $\bS$  such that
   $z_i\re (\bK\re\{v_0^{\bT}\})\sim s_i$.
 Such  $z_i$ exist by \SFAP.
 Let $\bT\re (|c^{\bT}_0|+1|)=\{z_i:i<p\}$.
 This constructs $\bT$ up to  length  $|c^{\bT}_0|+1$.

The rest of $\bT$ is constructed similarly:
Suppose  $n\ge 1$ and $\bT$ has been constructed up to the immediate  successors of its $(n-1)$-st coding node, $c^{\bT}_{n-1}$.
Take $W$ to be the set of nodes in $\bT$ of length
 $|c^{\bT}_{n-1}|+1$.
This set $W$ has the same size as $\bS(n)$;
let $\varphi: W\ra \bS(n)$ be the $\prec$-preserving bijection.
Take a level set $X$ of nodes in $\bS$ extending $W$ so that (a) for each $w\in W$, the number of nodes in $X$ extending $w$ is the same as the number of nodes in $\bS(n+1)$ extending $\varphi(w)$, and (b)
the tree $U$ generated by the meet-closure of $X$ is diagonal.
Again, we may assume that the splitting nodes in $U$  increase in length as their $\prec$-order increases.

Note that $X$ and  $\bS(n+1)$ have the same cardinality.
Let $p=|\bS(n+1)|$ and enumerate $X$ in $\prec$-increasing order as $\lgl x_i:i<p\rgl$.
Let  $i_*$ be the index so that
$x_{i_*}$ is the $\prec$-least member of $X$ extending
$\varphi(c_{n})$.
Let $c^{\bT}_{n}$ denote the coding node in $\bS$ of shortest length
extending $x_{i_*}$.
For each  $i\in p\setminus \{i_{*}\}$, take one  $y_i\in \bS$ of length $|c^{\bT}_{n}|$ extending $x_i$.
Let $y_{i_*}=c^{\bT}_{n}$,
$Y=\{y_i:i < p\}$, and
 $\bT\re |c^{\bT}_{n}|=Y$.
Let $\lgl s_i:i<p\rgl$  enumerate the nodes in
$\bS(n+1)$ in   $\prec$-increasing order.
Then for each $i<p$,
let $z_i$ be an immediate successor of $y_i$  in $\bS$ satisfying
\begin{equation}
z_i\re (\bK\re\{v_m^{\bT}:m\le n\})\sim
s_i\re \bK_{n+1},
\end{equation}
where $v_m^{\bT}$ is the vertex of $\bK$ represented by the coding node $c_m^{\bT}$.
Again, such $z_i$ exist by \SFAP.
Let $\bT\re (|c^{\bT}_{n}|+1)=\{z_i:i<p\}$.

In this manner, we construct a subtree $\bT$ of $\bS$.
It is straightforward to check that this construction satisfies  (1) and (2) of Definition \ref{defn.sct}
of diagonal coding tree.
Using  \SFAP,  we may construct $\bT$ so that
 property (3) holds.
As long as the language for $\mathcal{K}$ contains at least one relation symbol, $\bT$ will be  a perfect  tree.
Thus,
any \Fraisse\ limit for
$\mathcal{K}$ satisfies the Diagonal Coding Tree Property.
We show in Lemma \ref{lem.SFAPEP} that any \Fraisse\ limit for $\mathcal{K}$ will also satisfy the Extension
Property, completing the proof that \Fraisse\ limits of \Fraisse\ classes with \SFAP\  have \EEAP$^+$.
\end{proof}

The same argument works
as in Theorem \ref{thm.SFAPimpliesDCT}
  for all
unrestricted structures, and hence we have the following proposition as the Extension Property trivially holds.
(See \cite{Laflamme/Sauer/Vuksanovic06} for unrestricted binary relational structures.)

\begin{prop}\label{prop.UnrestrSDAPplus}
If  $\mathcal{K}$ is an unrestricted \Fraisse\ class, then any enumerated \Fraisse\ structure for $\mathcal{K}$   satisfies \EEAP$^+$.
\end{prop}

The next lemma completes the proofs of  Propositions
\ref{prop.LO_n} and
\ref{bQn}.

\begin{lem}\label{lem.DCTLO}
There is a diagonal coding tree
representing  $\bQ_n$, for each $n\ge 1$
Hence,
these structures
have the Diagonal Coding Tree Property.
\end{lem}

\begin{proof}
We have already seen in Figure \ref{fig.Qtree} in Subsection \ref{subsec.3.1} that $\bS(\bQ)=\bU(\bQ)$ is a skew tree with binary splitting.
Similarly, for $n\ge 2$, $\bU(\bQ_n)$ is a skew tree with binary splitting.
In Figure \ref{fig.Q2tree}, we have seen that $\bS(\bQ_n)$ consists of $n$-many trees which are isomorphic copies; $\bU(\bQ_n)$ consolidates these into one tree with $n$ different ``unary-colored''
coding nodes.
Thus, to construct a diagonal coding subtree  $\bT$ of $\bU(\bQ_n)$, it only remains to choose splitting nodes for $\bT$ (which are coding nodes in $\bU$ but not in $\bT$) and then choose other coding nodes in $\bU$ to be inherited as the coding nodes in $\bT$, so as to satisfy requirements (2) and (3) of Definition \ref{defn.sct}, the definition of diagonal coding subtree.
The construction is a slight modification of the one given in \cite{Laflamme/NVT/Sauer10}, where they constructed diagonal antichains of (non-coding) trees for $\bQ_n$.

Take the only node in $\bU(0)$, $c_0$, to be the least splitting node in $\bT$.
 Let $\bT\re 1$ consists of
 the two immediate successors of $c_0$ in $\bU$, say $s_0\prec s_1$.
Then extend $s_0$ to the next coding node in $\bU$, and label this node $c^{\bT}_0$.
If $n\ge 2$, we also require that
$c^{\bT}_0$   satisfies the same unary relations as $c_0$ does.
Take any extension $t_1\contains s_1$ in $\bU$ of length
 $|c^{\bT}_0|$.
The set $\{t_0,t_1\}$ make up the nodes in $\bT$ at the level of its least coding node,  $c^{\bT}_0$.
Extend $c^{\bT}_0$ $\prec$-leftmost in $\bU$, call this node $u_0$.
There is only one immediate successor of $t_1$ in $\bU$, call it $u_1$.
Let $\bT\re (|c^{\bT}_0|+1)=\{u_0,u_1\}$.

In general, given $n\ge 1$ and $\bT$ constructed up to nodes of length $|c^{\bT}_{n-1}|+1$,
enumerate these nodes in $\prec$-increasing order as $\lgl t_i:i<n+2\rgl$.
Let $j$ denote the index of the node that will be extended to the next coding node, $c^{\bT}_n$.
This is the only node that needs to branch before the level of $c^{\bT}_n$.
Let $s$  be the shortest   splitting node in $\bU$
extending $t_{j}$.
Denote  its immediate successors by $s_0,s_1$, where
$s_0\prec s_1$.
Let $c^{\bT}_n$ be the coding node of least length in $\bS$ extending $s_0$; if $n\ge 2$, also require that $c^{\bT}_n$ satisfies the same unary relation as $c_n$.
Extend all the nodes $s_0$ and $t_i$, $i\in (n+2)\setminus\{j\}$ to nodes in $\bS$ of length
$|c^{\bT}_n|$.
These nodes along with $c^{\bT}_n$ construct $\bT\re |c^{\bT}_n|$.
Take the $\prec$-leftmost extension of $c^{\bT}_n$ to be its immediate successor in
$\bT$.
All other nodes in $\bT\re |c^{\bT}_n|$ have only one immediate successor in $\bS$, so there is no choice to be made.

This constructs a diagonal tree $\bT$  representing a copy of $\bQ_n$.
Note that taking the $\prec$-leftmost extension of each coding node has the effect that
all extensions of any  coding node  $c^{\bT}_n$ in $\bT$ include the formula $x<v^{\bT}_n$,
satisfying  (3) of the definition of  diagonal coding tree.
\end{proof}

Next, we consider  ordered \SFAP\ classes.

\begin{lem}\label{lem.SFAPplusoderimpliesDCT}
Suppose $\mathcal{K}$ is a \Fraisse\ class satisfying \SFAP\
 and let  $\mathcal{K}^{<}$
denote  the \Fraisse\ class
 of  ordered
 expansions of members of
 $\mathcal{K}$.
 Then
 the \Fraisse\ limit $\bK^<$ of
  $\mathcal{K}^{<}$ satisfies \EEAP$^+$.
\end{lem}

\begin{proof}
Let $\mathcal{L}$ denote the language for $\mathcal{K}$,
and let $\mathcal{L}^*$ be the expansion $\mathcal{L}\cup\{<\}$, the language of $\mathcal{K}^<$.
Let $\bK^{<}$ denote
an enumerated structure for $\mathcal{K}^{<}$,
and let $\bK$ denote the  reduct of  $\bK^{<}$ to $\mathcal{L}$; thus, $\bK$ is an enumerated \Fraisse\ structure for $\mathcal{K}$.
The universes of $\bK$ and $\bK^{<}$ are  $\om$, which  we shall denote as
 $\lgl v_n:n<\om\rgl$.
Let
$\bU$ denote  the coding tree of $1$-types induced by $\bK$, and
$\bU^{<}$ denote  the coding tree of $1$-types induced by $\bK^{<}$.
As usual, we let $c_n$ denote the $n$-th coding node in $\bK$, and we will let $c_n^<$ denote the $n$-th coding node in $\bK^<$.
(Normally, if $<$ is in the language of a \Fraisse\ class $\mathcal{K}$, then we will simply write $\bK$ for its enumerated \Fraisse\ structure and $\bU$ for its induced coding tree of $1$-types, but here it will aid the reader to consider the juxtaposition of $\bU$ and $\bU^<$.)
Notice that  $\mathcal{K}^<$ satisfies  \EEAP:
 This holds because
  \SFAP\ implies \EEAP,
 $\mathcal{LO}$ satisfies \EEAP, and
 \EEAP\ is preserved under free superposition.
So it only remains to show that there is a diagonal coding tree for $\bK^<$.

Note that since $\mathcal{L}$ has  at least one non-unary relation symbol and since
$\mathcal{K}$ satisfies \SFAP,
every node in
the tree $\bU$  has at least two immediate successors.
The branching of $\bU$ and $\bU^<$ are related in the following way:
Each node $t\in\bU(0)^<$ has twice as many immediate successors in $\bU(1)^<$
as its reduct to $\mathcal{L}$ has in $\bU(1)$.
In general, for $n\ge 1$,
given a node  $t\in \bU^<(n)$,
let  $s$  denote the collection of formulas in $t$ using only relation  symbols in   $\mathcal{L}$
 and note that  $s\in \bU(n)$.
  The number of immediate successors of $t$ in
  $ \bU^<(n+1)$ is related to the number of immediate successors of $s$ in   $\bU(n+1)$ as follows:
  Let $(*)_n(t)$ denote the following property:
  \begin{enumerate}
  \item[$(*)_n(t)$:]
  \begin{center}
   $\{m<n:(x<v_m)\in t\}=\{m<n:(x<v_m)\in c^<_n\}$
   \end{center}
\end{enumerate}
If $(*)_n(t)$ holds,
then $t$ has twice as many immediate successors in
 $\bU^<(n+1)$ as $s$ has in  $\bU(n)$, owing to the fact that
 each $1$-type in $\bU(n+1)$ extending $s$  can be augmented by either of $(x<v_n)$ or $(v_n<x)$ to form an extension of $t$ in $\bU^<(n+1)$.
 If $(*)_n(t)$ does not  hold,
 then any  vertex  $v_i$, $i>n$, satisfying  $t$
  lies in an interval of the $<$-linearly ordered set $\{v_m:m<n\}$, where neither of  the endpoints  are  $v_n$.
 Thus, the order between $v_i$ and $v_n$ is already  determined by $t$;
 hence $t$ has the same number of immediate successors in
  $\bU^<(n+1)$ as $s$ has in  $\bU(n)$.

A diagonal coding subtree $\bT^<$ of $\bU^<$  can be constructed similarly as in Theorem \ref{thm.SFAPimpliesDCT}
with the following modifications:
Suppose $\bT^<$ has been constructed up to  a
level  set $W$, where either  $n=0$ and $W=\bU^<(0):=\{c^<_0\}$, or else $n\ge 1$ and  $W$ is the set of immediate successors of the $(n-1)$-st coding node  of $\bT^<$.
This set $W$ has the same size as $\bU^<(n)$;
let $\varphi: W\ra \bU^<(n)$ be
the $\prec$-preserving bijection.
As $\mathcal{K}$ is a free
amalgamation
class,
we may assume that for any $s\in\bU^<(n)$,
if $t,u$ in
$\bU^<(n+1)$  are
 immediate successors   of $s$  with
$(x< v_n)\in t$ and $(v_n<x)\in u$,
then $t\prec u$.
Note that the two $\prec$-least extensions of $s$ either both contain $(x<v_n)$, or else both contain
$(v_n<x)$.
Moreover, we may assume that the $\prec$-least immediate successor of $s$
contains negations of all relations in $s$ with $v_n$ as
a parameter.

Take a level set $X$ of nodes in $\bU^<$ extending $W$ so that  the following hold:
(a)
for each $w\in W$, the number of nodes in $X$ extending $w$ is the same as the number of nodes in $\bU^<(n+1)$ extending $\varphi(w)$, and
(b)
the tree $U$ generated by the meet-closure  of $X$ is diagonal,
 where each splitting node in $U$ is extended by
 its  two $\prec$-least immediate successors in $\bU^<$, and
  all non-splitting nodes are extended by the $\prec$-least extension in
  $\bU^<$.
As in Theorem \ref{thm.SFAPimpliesDCT},
 we may assume that the splitting nodes in $U$  increase in length as their $\prec$-order increases, though this has no bearing on the  theorems in the next section.

Let $p:=|\bU^<(n+1)|$ and index the nodes in $\bU^<(n+1)$ in $\prec$-increasing order as $\lgl s_i:i<p\rgl$.
Note that $X$ has $p$-many nodes; index them in $\prec$-increasing order as $\lgl x_i:i<p\rgl$.
Let $x_{i_*}$ denote  the $\prec$-least member of $X$ extending
$\varphi(c^{<}_{n})$, and
extend $x_{i_*}$ to a coding node  in $\bU^<$ satisfying the same $\gamma\in\Gamma$ as $c_n$; label it $y_{i_*}$.
This node $y_{i_*}$ will be the $n$-th coding node, $c^{\bT^<}_{n}$, of  the diagonal coding subtree $\bT^<$ of $\bU^<$ which we are constructing.
For each  $i\in p\setminus \{i_*\}$,
 take one  $y_i\in \bU^<$ of length $|c^{\bT^<}_{n}|$ extending $x_i$
so that  $y_i$ is the
 $\prec$-least  extension of $x_i$, subject to  the following:
Let $n_*$ be the index such that
 $c^{\bT^<}_{n}=c^{<}_{n_*}$.
For $i<p$, if  $(v_{n}<x)$ is in $s_i$,
then we take $y_i$ so that for some
$m<n_*$
such that $v_n<v_m$,
$(v_m<x)$ is in $y_i$.
This has the effect that
if  $(v_{n}<x)$ is in $s_i$,
then any vertex $v_j$ represented by a coding node extending $s_i$ will satisfy $v_m<v_j$;
and since $v_n<v_m$, it will follow that $v_n<v_j$;
hence  $(v_{n_*}<x)$ is automatically in $y_i$.
Likewise, if
 $(x<v_{n})$ is in $s_i$,
then we take $y_i$ so that for some
$m<n_*$
such that $v_m<v_n$,
$(x<v_m)$ is in $y_i$.

Let $Y=\{y_i:i<p\}$ and
define  the  set of nodes  in
$\bT^<$ at the level of $c^{\bT^<}_{n}$ to be $Y$.
For each $i<p$,
let $z_i$ be an immediate successor of $y_i$  in $\bU^<$ satisfying
\begin{equation}
z_i\re (\bK\re\{v_j^{\bT^<}:j\le n\})\sim
s_i\re \bK_{n+1},
\end{equation}
where $v_j^{\bT^<}$ is the vertex
represented by $c_j^{\bT^<}$.
This is possible by \SFAP.
For the linear order, this was taken care of by \EEAP\ and
our choice of $y_i$.
Let $\bT^<\re (|c^{\bT}_{n}|+1)=\{z_i:i<p\}$.

In this manner, we construct a coding subtree $\bT^<$ of $\bU^<$ which is diagonal, representing a substructure of $\bK^<$ which is again isomorphic to
$\bK^<$.
By extending coding nodes in $\bT^<$ by their $\prec$-least extensions in $\bU^<T$,
we satisfy
 (3) of  the definition of diagonal coding tree .

 Hence,
 $\bK^<$
satisfies the Diagonal Coding Tree Property, and the Extension Property trivially holds.
 Therefore,
 $\bK^<$
 satisfies \EEAP$^+$.
\end{proof}

A similar construction produces a diagonal coding tree for any enumerated \Fraisse\ limit of the ordered expansion of an unrestricted \Fraisse\ class.  Thus, we have the following.

\begin{lem}\label{lem.orderedunrestr}
Suppose $\mathcal{K}$ is an unrestricted \Fraisse\ class
 and let  $\mathcal{K}^{<}$
denote  the \Fraisse\ class
 of  ordered
 expansions of members of
 $\mathcal{K}$.
 Then
 the \Fraisse\ limit $\bK^<$ of
  $\mathcal{K}^{<}$ satisfies \EEAP$^+$.
\end{lem}

Next, we construct  diagonal coding  trees for the
linear order with a
convexly ordered equivalence relation and its iterates of coarsenings,
as well as for the ``mixed'' structures consisting of a single linear order, finitely many nested convexly ordered equivalence relations, and a
vertex
partition into finitely many definable dense pieces.
The main difference between these structures and the ones considered in the previous lemmas, is that now one has to be careful when choosing splitting nodes, as
not all splitting nodes can be extended to the same structures.
The following lemma completes the proof of
the Diagonal Coding Tree Property in
Proposition \ref{prop.loe}.

\begin{lem}\label{lem.Q_Qbiskew}
$(\bQ_{\bQ})_n$, for each $n\ge 1$,
has  the Diagonal Coding Tree Property
Moreover,
the \Fraisse\ limit of
any  class $\mathcal{K}$ in
$\mathcal{LOE}_{1,n,p}$
also has the Diagonal Coding Tree Property.
\end{lem}

\begin{proof}
We present the construction for $\bQ_{\bQ}$ and then
discuss the construction for the more general case.
Let
$\bU$ denote $\bU(\bQ_{\bQ})$.
It may aid the reader to recall  Figure \ref{fig.QQtree}, where a graphic is   presented for a particular enumeration of
$\bQ_{\bQ}$.

We  construct a subtree $\bT$ of
$\bU$
 which is diagonal and such that
for each $m$, the immediate successors of the nodes in
$\bT\re |c^{\bT}_m|$
have $1$-types over $\bQ_{\bQ}\re \{v^{\bT}_j:j\le m\}$ which
are in one-to-one correspondence (in $\prec$-order)
with the $1$-types in
$\bU(m+1)$.
The idea is relatively simple:
We work our way from the outside (non-equivalence) inward (equivalence) in the way we construct the splitting nodes in $\bT$.

Given $\bT\re |c^{\bT}_{m-1}|$,
let
$\varphi:\bU(m)\ra \bT\re |c^{\bT}_{m-1}|$
 be the $\prec$-preserving bijection, and
let $t_*$ denote the node
$\varphi(c_m)$ in $\bT\re |c^{\bT}_{m-1}|$.
This $t_*$ is the node which we need to extend to the next coding node.
Recall that only the coding nodes in
$\bU$
 have more than one immediate successor; so $t_*$ is the only node we need to extend to one or three splitting nodes before making the level  $\bT\re |c^{\bT}_m|$.

The simplest case is when the coding node $c_m$ has two immediate successors:
these contain $\{x<v_m, xE v_m\}$ and $\{v_m<x, xE v_m\}$, respectively.
First extend $t_*$ to a coding node
$c_i\in \bU$,
 and then take extensions $s_0,s_1$ of this coding node  so that
$\{x<v_i, x E v_i\}\sse s_0$ and
$ \{v_i<x, x E v_i\}\sse s_1$.
Extend $s_0$ to a coding node
$c_j\in \bU$,
and define $c^{\bT}_m=c_j$ and $v^{\bT}_m=v_j$.
Let $u_0$ be the extension of $c^{\bT}_m$ in $\bU$
which contains $\{x<v_j, xE v_j\}$.
Extend $s_1$ to a node $t_1\in \bU\re |c^{\bT}_m|$, and let $u_1$ be the immediate successor of $t_1$ in $\bU$.
Extend  all other nodes in $\bT\re |c^{\bT}_{m-1}|$ (besides $t_*$) to a node in $\bU$ of length  $|c^{\bT}_m|$, and let $\bT\re |c^{\bT}_m|$ consist of these nodes along with $t_0$ and $t_1$.
Let $\bT\re (|c^{\bT}_m|+1)$ consist of $u_0, u_1$, and one immediate successor of each of the nodes in $\bT\re |c^{\bT}_m|$.
By the transitivity of both relations $<$ and $E$,
we obtain that the $\prec$-preserving  bijection between
 $\bU(m+1)$ and $\bT\re (|c^{\bT}_m|+1)$
preserves passing types over $\bQ_{\bQ}\re \{v^{\bT}_k:k\le m\}$.

If the coding node $c_m$ has four immediate successors,
then these extensions consist  of all  choices from among $\{x<v_m,v_m<x\}$ and $\{xE v_m,x \hskip-.05in \not \hskip-.05in E v_m\}$.
We start on the outside with non-equivalence and work our way inside to equivalence.
First extend $t_*$ to a coding node $c_i\in \bU$ which has
four immediate successors, and let
$s_0$ denote the extension with
$\{x<v_i,x \hskip-.05in \not \hskip-.06in E v_i\}$
and
$s_3$ denote the extension with
$\{v_i<x,x \hskip-.05in \not \hskip-.04in E v_i\}$.
Again, extend $s_0$ to a coding node $c_j\in \bS$
 which has
four immediate successors, and let
$s_0$ denote the extension with
$\{x<v_j,x \hskip-.06in \not \hskip-.04in E v_j\}$
and
$s_1$ denote the extension with
$\{v_j<x,x \hskip-.05in \not \hskip-.04in E v_j\}$.
Then extend $s_1$ to any coding node $c_k$.
Take $c_\ell$ to be a coding node extending $c_k\cup\{x<v_k, x E v_k\}$,
and define $c^{\bT}_m=c_{\ell}$ and $v^{\bT}_m=v_{\ell}$.
Let $t_0$ be the  $\prec$-leftmost extension of $s_0$
in $\bU(\ell)$,
let $t_2$ be the $\prec$-leftmost extension of
$c_k\cup\{v_k<x, x E v_k\}$ in $\bU(\ell)$,
and let $t_3$ be the $\prec$-leftmost extension of $s_3$ in $\bU(\ell)$.
Finally, define $\bT\re |c^{\bT}_m|$ to consist of
$\{t_0,
c^{\bT}_m, t_2,t_3\}$ along with
the leftmost extensions in $\bU(\ell)$ of the nodes in
$(\bT\re |c^{\bT}_{m-1}|)\setminus\{t_*\}$.
Let the nodes in $\bT\re |c^{\bT}_{m}+1|$
consist of $c^{\bT}_{m}\cup\{x<v^{\bT}_m,xEv^{\bT}_m\}$, along with the immediate successors  in $\bU(\ell+1)$
of
the rest of the nodes in $\bT\re |c^{\bT}_{m}|$
It is routine to check that, by transitivity of the relations $<$ and $E$,
the immediate successors

The idea for general $(\bQ_{\bQ})_n$ is similar.
Here we have a sequence of convex equivalence relations $\lgl E_i:i<n\rgl$, where for each $i<n-1$,
$E_{i+1}$ coarsens $E_i$.
Similarly to the above,  each coding node  $c_m$ has
$2(j+1)$ many immediate successors, for some $j\le n$.
The  immediate successors
run through all combinations of choices from among $\{x<v_m,v_m<x\}$
and $\{x  E_0 v_m\}\cup
\{(xE_{i+1} v_m\wedge
x \hskip-.05in \not \hskip-.04in E_{i} v_m) : i< j\}$.
When constructing skew splitting,
 in order to set up so that the desired passing types  are available  at the next coding node of $\bT$,
we start on the ``outside'' with  types containing
$(xE_{j} v_m\wedge
x \hskip-.05in \not \hskip-.04in E_{j-1} v_m)$
and work our way inward, with the increasingly finer equivalence relations,
analogously to how the case of four immediate successors was handled above for $\bQ_{\bQ}$.

The presence of any unary relations has no effect on the existence of diagonal coding trees.
\end{proof}


\section{Forcing exact upper bounds for big Ramsey degrees}\label{sec.FRT}

This section contains the Ramsey theorem
for colorings of copies of a given
finite substructure of a \Fraisse\ structure  satisfying
\EEAP$^+$.
Theorem
\ref{thm.onecolorpertype}
provides
upper bounds for the big Ramsey degrees
of such structures
when the language has relation symbols of arity at most two, and these
turn out to be exact.
The proof of exactness  will be given   in Section \ref{sec.brd}.

The key combinatorial content of Theorem
\ref{thm.onecolorpertype}
occurs in Theorem \ref{thm.matrixHL},
where
 we  use the technique of forcing to  essentially conduct an unbounded search for a finite object,
 achieving within ZFC
one color per level set extension of a given finite
tree.
It is important to
note
that we never actually go to a generic extension.
In fact,  the forced generic object  is very much {\em not} a coding tree.
Rather, we use the forcing to
do two things:
(1)
Find a good set of nodes from which we can start to build a subtree which can have the  desired homogeneity properties; and
(2)
Use the forcing to guarantee
 the existence of  a finite object with certain properties.
 Once found, this object, being finite, must exist in the ground model.

We take here a sort of amalgamation of techniques  developed in  \cite{DobrinenJML20}, \cite{DobrinenH_k19}, and \cite{DobrinenRado19}, making adjustments as necessary.
The main differences from previous work are the following:
The forcing poset is on trees of
 $1$-types; as such,
  we work
with the  general notion of  passing type, in place of  passing number used in the papers \cite{DobrinenRado19}, \cite{DobrinenH_k19}, \cite{DobrinenJML20}, and \cite{Zucker20} for  binary relational structures.
Moreover, Definition \ref{def.plussim} presents a stronger requirement  than just similarity.
This
addresses
both the fact that relations can be of any arity,
and the fact that we consider \Fraisse\  classes
which
have disjoint, but not necessarily free, amalgamation.

We now set up notation, definitions,  and assumptions for
Theorem \ref{thm.matrixHL}, beginning with the following convention.
We also define the Extension Property, which is one of the conditions for \EEAP$^+$ to hold.

\begin{convention}\label{conv.Gamma_ts}
Let $\mathcal{K}$ be a \Fraisse\ class in a language $\mathcal{L}$ and $\bK$ a \Fraisse\ limit of $\mathcal{K}$.
If
(a)
$\mathcal{K}$  satisfies \SFAP,
or
(b)
$\bK$
 satisfies \EEAP$^+$ and either
 has no unary relations or has no transitive relations, then
we  work inside a diagonal coding subtree $\bT$ of
$\bS$.
Otherwise, we work inside a diagonal coding subtree $\bT$ of $\bU$.
 \end{convention}

\begin{rem}
All proofs in this section could be done working inside $\bU$.
Then in the case
when $\mathcal{L}$ has unary relation symbols and $\bK$ has no transitive relations,
in order to obtain the optimal upper bounds,
one would need to  take a diagonal  antichain of coding nodes, $\bD$,  in
Lemma \ref{lem.bD}, which has
the following properties:
There exists a level $\ell$ such that
$\bD\re\ell$ has $|\Gamma|$-many nodes, labeled
$d_\gamma$ where $\gamma\in\Gamma$,
such that
for each coding  node $c^{\bD}_n$ in $\bD$ there is exactly one $\gamma\in\Gamma$, call it $\gamma^\ast$, such that $c^{\bD}_n$ extends
$d_{\gamma^\ast}$.  Further, for the vertex $v^{\bD}_n$ represented by $c^{\bD}_n$, $\gamma^\ast(v^{\bD}_n)$ holds in $\bK$.
The end result of this approach is equivalent to working in $\bS$.

The results in this section
 could also be attained working solely in $\bS$.
 However, in the case
  when
  $\mathcal{L}$ has unary relation symbols and $\bK$ has a transitive relation,
 $\bS$ will not contain  a diagonal coding subtree, so the proofs would have to be modified to allow for more than one splitting node of a given length (see for instance,  Figure \ref{fig.Q2tree} showing $\bS(\bQ_2)$).
Convention \ref{conv.Gamma_ts} is intended to give the reader the idea  of when each approach (working in $\bS$ or working in $\bU$)
is most natural.

We point out that  if
$\mathcal{L}$
has no unary
relation symbols,
then $\bS=\bU$.
  \end{rem}

Let $T$ be a diagonal coding tree for the \Fraisse\ limit $\bK$ of some \Fraisse\
class $\mathcal{K}$.
Recall that the tree ordering on $T$ is simply inclusion.
We recapitulate notation from Subsection \ref{subsec.3.1}:
Each $t \in T$ can be thought of as a sequence $\lgl t(i) : i < |t| \rgl$ where $t(i) = (t \re \bK_i) \!\setminus\! (t \re \bK_{i-1})$.
For
 $t\in T$ and $\ell\le |t|$,
$t\re \ell$ denotes
$\bigcup_{i<\ell}t(i)$, which we can
think of
as the sequence $\lgl t(i):i< \ell\rgl$,
the initial segment of $t$ with domain $\ell$.
Note that
$t\re\ell\in \bS(\ell-1)$
(or $t\re \ell\in \bU(\ell-1)$).
(We let $\bS(-1)=\bU(-1)$ denote the set containing the empty set, just so that we do not have to always write $\ell\ge 1$.)

The following extends Notation
\ref{notn.cong<}
to
subsets of trees.
For a  finite subset $A\sse\bT$,  let
\begin{equation}
 \ell_A=\max\{|t|:t\in A\}\mathrm{\ \ and\ \ } \max(A)=
\{s\in A: |s|=\ell_A\}.
 \end{equation}
For $\ell\le \ell_A$,
let
\begin{equation}
A\re \ell=\{t\re \ell : t\in A\mathrm{\ and\ }|t|\ge \ell\}
\end{equation}
 and let
\begin{equation}
A\rl \ell=\{t\in A:|t|< \ell\}\cup A\re \ell.
\end{equation}
Thus, $A\re \ell$ is a level set, while $A\rl \ell$ is the set of nodes in $A$ with length less than $\ell$ along with the truncation
to $\ell$ of the  nodes in $A$ of length at least
 $\ell$.
Notice that
$A\re \ell=\emptyset$ for $\ell>\ell_A$, and
 $A\rl \ell=A$  for  $\ell\ge \ell_A$.
 Given $A,B\sse T$, we say that $B$ is an {\em initial segment} of $A$ if   $B=A\rl \ell$
 for some $\ell$ equal to
   the length of some node in $A$.
   In this case, we also say that
   $A$ {\em end-extends} (or just {\em extends}) $B$.
If $\ell$ is not the length of any node in $A$, then
  $A\rl \ell$ is not a subset  of $A$, but  is  a subset of $\widehat{A}$, where
  $\widehat{A}$ denotes $\{t\re n:t\in A\mathrm{\ and\ } n\le |t|\}$.

Define $\max(A)^+$ to be the set of nodes
$t$ in $T\re (\ell_A+1)$ such that $t$ extends $s$ for some $s \in \max(A)$.
Given a node $t\in T$ at the level of a coding node in $T$, $t$ has exactly one immediate successor in $\widehat{T}$, which  we recall
from Notation \ref{notn.cong<}
is  denoted as
$t^+$.

\begin{defn}[$+$-Similarity]\label{def.plussim}
Let $T$ be a diagonal coding tree for
the \Fraisse\ limit $\bK$ of
a \Fraisse\
class $\mathcal{K}$, and
suppose $A$ and $B$
are finite subtrees of $T$.
We write  $A\plussim B$ and say that
 $A$ and $B$ are
 {\em $+$-similar} if and only if
  $A\sim B$ and
 one of the following two cases holds:
 \begin{enumerate}
 \item[]
   \begin{enumerate}
\item[\bf Case 1.]
 If $\max(A)$  has a splitting node in $T$,
 then so does $\max(B)$,
  and the similarity map from $A$ to $B$  takes the splitting node in $\max(A)$ to the splitting node in $\max(B)$.
    \end{enumerate}
     \end{enumerate}
       \begin{enumerate}
    \item[]
        \begin{enumerate}
  \item[\bf Case 2.]
If $\max(A)$ has a coding node, say
 $c^A_n$,
and  $f:A\ra B$ is  the similarity map,
then
  $s^+(n;A)\sim f(s)^+(n;B)$ for each $s\in \max(A)$.
 \end{enumerate}
    \end{enumerate}

Note that $\plussim$ is an  equivalence relation, and  $A\plussim B$ implies $A\sim B$.
When $A\sim B$ ($A\plussim B$), we say that they have the same {\em similarity type} ({\em $+$-similarity type}).
\end{defn}

\begin{rem}\label{rem.SplussimT}
For infinite trees $S$ and $T$ with no terminal  nodes, $S\sim T$ implies that for
each $n$, letting $d^S_n$ and $d^T_n$ denote the $n$-th critical nodes of $S$ and $T$, respectively,
$S\re|d_n^S|\plussim T\re|d_n^T|$.
\end{rem}

We adopt the following notation  from topological Ramsey space theory (see \cite{TodorcevicBK10}).
Given  $k<\om$,
we define
$r_k(T)$ to be  the
restriction of $T$ to the levels of the first $k$ critical nodes of $T$;
that is,
\begin{equation}
r_k(T)=\bigcup_{m<k}T(m),
\end{equation}
where $T(m)$ denotes the set of all nodes in $T$ with length equal to $|d^T_m|$.
It follows from Remark \ref{rem.SplussimT} that  for any
 $S,T\in \mathcal{T}$,
$r_k(S)\plussim r_k(T)$.
Define $\mathcal{AT}_k$ to be the set of {\em $k$-th approximations}
to members of $\mathcal{T}$;
that is,
\begin{equation}
\mathcal{AT}_k=\{r_k(T):T\in\mathcal{T}\}.
\end{equation}
For $D\in\mathcal{AT}_k$
and $T\in\mathcal{T}$,
define the set
\begin{equation}
[D,T]=\{S\in \mathcal{T}:r_k(S)=D\mathrm{\ and\ } S\le T\}.
\end{equation}
Lastly, given
 $T\in\mathcal{T}$,
 $D=r_k(T)$,  and $n>k$,
  define
\begin{equation}
r_n[D,T]=\{r_n(S):S\in [D,T]\}.
\end{equation}

More generally, given any $A\sse T$, we use $r_k(A)$ to denote the first $k$ levels of the tree induced by the meet-closure of $A$.
We now have the necessary ideas to define the Extension Property.

Recall from Convention \ref{conv.Gamma_ts} that $\bT$ is a fixed diagonal coding tree (in $\bS$ or in $\bU$) for an enumerated \Fraisse\ limit $\bK$ of
a \Fraisse\ class $\mathcal{K}$.

\begin{defn}[Extension Property]\label{defn.ExtProp}
We say that
$\bK$
has the {\em Extension Property} when
either (1) or   (2) holds:
\begin{enumerate}
\item
Suppose $A$ is a finite or infinite subtree of   some
$T\in\mathcal{T}$.
Let $k$ be given  and suppose
$\max(r_{k+1}(A))$ has a splitting node.
Suppose that $B$ is a $+$-similarity copy of $r_k(A)$ in $T$.
Let  $u$ denote the splitting node in $\max(r_{k+1}(A))$,
and let
 $s$ denote  the node in $\max(B)^+$ which must be extended to a splitting node in order to obtain a $+$-similarity copy of $r_{k+1}(A)$.
If $s^*$ is a splitting node  in $T$  extending $s$,
then  there are extensions of the rest of the nodes in $\max(B)^+$ to the same length as $s^*$ resulting in a $+$-similarity copy
of $r_{k+1}(A)$ which
can be extended to a copy of $A$.
\item
There
is some $2\le q<\om$,
and a function
 $\psi$ defined on the set of splitting nodes in $\bT$ and  having range $q$,
 such that  the following holds:
 \begin{enumerate}
 \item[(a)]
Suppose $A$ is a finite or infinite subtree of    some $T\in\mathcal{T}$.
Let $k$ be given  and suppose
$\max(r_{k+1}(A))$ has a splitting node.
Suppose that $B$ is a $+$-similarity copy of $r_k(A)$ in $T$ such that  the similarity map
$f:r_k(A)\ra B$
has the property that
for each splitting node $t\in r_k(A)$,
 $\psi(t)=\psi(f(t))$.
Let  $u$ denote the splitting node in $\max(r_{k+1}(A))$,
and let
 $s$ denote  the node in $\max(B)^+$ which must be extended to a splitting node in order to obtain a $+$-similarity copy of $r_{k+1}(A)$.
 Then  for each  $s'\contains s$ in $T$,
 there exists  a splitting node $s^*\in T$ extending $s'$ such that $\psi(s^*)=\psi(u)$.
 Moreover,
 given such an $s^*$,
 there are extensions of the rest of the nodes in $\max(B)^+$ to the same length as $s^*$ resulting in a $+$-similarity copy of $r_{k+1}(A)$ which  can be extended to a copy of $A$.
\item[(b)]
The language for $\bK$ has at least one binary relation symbol (besides equality), and the
value  of $\psi$ is determined  by some partition of all  pairs   of
 partial $1$-types  involving only binary
 relation symbols
 over a one-element structure into pieces
 $Q_0,\dots,Q_{q-1}$,
such that
whenever $s$ is a splitting node in $\bT$,
$\psi(s)=m$ if and only if the following hold:
whenever  $c^{\bT}_j,c^{\bT}_k$ are coding nodes in $\bT$ with $c^{\bT}_j\wedge c^{\bT}_k=s$,
then
the pair of partial $1$-types of $v^{\bT}_j$ and  $v^{\bT}_k$ over
$\bK\re\{v_i\}$  is in $Q_m$.
\end{enumerate}
\end{enumerate}
\end{defn}

\begin{rem}
The Extension Property  (1) easily holds  for \Fraisse\ limits of all \Fraisse\ classes satisfying \SFAP, as we show below in Lemma \ref{lem.SFAPEP};
and similarly  for their ordered
expansions.
The same is true for   \Fraisse\ limits of all unrestricted  \Fraisse\ classes and their ordered expansions.
In these cases, all splitting nodes   in $\bT$
allow for the
construction of a $+$-similarity copy of $A$.
\Fraisse\ limits of classes such as $\mathcal{P}_n$, as well as the
four non-trivial
reducts of the rationals, also trivially have the Extension Property.

The convexly ordered equivalence relations
$\mathcal{COE}_n$, as well as the more general classes $\mathcal{COE}_{n,p}$,
satisfy (2) of the Extension property with
$q=n+1$.
See Lemma \ref{lem.EPCOE} for a discussion.
Part (2a) of the Extension property is sufficient for the proof of Theorem
\ref{thm.matrixHL},
and (2b) is sufficent for the proof of Theorem
\ref{thm.persistence}.
\end{rem}

\begin{lem}\label{lem.SFAPEP}
\SFAP\ implies the Extension Property.
\end{lem}

\begin{proof}
We will actually prove a slightly stronger statement which implies
(1) of
 the Extension Property.
Let $A$ be a subtree of some $T\in\mathcal{T}$.
Without loss of generality,  we may assume that either $A$ is infinite  and has infinitely many coding nodes, or else $A$ is finite and the  node in $A$ of maximal length is a coding node.
Let $m$ either be $0$,
 or else  let $m$ be a positive integer such that $\max(r_m(A))$ has a coding node.
 Let $n>m$ be least above $m$ such that $\max(r_n(A))$ has a coding node; let $c^A_i$ denote this coding node.

 Now suppose that $B$ is a $+$-similarity copy of $r_m(A)$, and
suppose  $C$ is an extension of $B$ in $T$ such that $C$ is $+$-similar to $r_{n-1}(A)$.
(Such a $C$ is easy to construct since $\bS$ is a perfect tree whenever
$\bK$ has at least one non-trivial relation of arity greater than one.)
Let $X$ denote $\max(r_{n-1}(A))^+$,
let $Y$ denote $\max(C)^+$, and let
$\varphi$ be the $+$-similarity map from $X$ to $Y$.
Let $t$ denote the node in $X$ which extends to the coding node in $\max(r_n(A))$, and let
$y$  denote  $\varphi(t)$.
Extend $y$ to some coding node  $c^T_{i'}$ in $T$ such that the substructure  of $\bK$
represented by the coding nodes in $B$ along with $c^T_{i'}$ is isomorphic to the substructure of $\bK$ represented by the coding nodes in $r_n(A)$.

Fix any $u\in X$ such that $u\ne t$, and
let $z$ denote $\varphi(u)$.
Let $c^T_j$ denote the least coding node in $A$ extending $u$.
By \SFAP, there
is an extension  of $z$ to some coding node $c^T_{j'}$ representing a vertex $w'$ in $\bK$ such that
the substructure of $\bK$ represented by the coding nodes in $B$ along with $c^T_{i'}$ and $c^T_{j'}$
is isomorphic to  the substructure of $\bK$ represented by the coding nodes in $r_n(A)$ along with $c^T_j$.
Let $u'$ denote the unique extension of $u$ in $\max(r_n(A))$, and let $z'$ denote the
 truncation of $c^T_{j'}$ to the length $|c^T_{i'}|+1$.
 Then $(z')^+(c^T_{j'};B)\sim (u')^+(c^T_i;r_m(A))$.
  Therefore, the union of $C$ along with
 $\{u':u\in Y\setminus\{y\}\}\cup \{c^T_{i'}\}$ is $+$-similar to $r_n(A)$.
 It follows that the Extension Property holds.
\end{proof}

\begin{lem}\label{lem.EPCOE}
The
\Fraisse\ limits of the
\Fraisse\ classes
$\mathcal{COE}_{n,p}$
satisfy
the Extension Property.
\end{lem}

\begin{proof}
Fix
$n, p$, with at least one  of $n,p$ greater than one,
 and denote
the \Fraisse\ limit of $\mathcal{COE}_{n,p}$ as $\bK$.
Recall
that the language of
$\mathcal{COE}_{n,p}$
has one binary relation symbol $<$, interpreted
in $\bK$
as a linear order;
 $p$-many unary relation symbols, $P_0,\dots, P_{p-1}$,
whose interpretations partition $\bK$ into $p$-many dense pieces;
and $n$-many binary relation symbols  $E_0,\dots,E_{n-1}$, each  interpreted
in $\bK$
as an equivalence relation
 with convexly ordered equivalence classes
 such that if $i<j<n$, then
 $E^\bK_j$ coarsens $E^\bK_i$.
 Let $v^{\bT}_m$ denote the vertex in $\bK$ represented by the $m$-th coding node, $c^{\bT}_m$, in $\bT$.
Given a splitting node $s$ in $\bT$,
define $\psi(s)=n$ if
$\neg E_{n-1}(x, v^{\bT}_m)$ is in  $s$, for all $m<|s|-1$.
Otherwise,
define $\psi(s)$ to be the least $i<n$ such that
$E_{i}(x, v^{\bT}_m)$ is in  $s$, for some $m<|s|-1$.
 In this second case, any coding node in $\bT$ extending $s$ represents a vertex which is in the same $E_{\psi(s)}$-equivalence class as some vertex $v^{\bT}_m$ for some $m<|s|-1$.
 Note that
  if $s\sse t$ are splitting nodes in $\bT$,
 then  $\psi(s) \ge \psi(t)$.
It is routine to check that
$\mathcal{COE}_{n,p}$
satisfies
(2) of the Extension Property
  with this $\psi$.
\end{proof}

By an {\em antichain} of coding nodes, we mean a set of  coding nodes  which is pairwise  incomparable with respect to  the tree partial order of inclusion.
\vskip.1in

\noindent\bf{Set-up for Theorem \ref{thm.matrixHL}.} \rm
Let $T$ be a diagonal coding tree  in $\mathcal{T}$.
Fix a finite antichain  of coding nodes  $\tilde{C}\sse T$.
We abuse notation and also write  $\tilde{C}$ to denote  the tree that its meet-closure induces in $T$.
Let $\tilde{A}$ be a fixed proper initial segment of
$\tilde{C}$, allowing for $\tilde{A}$ to be the empty set.
Thus, $\tilde{A}= \tilde{C}\rl \ell$, where $\ell$
is the length of some splitting or  coding node in
$\tilde{C}$
(let $\ell=0$ if $\tilde{A}$ is empty).
Let $\ell_{\tilde{A}}$ denote this $\ell$, and note that
any non-empty
 $\max(\tilde{A})$ either  has a coding node or a splitting node.
Let $\tilde{x}$ denote the shortest splitting or coding node in $\tilde{C}$ with length greater than $\ell_{\tilde{A}}$,
and define  $\tilde{X}=\tilde{C}\re |\tilde{x}|$.
Then $\tilde{A}\cup\tilde{X}$ is an initial segment of
 $\tilde{C}$; let $\ell_{\tilde{X}}$ denote $|\tilde{x}|$.
There are two cases:

\begin{enumerate}
\item[]
\begin{enumerate}
\item[\bf Case (a).]
$\tilde{X}$ has a splitting node.
\end{enumerate}
\end{enumerate}

\begin{enumerate}
\item[]
\begin{enumerate}
\item[\bf Case (b).]
$\tilde{X}$ has a coding node.
\end{enumerate}
\end{enumerate}

Let $d+1$ be the number of nodes in $\tilde{X}$ and  index these nodes as $\tilde{x}_i$, $i\le d$,
where $\tilde{x}_d$ denotes  the critical  node (recall that {\em critical node}  refers to  a splitting or  coding node).
Let
\begin{equation}\label{eq.tildeB}
\tilde{B}=\tilde{C} \re (\ell_{\tilde{A}}+1).
\end{equation}
Then
  $\tilde{X}$ is a level set   equal to or  end-extending  the level set $\tilde{B}$.
  For each $i\le d$, define
  \begin{equation}
  \tilde{b}_i=\tilde{x}_i\re \ell_{\tilde{B}}.
\end{equation}
Note that we
consider nodes in  $\tilde{B}$  as simply nodes to be extended; it
does not matter
whether the nodes in $\tilde{B}$ are coding, splitting, or neither in $T$.

\begin{defn}[Weak similarity]\label{defn.weaksim}
Given finite subtrees $S,T\in\mathcal{T}$ in which each coding node is terminal,
we say that  $S$ is {\em weakly similar} to $T$, and write $S\wsim T$,
if and only if
$S\setminus \max(S)\plussim T\setminus\max(T)$.
\end{defn}

\begin{defn}[$\Ext_T(B;\tilde{X})$]\label{defn.ExtBX}
Let $T\in\mathcal{T}$ be  fixed  and let $D=r_n(T)$ for some $n<\om$.
Suppose $A$ is a subtree of $D$ such that
$A \plussim \tilde{A}$ and
 $A$ is extendible to a similarity copy of $\tilde{C}$ in $T$.
Let $B$ be a subset of  the level set $\max(D)^+$ such that $B$ end-extends   (or equals)
$\max(A)^+$
and $A\cup B\wsim\tilde{A}\cup\tilde{B}$.
Let $X^*$ be a level set end-extending $B$ such that $A\cup X^*\plussim \tilde{A}\cup\tilde{X}$.  Let $U^*=T\rl (\ell_B-1)$.
Define
$\Ext_T(B;X^*)$ to be the collection of all level sets $X\sse T$
such that
\begin{enumerate}
\item
$X$ end-extends $B$;
\item
$U^*\cup X\plussim U^*\cup X^*$;
 and
\item
$A\cup X$ extends to a copy of $\tilde{C}$.
\end{enumerate}
\end{defn}

For Case (b), condition (3) follows from (2).
For Case (a),
the Extension Property
guarantees that
for any level set  $Y$ end-extending $B$,
there is a level set $X$ end-extending $Y$ such that $A\cup X$ satisfies
condition (3).
In both cases, condition (2) implies that $A\cup X\plussim \tilde{A}\cup\tilde{X}$.

The following theorem of \Erdos\  and Rado will  provide the pigeonhole principle for the forcing proof.

\begin{thm}[\Erdos-Rado, \cite{Erdos/Rado56}]\label{thm.ER}
For $r<\om$ and $\mu$ an infinite cardinal,
$$
\beth_r(\mu)^+\ra(\mu^+)_{\mu}^{r+1}.
$$
\end{thm}

We are now ready to prove the Ramsey theorem for level set extensions of a given finite tree.

\begin{thm}\label{thm.matrixHL}
Suppose that
$\mathcal{K}$
has \Fraisse\ limit $\bK$ satisfying
\EEAP$^+$, and
$T\in\mathcal{T}$ is given.
Let $\tilde{C}$ be a finite antichain of coding nodes in $T$,   $\tilde{A}$ be an initial segment of $\tilde{C}$,
and $\tilde{B}$ and $\tilde{X}$ be defined as above.
Suppose
 $D=r_n(T)$ for some $n<\om$,
 and
 $A\sse D$
and $B\sse \max(D^+)$
satisfy
$A\cup B\wsim \tilde{A}\cup\tilde{B}$.
Let  $X^*$ be a level set end-extending $B$ such that $A\cup X^*\plussim \tilde{A}\cup\tilde{X}$.
Then given any coloring
  $h:  \Ext_T(B;X^*)\ra 2$,
  there is a coding tree $S\in [D,T]$ such that
$h$ is monochromatic on $\Ext_S(B;X^*)$.
\end{thm}

\begin{proof}
Enumerate the nodes in $B$
as $s_0,\dots,s_d$ so that  for
any $X\in\Ext_T(B;X^*)$,
the critical node in $X$ extends $s_d$.
Let $M$ denote the collection of all  $m\ge n$  for which  there is a member of
 $\Ext_T(B;X^*)$ with  nodes  in $T(m)$.
 Note that this set $M$ is the same for any $S\in\mathcal{T}$, since  $S\sim T$ for all
 $S,T\in \mathcal{T}$.
 Let $L=\{|t|:\exists m\in M\, (t\in T(m))\}$,
 the collection of lengths of nodes in the levels $T(m)$ for $m\in M$.

For
  $i\le d$,   let  $T_i=\{t\in T:t\contains s_i\}$.
Let $\kappa$ be large enough, so that the partition relation $\kappa\ra (\aleph_1)^{2d}_{\aleph_0}$ holds.
The following forcing notion $\bP$    adds $\kappa$ many paths through  each  $T_i$,  $i< d$,
and one path  through $T_d$.
\vskip.1in

In both   Cases (a) and (b), define
$\bP$ to be  the set of finite partial functions $p$   such that
$$
p:(d\times\vec{\delta}_p)\cup\{d\}\ra T(m_p),
$$
where
\begin{enumerate}
\item
  $m_p\in M$ and $\vec{\delta}_p$ is a finite subset of $\kappa$;
 \item
 $\{p(i,\delta) : \delta\in  \vec{\delta}_p\}\sse  T_i(m_p)$ for each $i<d$;
 \item
 $p(d)$ is the critical node in $T_d(m_p)$; and
\item
 For any choices of $\delta_i\in\vec{\delta}_p$,
 the level set $\{p(i,\delta_i):i<d\}\cup\{p(d)\}$ is a member of $\Ext_T(B;X^*)$.
 \end{enumerate}
Given $p\in\bP$,
 the {\em range of $p$} is defined as
$$
\ran(p)=\{p(i,\delta):(i,\delta)\in d\times \vec\delta_p\}\cup \{p(d)\}.
$$
Let $\ell_p$ denote the length of the nodes in $\ran(p)$.
If also $q\in \bP$ and $\vec{\delta}_p\sse \vec{\delta}_q$, then we let $\ran(q\re \vec{\delta}_p)$ denote
$\{q(i,\delta):(i,\delta)\in d\times \vec{\delta}_p\}\cup \{q(d)\}$.

In Case (a),
the partial
 ordering on $\bP$  is  simply reverse inclusion:
$q\le p$ if and only if
\begin{enumerate}
\item
$m_q\ge m_p$, $\vec{\delta}_q\contains \vec{\delta}_p$,
$q(d)\contains p(d)$;
and
\item
$q(i,\delta)\contains p(i,\delta)$ for each $(i,\delta)\in d\times \vec{\delta}_p$.
\end{enumerate}
In Case (b), we define
$q\le p$  if and only if (1) and (2) hold and additionally, the following third requirement holds:
\begin{enumerate}
\item[(3)]
Letting $U=T\re (\ell_p-1)$,
$U\cup\ran(p)\plussim
U\cup \ran(q\re\vec{\delta}_p)$.
\end{enumerate}
(Requirement (3)  is stronger than that which was used for the Rado graph in \cite{DobrinenRado19}, because
for relations of  arity three or more, the extension $q$ must preserve information about $1$-types over the fixed finite structure which we wish to extend.)
Then   $(\bP,\le)$ is a separative, atomless partial order.

The next part of the proof (up to and including Lemma \ref{lem.compat}) follows that of  \cite{DobrinenRado19} almost verbatim.
The key difference between the work here and in \cite{DobrinenRado19} is that here, information which makes the proof work for relations of any arity is
embedded  in the definition of $\Ext_T(B,X^*)$.
For  $(i,\al)\in d\times \kappa$,
 let
\begin{equation}
 \dot{b}_{i,\al}=\{\lgl p(i,\al),p\rgl:p\in \bP\mathrm{\  and \ }\al\in\vec{\delta}_p\},
\end{equation}
 a $\bP$-name for the $\al$-th generic branch through $T_i$.
Let
\begin{equation}
\dot{b}_d=\{\lgl p(d),p\rgl:p\in\bP\},
\end{equation}
a $\bP$-name for the generic branch through $T_d$.
 Given a generic filter   $G\sse \bP$, notice that
 $\dot{b}_d^G=\{p(d):p\in G\}$,
 which  is a cofinal path of  critical   nodes in $T_d$.
Let $\dot{L}_d$ be a $\bP$-name for the set of lengths of critical   nodes in $\dot{b}_d$, and note that
$\bP$ forces  that $\dot{L}_d\sse L$.
Let $\dot{\mathcal{U}}$ be a $\bP$-name for a non-principal ultrafilter on $\dot{L}_d$.
Given  $p\in \bP$, recall that  $\ell_p$ denotes the lengths of the nodes in $\ran(p)$,
and  notice  that
\begin{equation}
 p\forces  \forall
 (i,\al)\in d\times \vec\delta_p\, (\dot{b}_{i,\al}\re \ell_p= p(i,\al)) \wedge
( \dot{b}_d\re \ell_p=p(d)).
\end{equation}

We will write sets $\{\al_i:i< d\}$ in $[\kappa]^d$ as vectors $\vec{\al}=\lgl \al_0,\dots,\al_{d-1}\rgl$ in strictly increasing order.
For $\vec{\al}\in[\kappa]^d$,  let
\begin{equation}
\dot{b}_{\vec{\al}}=
\lgl \dot{b}_{0,\al_0},\dots, \dot{b}_{d-1,\al_{d-1}},\dot{b}_d\rgl.
\end{equation}
For  $\ell<\om$,
 let
 \begin{equation}
 \dot{b}_{\vec\al}\re \ell=
 \lgl \dot{b}_{0,\al_0}\re \ell,\dots, \dot{b}_{d-1,\al_{d-1}}\re \ell,\dot{b}_d\re \ell\rgl.
 \end{equation}
 One sees that
$h$  is a coloring on level sets  of the form $\dot{b}_{\vec\al}\re \ell$
whenever this is
 forced to be a member of $\Ext_T(B; X^*)$.
Given $\vec{\al}\in [\kappa]^d$ and  $p\in \bP$ with $\vec\al\sse\vec{\delta}_p$,
let
\begin{equation}
X(p,\vec{\al})=\{p(i,\al_i):i<d\}\cup\{p(d)\},
\end{equation}
recalling  that this level set
$X(p,\vec{\al})$
 is a member of
  $\Ext_T(B;X^*)$.

For each $\vec\al\in[\kappa]^d$,
choose a condition $p_{\vec{\al}}\in\bP$ satisfying the following:
\begin{enumerate}
\item
 $\vec{\al}\sse\vec{\delta}_{p_{\vec\al}}$.
\item
There is an $\varepsilon_{\vec{\al}}\in 2$
 such that
$p_{\vec{\al}}\forces$
``$h(\dot{b}_{\vec{\al}}\re \ell)=\varepsilon_{\vec{\al}}$
for $\dot{\mathcal{U}}$ many $\ell$ in $\dot{L}_d$''.
\item
$h(X(p_{\vec\al},\vec{\al}))=\varepsilon_{\vec{\al}}$.
\end{enumerate}
Such conditions can be found as follows:
Fix some $X\in \Ext_T(B;X^*)$ and let $t_i$ denote the node in $X$ extending $s_i$, for each $i\le d$.
For $\vec{\al}\in[\kappa]^d$,  define
$$
p^0_{\vec{\al}}=\{\lgl (i,\delta), t_i\rgl: i< d, \ \delta\in\vec{\al} \}\cup\{\lgl d,t_d\rgl\}.
$$
Then
 (1) will  hold for all $p\le p^0_{\vec{\al}}$,
 since $\vec\delta_{p_{\vec\al}^0}= \vec\al$.
 Next,
let  $p^1_{\vec{\al}}$ be a condition below  $p^0_{\vec{\al}}$ which
forces  $h(\dot{b}_{\vec{\al}}\re \ell)$ to be the same value for
$\dot{\mathcal{U}}$  many  $\ell\in \dot{L}_d$.
Extend this to some condition
 $p^2_{\vec{\al}}\le p_{\vec{\al}}^1$
 which
 decides a value $\varepsilon_{\vec{\al}}\in 2$
 so that
 $p^2_{\vec{\al}}$ forces
  $h(\dot{b}_{\vec{\al}}\re \ell)=\varepsilon_{\vec{\al}}$
for $\dot{\mathcal{U}}$ many $\ell$ in $\dot{L}_d$.
Then
 (2) holds for all $p\le p_{\vec\al}^2$.
If $ p_{\vec\al}^2$ satisfies (3), then let $p_{\vec\al}=p_{\vec\al}^2$.
Otherwise,
take  some  $p^3_{\vec\al}\le p^2_{\vec\al}$  which forces
$\dot{b}_{\vec\al}\re \ell\in \Ext_T(B;X^*)$ and
$h'(\dot{b}_{\vec\al}\re \ell)=\varepsilon_{\vec\al}$
for
some $\ell\in\dot{L}$
 with
$\ell_{p^2_{\vec\al}}< \ell\le \ell_{p^3_{\vec\al}}$.
Since $p^3_{\vec\al}$  forces  that $\dot{b}_{\vec\al}\re \ell$ equals
$\{p^3_{\vec\al}(i,\al_i)\re \ell:i<d\}\cup \{p^3_{\vec\al}(d)\re \ell\}$,
which is exactly
$X(p^3_{\vec\al}\re \ell,\vec\al)$,
and this level set is in the ground model, it follows that $h(X(p^3_{\vec\al}\re \ell,\vec\al))
=\varepsilon_{\vec\al}$.
Let
$p_{\vec\al}$ be $p^3_{\vec\al}\re \ell$.
Then $p_{\vec\al}$ satisfies (1)--(3).

Let $\mathcal{I}$ denote the collection of all functions $\iota: 2d\ra 2d$ such that
for each $i<d$,
$\{\iota(2i),\iota(2i+1)\}\sse \{2i,2i+1\}$.
For $\vec{\theta}=\lgl \theta_0,\dots,\theta_{2d-1}\rgl\in[\kappa]^{2d}$,
$\iota(\vec{\theta}\,)$ determines the pair of sequences of ordinals $\lgl \iota_e(\vec{\theta}\,),\iota_o(\vec{\theta}\,)\rgl$, where
\begin{align}
\iota_e(\vec{\theta}\,)&=
\lgl \theta_{\iota(0)},\theta_{\iota(2)},\dots,\theta_{\iota(2d-2))}\rgl\cr
\iota_o(\vec{\theta}\,)&=
 \lgl\theta_{\iota(1)},\theta_{\iota(3)},\dots,\theta_{\iota(2d-1)}\rgl.
 \end{align}

We now proceed to  define a coloring  $f$ on
$[\kappa]^{2d}$ into countably many colors.
Let $\vec{\delta}_{\vec\al}$ denote $\vec\delta_{p_{\vec\al}}$,
 $k_{\vec{\al}}$ denote $|\vec{\delta}_{\vec\al}|$,
$\ell_{\vec{\al}}$ denote  $\ell_{p_{\vec\al}}$, and let $\lgl \delta_{\vec{\al}}(j):j<k_{\vec{\al}}\rgl$
denote the enumeration of $\vec{\delta}_{\vec\al}$
in increasing order.
Given  $\vec\theta\in[\kappa]^{2d}$ and
 $\iota\in\mathcal{I}$,   to reduce subscripts
 let
$\vec\al$ denote $\iota_e(\vec\theta\,)$ and $\vec\beta$ denote $\iota_o(\vec\theta\,)$, and
define
\begin{align}\label{eq.fiotatheta}
f(\iota,\vec\theta\,)= \,
&\lgl \iota, \varepsilon_{\vec{\al}}, k_{\vec{\al}}, p_{\vec{\al}}(d),
\lgl \lgl p_{\vec{\al}}(i,\delta_{\vec{\al}}(j)):j<k_{\vec{\al}}\rgl:i< d\rgl,\cr
& \lgl  \lgl i,j \rgl: i< d,\ j<k_{\vec{\al}},\ \mathrm{and\ } \delta_{\vec{\al}}(j)=\al_i \rgl,\cr
&\lgl \lgl j,k\rgl:j<k_{\vec{\al}},\ k<k_{\vec{\beta}},\ \delta_{\vec{\al}}(j)=\delta_{\vec{\beta}}(k)\rgl\rgl.
\end{align}
Fix some ordering of $\mathcal{I}$ and define
\begin{equation}
f(\vec{\theta}\,)=\lgl f(\iota,\vec\theta\,):\iota\in\mathcal{I}\rgl.
\end{equation}

By the \Erdos-Rado Theorem  \ref{thm.ER},  there is a subset $K\sse\kappa$ of cardinality $\aleph_1$
which is homogeneous for $f$.
Take $K'\sse K$ so that between each two members of $K'$ there is a member of $K$.
Given  sets of ordinals $I$ and $J$,  we write $I<J$  to mean that  every member of $I$ is less than every member of $J$.
Take  $K_i\sse K'$  be  countably infinite subsets
satisfying
  $K_0<\dots<K_{d-1}$.

Fix some  $\vec\gamma\in \prod_{i<d}K_i$, and define
\begin{align}\label{eq.star}
&\varepsilon^*=\varepsilon_{\vec\gamma},\ \
k^*=k_{\vec\gamma},\ \
t_d=p_{\vec\gamma}(d),\cr
t_{i,j}&=p_{\vec{\gamma}}(i,\delta_{\vec{\gamma}}(j))\mathrm{\ for\ }
i<d,\ j<k^*.
\end{align}
We show that the values  in equation (\ref{eq.star}) are the same for any choice of
$\vec\gamma$.

\begin{lem}\label{lem.onetypes}
 For all $\vec{\al}\in \prod_{i<d}K_i$,
 $\varepsilon_{\vec{\al}}=\varepsilon^*$,
$k_{\vec\al}=k^*$,  $p_{\vec{\al}}(d)=t_d$, and
$\lgl p_{\vec\al}(i,\delta_{\vec\al}(j)):j<k_{\vec\al}\rgl
=
 \lgl t_{i,j}: j<k^*\rgl$ for each $i< d$.
\end{lem}

\begin{proof}
Let
 $\vec{\al}$ be any member of $\prod_{i<d}K_i$, and let $\vec{\gamma}$ be the set of ordinals fixed above.
Take  $\iota\in \mathcal{I}$
to be the identity function on $2d$.
Then
there are $\vec\theta,\vec\theta'\in [K]^{2d}$
such that
$\vec\al=\iota_e(\vec\theta\,)$ and $\vec\gamma=\iota_e(\vec\theta'\,)$.
Since $f(\iota,\vec\theta\,)=f(\iota,\vec\theta'\,)$,
it follows that $\varepsilon_{\vec\al}=\varepsilon_{\vec\gamma}$, $k_{\vec{\al}}=k_{\vec{\gamma}}$, $p_{\vec{\al}}(d)=p_{\vec{\gamma}}(d)$,
and $\lgl \lgl p_{\vec{\al}}(i,\delta_{\vec{\al}}(j)):j<k_{\vec{\al}}\rgl:i< d\rgl
=
\lgl \lgl p_{\vec{\gamma}}(i,\delta_{\vec{\gamma}}(j)):j<k_{\vec{\gamma}}\rgl:i< d\rgl$.
\end{proof}

Let $l^*$ denote the length of the  node $t_d$, and notice that
the  node
 $t_{i,j}$ also has length $l^*$,  for each   $(i,j)\in d\times k^*$.

\begin{lem}\label{lem.j=j'}
Given any $\vec\al,\vec\beta\in \prod_{i<d}K_i$,
if $j,k<k^*$ and $\delta_{\vec\al}(j)=\delta_{\vec\beta}(k)$,
 then $j=k$.
\end{lem}

\begin{proof}
Let $\vec\al,\vec\beta$ be members of $\prod_{i<d}K_i$   and suppose that
 $\delta_{\vec\al}(j)=\delta_{\vec\beta}(k)$ for some $j,k<k^*$.
For  $i<d$, let  $\rho_i$ be the relation from among $\{<,=,>\}$ such that
 $\al_i\,\rho_i\,\beta_i$.
Let   $\iota$ be the member of  $\mathcal{I}$  such that for each $\vec\theta\in[K]^{2d}$ and each $i<d$,
$\theta_{\iota(2i)}\ \rho_i \ \theta_{\iota(2i+1)}$.
Fix some
$\vec\theta\in[K']^{2d}$ such that
$\iota_e(\vec\theta)=\vec\al$ and $\iota_o(\vec\theta)= \vec\beta$.
Since between any two members of $K'$ there is a member of $K$, there is a
 $\vec\zeta\in[K]^{d}$ such that  for each $i< d$,
 $\al_i\,\rho_i\,\zeta_i$ and $\zeta_i\,\rho_i\, \beta_i$.
Let   $\vec\mu,\vec\nu$ be members of $[K]^{2d}$ such that $\iota_e(\vec\mu)=\vec\al$,
$\iota_o(\vec\mu)=\vec\zeta$,
$\iota_e(\vec\nu)=\vec\zeta$, and $\iota_o(\vec\nu)=\vec\beta$.
Since $\delta_{\vec\al}(j)=\delta_{\vec\beta}(k)$,
the pair $\lgl j,k\rgl$ is in the last sequence in  $f(\iota,\vec\theta)$.
Since $f(\iota,\vec\mu)=f(\iota,\vec\nu)=f(\iota,\vec\theta)$,
also $\lgl j,k\rgl$ is in the last  sequence in  $f(\iota,\vec\mu)$ and $f(\iota,\vec\nu)$.
It follows that $\delta_{\vec\al}(j)=\delta_{\vec\zeta}(k)$ and $\delta_{\vec\zeta}(j)=\delta_{\vec\beta}(k)$.
Hence, $\delta_{\vec\zeta}(j)=\delta_{\vec\zeta}(k)$,
and therefore $j$ must equal $k$.
\end{proof}

For each $\vec\al\in \prod_{i<d}K_i$, given any   $\iota\in\mathcal{I}$, there is a $\vec\theta\in[K]^{2d}$ such that $\vec\al=\iota_o(\vec\al)$.
By the second line of equation  (\ref{eq.fiotatheta}),
there is a strictly increasing sequence
$\lgl j_i:i< d\rgl$  of members of $k^*$ such that
$\delta_{\vec\gamma}(j_i)=\al_i$.
By
homogeneity of $f$,
this sequence $\lgl j_i:i< d\rgl$  is the same for all members of $\prod_{i<d}K_i$.
Then letting
 $t^*_i$ denote $t_{i,j_i}$,
 one sees that
\begin{equation}
p_{\vec\al}(i,\al_i)=p_{\vec{\al}}(i, \delta_{\vec\al}(j_i))=t_{i,j_i}=t^*_i.
\end{equation}
Let $t_d^*$ denote $t_d$.

\begin{lem}\label{lem.compat}
For any finite subset $\vec{J}\sse \prod_{i<d}K_i$,
$p_{\vec{J}}:=\bigcup\{p_{\vec{\al}}:\vec{\al}\in \vec{J}\,\}$
is a member of $\bP$ which is below each
$p_{\vec{\al}}$, $\vec\al\in\vec{J}$.
\end{lem}

\begin{proof}
Given  $\vec\al,\vec\beta\in \vec{J}$,
if
 $j,k<k^*$ and
 $\delta_{\vec\al}(j)=\delta_{\vec\beta}(k)$, then
 $j$ and $k$ must be equal, by
 Lemma  \ref{lem.j=j'}.
Then  Lemma \ref{lem.onetypes} implies
that for each $i<d$,
\begin{equation}
p_{\vec\al}(i,\delta_{\vec\al}(j))=t_{i,j}=p_{\vec\beta}(i,\delta_{\vec\beta}(j))
=p_{\vec\beta}(i,\delta_{\vec\beta}(k)).
\end{equation}
Hence,
 for all
$\delta\in\vec{\delta}_{\vec\al}\cap
\vec{\delta}_{\vec\beta}$
and  $i<d$,
$p_{\vec\al}(i,\delta)=p_{\vec\beta}(i,\delta)$.
Thus,
$p_{\vec{J}}:=
\bigcup \{p_{\vec{\al}}:\vec\al\in\vec{J}\}$
is a  function with domain $\vec\delta_{\vec{J}}\cup\{d\}$, where
$\vec\delta_{\vec{J}}=
\bigcup\{
\vec{\delta}_{\vec\al}:
\vec\al\in\vec{J}\,\}$; hence
, $p_{\vec{J}}$ is a member of $\bP$.
Since
for each $\vec\al\in\vec{J}$,
$\ran(p_{\vec{J}}\re \vec{\delta}_{\vec\al})=\ran(p_{\vec\al})$,
it follows that
$p_{\vec{J}}\le p_{\vec\al}$ for each $\vec\al\in\vec{J}$.
\end{proof}

This ends the material drawn directly from \cite{DobrinenRado19}.

We now proceed to  build a (diagonal coding)  tree $S\in [D,T]$ so that  the coloring
$h$ will be  monochromatic on $\Ext_S(B;X^*)$.
Recall that $n$ is the integer such that $D=r_n(T)$.
Let $\{ m_j:j<\om\}$ be the strictly increasing enumeration of $M$, noting that $m_0\ge n$.
For each $i\le d$,
extend the node $s_i\in B$ to the node $t^*_i$.
Extend each node $u$
 in $\max(D)^+\setminus B$  to some  node  $u^*$ in $T\re \ell^*$.
 If $X^*$ has a coding node and $m_0=n$,
 require also that
 $(u^*)^+(u^*;D)\sim u^+(u;D)$;
 \EEAP\ ensures that such $u^*$ exist.
 Set
\begin{equation}
U^*=\{t^*_i:i\le d\}\cup\{u^*:u\in \max(D)^+\setminus B\}
\end{equation}
and note that $U^*$
end-extends   $\max(D)^+$.

If $m_0=n$,
then
$D\cup U^*$ is a member of $r_{m_0+1}[D,T]$.
In
 this case, let $U_{m_0+1}=D\cup U^*$, and
 let $U_{m_1}$ be any member of $r_{m_1}[U_{m_0+1},T]$.
Note  that   $U^*$ is  the only member of $\Ext_{U_{m_1}}(B;X^*)$, and it has $h$-color $\varepsilon^*$.
Otherwise,
 $m_0>n$.
 In this case,
 take some  $U_{m_0}\in r_{m_0}[D,T]$ such that
 $\max(U_{m_0})$ end-extends $U^*$,
 and
 notice that   $\Ext_{U_{m_0}}(B;X^*)$ is empty.

Now assume that  $j<\om$ and
 we have constructed $U_{m_j}\in r_{m_j}[D,T]$
  so that every member of $\Ext_{U_{m_j}}(B;X^*)$
 has $h$-color $\varepsilon^*$.
Fix some  $V\in r_{m_j+1}[U_{m_j} ,T]$ and let $Y=\max(V)$.
We will extend the nodes in $Y$  to construct
$U_{m_j+1}\in r_{m_j+1}[U_{m_j},T]$
with the property that all members of $\Ext_{U_{m_j+1}}(B;X^*)$ have the same
 $h$-value $\varepsilon^*$.
This will be achieved  by constructing
 the condition $q\in\bP$, below, and then extending it to some condition $r \le q$ which decides that  all members of
 $\Ext_T(B;X^*)$ coming from the nodes in $\ran(r)$ have $h$-color $\varepsilon^*$.

Let $q(d)$ denote  the  splitting node or coding  node in $Y$ and let $\ell_q=|q(d)|$.
For each $i<d$,
let  $Y_i$ denote  $Y\cap T_i$.
For each $i<d$,
take  a set $J_i\sse K_i$ of  size card$(Y_i)$
and label the members of $Y_i$ as
$\{z_{\al}:\al\in J_i\}$.
Let $\vec{J}$ denote $\prod_{i<d}J_i$.
By   Lemma \ref{lem.compat},
the set $\{p_{\vec\al}:\vec\al\in\vec{J}\}$ is compatible,  and
$p_{\vec{J}}:=\bigcup\{p_{\vec\al}:\vec\al\in\vec{J}\}$ is a condition in $\bP$.

Let
 $\vec{\delta}_q=\bigcup\{\vec{\delta}_{\vec\al}:\vec\al\in \vec{J}\}$.
For $i<d$ and $\al\in J_i$,
define $q(i,\al)=z_{\al}$.
It follows  that for each
$\vec\al\in \vec{J}$ and $i<d$,
\begin{equation}
q(i,\al_i)\contains t^*_i=p_{\vec\al}(i,\al_i)=p_{\vec{J}}(i,\al_i),
\end{equation}
and
\begin{equation}
q(d)\contains t^*_d=p_{\vec\al}(d)=p_{\vec{J}}(d).
\end{equation}

For   $i<d$ and $\delta\in\vec{\delta}_q\setminus
J_i$,
we need to extend each node $p_{\vec{J}}(i,\delta)$ to some  node of length $\ell_q$ in order to construct a condition $q$ extending $p_{\vec{J}}$.
These nodes will not be a part of the construction of $U_{m_j+1}$, however; they only are only a technicality allowing us to find some $r\le q\le p_{\vec{J}}$ from which we will build $U_{m_j+1}$.
In Case (a),
let
$q(i,\delta)$ be  any  extension
 of $p_{\vec{J}}(i,\delta)$ in $T$ of length $\ell_q$.
 In Case (b),
 let $q(i,\delta)$ be  any  extension
 of $p_{\vec{J}}(i,\delta)$ in $T$ of length $\ell_q$  with
 \begin{equation}
 q(i,\delta)^+(q(d); T\re(\ell^*-1))
 \sim
 p_{\vec{J}}(i,\delta)^+(p_{\vec{J}}(d); T\re(\ell^*-1)).
 \end{equation}
 The \EEAP\ guarantees the existence of such $q(i,\delta)$.
Define
\begin{equation}
q=\{q(d)\}\cup \{\lgl (i,\delta),q(i,\delta)\rgl: i<d,\  \delta\in \vec{\delta}_q\}.
\end{equation}
This $q$ is a condition in $\bP$, and $q\le p_{\vec{J}}$.

Now take an $r\le q$ in  $\bP$ which  decides some $\ell_j$ in $\dot{L}_d$ for which   $h(\dot{b}_{\vec\al}\re \ell_j)=\varepsilon^*$, for all $\vec\al\in\vec{J}$.
This is possible since for all $\vec\al\in\vec{J}$,
$p_{\vec\al}$ forces $h(\dot{b}_{\vec\al}\re \ell)=\varepsilon^*$ for $\dot{\mathcal{U}}$ many $\ell\in \dot{L}_d$.
By the same argument as in creating the conditions $p_{\vec\al}$,
 we may assume that
 the nodes in the image of $r$ have length  $\ell_j$.
Since
$r$ forces $\dot{b}_{\vec{\al}}\re \ell_j=X(r,\vec\al)$
for each $\vec\al\in \vec{J}$,
and since the coloring $h$ is defined in the ground model,
it follows that
$h(X(r,\vec\al))=\varepsilon^*$ for each $\vec\al\in \vec{J}$.
Let
\begin{equation}
Y_0=\{q(d)\}\cup \{q(i,\al):i<d,\ \al\in J_i\},
\end{equation}
 and let
\begin{equation}
Z_0=\{r(d)\}\cup \{r(i,\al):i<d,\ \al\in J_i\}.
\end{equation}

Now we consider the two cases separately.
In Case (a),
let $Z$ be the level set
consisting of  the nodes in $Z_0$ along with
a  node  $z_y$ in $T\re \ell_j$ extending  $y$, for each
 $y\in
Y\setminus Y_0$.
Then $Z$ end-extends $Y$.
By \EEAP, it does not matter how the nodes  $z_y$ are chosen.
Letting $U_{m_j+1}=U_{m_j}\cup Z$,
we see that $U_{m_j+1}$ is a member of $r_{m_j+1}[U_{m_j},T]$ such that
$h$ has value $\varepsilon^*$ on $\Ext_{U_{m_j+1}}(B;X^*)$.

In Case (b),
 $r(d)$ is a coding node.
 Since $r\le q$,
the nodes in $\ran(r\re\delta_q)$ have the same passing types over $T\rl \ell_q$
as the nodes in $\ran(q)$ have over $T\rl \ell_q$.
We now need to extend all the other members of
$Y\setminus Y_0$
 to nodes with the required
 passing types at $r(d)$.
 For each
$y\in Y\setminus Y_0$,
 choose
a  member  $z_y\supset y$  in $T_d\re \ell_j$
so that
\begin{equation}
z_y^+(r(d);U_{m_j})\sim y^+(q(d); U_{m_j}).
\end{equation}
\EEAP\ ensures the existence of  such $z_y$.
Let
$Z$ be the level set
consisting of the nodes in $Z_0$ along with
the nodes $z_y$ for
 $y\in Y\setminus Y_0$.
Then $Z$ end-extends $Y$ and moreover,
$U_{m_j}\cup Z\plussim V$.
Letting $U_{m_j+1}=U_{m_j}\cup Y$,
we see that $U_{m_j+1}$ is a member of $r_{m_j+1}[U_{m_j},T]$  and
$h$ has value $\varepsilon^*$ on $\Ext_{U_{m_j+1}}(B;X^*)$.

Now that we have constructed $U_{m_j+1}$,
let $U_{m_{j+1}}$ be any member of
$r_{m_{j+1}}[U_{m_j+1},T]$.
This completes the inductive construction.
Let $S=\bigcup_{j<\om}U_{m_j}$.
Then $S$ is a member of $[D,T]$ and
 for each $X\in\Ext_{S}(B)$,  $h(X)=\varepsilon^*$.
Thus, $S$ satisfies the theorem.
\end{proof}

Indivisibility follows immediately from the previous theorem.
\vskip.1in

\noindent{\em Proof of Theorem \ref{thm.indivisibility}}.
Let $\bfC$ be a singleton structure in $\mathcal{K}$, and suppose $h$ is a coloring of all copies of $\bfC$ inside $\bK$ into two colors.
For the proof of indivisibility,  it suffices   to work in diagonal coding subtrees of $\bU(\bK)$.
Given a diagonal coding tree $T$, let $X^*$ be the coding node in $T$ of least length representing a copy of $\bfC$.
Let $A=D=r_0(T)$ be the empty sequence,  let $B$ be the initial segment of $X^*$ of length one.
Then
Theorem \ref{thm.matrixHL} provides us with a coding tree $S\in [D,T]$ such  that $h$ is monochromatic on $\Ext_S(B;X^*)$.
Since $D=r_0(T)$, every coding node in $S$ representing a copy of
$\bfC$
is a member of  $\Ext_S(B;X^*)$.
Thus, $\bK$ is indivisible.
 \hfill $\square$
\vskip.1in

\begin{rem}\label{rem.plussimfornoncn}
By  the construction in the previous proof, in Case (b) the coding nodes in any member $X\in\Ext_S(B;X^*)$ extend the coding node $t^*_d$.
It then follows from
(3) in Definition \ref{defn.sct}
that for every level set $X\sse S$ with $A\cup X\sim\tilde{A}\cup X^*$, the coding node $c$ in $X$  automatically satisfies
$c^+(c;A)\sim  (t^*_d)^+(t^*_d;A)
\sim \tilde{x}_d^+(\tilde{x}_d;\tilde{A})$, where $\tilde{x}_d$ denotes the  coding node in $X^*$.
Thus, $A\cup X\plussim \tilde{A}\cup X^*$ if and only if the
non-coding nodes in $X$ have immediate successors with similar passing types over $A\cup\{c\}$ as their counterparts in $X^*$ have over $\tilde{A}\cup\{\tilde{x}_d\}$.

Moreover, for languages with only unary and binary relations, in Case (b)  the set
$\Ext_T(B;X^*)$ is exactly the set of all end-extensions $X$ of $B$ such that $A\cup X\plussim \tilde{A}\cup\tilde{X}$.
These observations
will be useful in the proof of  next theorem.
\end{rem}

Recall that two antichains of coding nodes are considered similar if the trees induced by their meet-closures are similar.

\begin{thm}\label{thm.onecolorpertreetype}
Suppose that $\mathcal{K}$
is a \Fraisse\ class in a language with
relation symbols  of arity at most two,
and suppose that $\mathcal{K}$
has a \Fraisse\ limit satisfying
\EEAP$^+$.
Let
$\bT$ be  a diagonal coding subtree of $\bU(\bK)$,
let
$\tilde{C}\sse\bT$ be an antichain of coding nodes,
 and let $T\in\mathcal{T}$ be fixed.
Given any coloring of the set $\{C\sse T:C\sim \tilde{C}\}$,
there is an $S\le T$ such that all  members  of
$\{C\sse S:C\sim \tilde{C}\}$ have the same color.
\end{thm}

\begin{proof}
The proof is by reverse induction on the levels in $\tilde{C}$.
Suppose that $\tilde{C}$ has $n\ge 1$ levels.
If $n=1$, then by Theorem \ref{thm.indivisibility}, we are done.
Otherwise, assume $n\ge 2$ and
 let
$\tilde{X}$ denote $\tilde{C}\re \ell_{\tilde{C}}$, the maximum level of $\tilde{C}$.
Then $\tilde{X}$ is a single coding node.
Let $\tilde{A}$ denote $\tilde{C}\setminus\tilde{X}$; that is, $\tilde{A}$ is the initial segment of  all but the maximum level of $\tilde{C}$.

Let $m_0$ be the least integer  such that  $r_{m_0}(T)$ contains a $+$-similarity copy of $\tilde{A}$ extending to a copy of $\tilde{C}$,
 and
 let $D_0=r_{m_0}(T)$.
Let  $A^0,\dots,A^j$ list
  those   $A\sse T$ such that
  $\max(A)\sse\max(D_0)$ and
 $A$ extends to a similarity copy of $\tilde{C}$.
 For $i\le j$, let $B^i$ denote $(A^i)^+$,
 which we recall is the tree consisting of the nodes in $A^i$ along with all immediate successors of nodes in $\max(A^i)$.
 (These immediate successors are the same whether we consider them in $\bT$ or in $T$.)
 Each $B^j$ is a subtree of $(D_0)^+$.
  Apply Theorem \ref{thm.matrixHL}
   to obtain a $T_{0}^0\in [D_0,T]$ such that $h$ is monochromatic on $\Ext_{T_0^0}(B^0;\tilde{X})$.
Repeat this process,  each time thinning the previous tree
to obtain $T_0^{i+1}\in [D_0,T_0^i]$  so that
for each $i\le j$,  $\Ext_{T_0^j}(B^i;\tilde{X})$ is monochromatic.
Let $T_0$ denote $T_0^j$.
Then for each $A^i$, $i\le j$,
every extension of $A^i$ to a similarity copy of $\tilde{A}\cup\tilde{X}$ inside $T_0$ has the same color.

Given $k<\om$ and $T_k$, let $m_{k+1}$ be the least integer greater than $m_k$
 such that $r_{m_{k+1}}(T_k)$ contains a $+$-similarity copy of $\tilde{A}$  extending to a copy of $\tilde{C}$.
 Let $D_{k+1}=r_{m_{k+1}}(T_k)$, and
 index those $A$  with $\max(A)\sse \max(D_{k+1})$
such that
 $A$ extends to a similarity copy of $\tilde{C}$
  as $A^i$, $i\le j$ for some $j$.
 Repeat the above process applying  Theorem \ref{thm.matrixHL}  finitely many times to obtain
 a $T_{k+1}\in
[D_{k+1},T_k]$  with the property that
for each  $i\le j$,
all similarity copies of  $\tilde{A}\cup\tilde{X}$ in $T_{k+1}$
 extending $A^i$
have the same color.

Since each $T_{k+1}$ is a member of $ [D_{k+1}, T_k]$,  the union $\bigcup_{k<\om} D_k$ is a member of $\mathcal{T}$, call it $S_1$.
This induces a well-defined coloring of the copies of $\tilde{A}$ in $S_1$ as follows:
Given
$A\sse S_1$ a similarity copy of $\tilde{A}$ extending to a copy of $\tilde{C}$,
let $k$ be least such that $A$ is contained in $r_{m_k}(S_1)$.
Then $\max(A)$ is contained in $\max(D_k)$,
and   $S_1\in [D_k,T_k]$ implies that
for each level set extension $X$ of $A$ in $S_1$ such that
$A\cup X\sim \tilde{A}\cup\tilde{X}$,
these similarity copies of $\tilde{C}$ have the same color.

This now induces a coloring on $+$-similarity copies of
$\tilde{A}$ inside $S_1$.
Let $\tilde{C}_{n-1}$ denote this $\tilde{A}$,
$\tilde{X}_{n-1}$ denote $\max(\tilde{C}_{n-1})$,
and
 $\tilde{A}_{n-1}$ denote $\tilde{C}_{n-1}\setminus
 \tilde{X}_{n-1}$.
Repeat the  argument in the previous
three paragraphs
to obtain $S_2\le S_1$ such that
for each $+$-similarity copy of $\tilde{A}_{n-1}$ in $S_2$,
all extensions to $+$-similarity copies of $\tilde{C}_{n-1}$ in $S_2$ have the same color.

At the end of the reverse induction, we obtain an $S:=S_n\le T$
such that all similarity copies of $\tilde{C}$ in $S$
 have the same color.
\end{proof}

\begin{rem}\label{rem.Gamma_ts}
Recall Convention \ref{conv.Gamma_ts} that if
(a)
$\mathcal{K}$  satisfies \SFAP,
or
(b)
$
\bK$ satisfies \EEAP$^+$ and either has no  unary relations or has no transitive relations, then
we work inside a diagonal coding subtree  $\bT$ of
$\bS$.
Otherwise, we
have a transitive relation as well as unary relations
and we work inside $\bU$.
In this case, any subset $D\sse\bU$ for which $\bK\re D$ contains a copy of $\bK$ will contain a subset $D'\sse D$ such that
$\bK\re D'\cong \bK$ and
for each non-terminal node $s\in D'$ and each $\gamma\in\Gamma$,
there is a coding node $c\in D'$ extending $s$
 such that $\gamma(v)$ holds in $\bK$, where $v$ is the vertex of $\bK$ represented by $c$.
\end{rem}

The next lemma shows that
if $\mathcal{K}$
has \Fraisse\ limit $\bK$ satisfying
\EEAP$^+$, then
within any diagonal coding tree, there is an antichain of coding nodes representing a copy of $\bK$.

\begin{lem}\label{lem.bD}
Suppose \Fraisse\ class $\mathcal{K}$
has \Fraisse\ limit $\bK$ satisfying
\EEAP$^+$.
If $\mathcal{K}$ satisfies \SFAP\, or
 $\bK$ either has no transitive relation or  has no  unary relations, let $T$ be
 a diagonal coding subtree of $\bS(\bK)$; otherwise, let $T$ be a  diagonal
 coding subtree of $\bU(\bK)$.
Then  there is an infinite  antichain of coding nodes  $\bD\sse T$  so that $\bK\re \bD\cong^{\om}\bK$.
\end{lem}

\begin{proof}
We will use
$c^{\bD}_n$ to denote the $n$-th coding node in $\bD$,
and $v^{\bD}_n$ to denote the vertex in $\bK$ coded by
$c^{\bD}_n$.
The antichain $\bD$ will look almost exactly like $T$ in the following sense:
For each $n$,
the level set of $\bD$ containing the $n$-th coding node,
denoted $\bD\re |c^{\bD}_n|$,
will have exactly one more node than $T\re |c^T_n|$, and the $\prec$-preserving bijection between $T\re |c^T_n|$ and
$(\bD\re |c^{\bD}_n|)\setminus\{c^{\bD}_n\}$ will preserve passing types of the immediate successors.
(This is not necessary to the results on big Ramsey degrees, but since we can do this, we will.)
Moreover, letting $T'$ be the coding tree
obtained by
deleting the coding nodes in $\bD$ and declaring
 the node  $t$ in   $(\bD\re |c^{\bD}_n|)\setminus\{c^{\bD}_n\}$ which has  $t\wedge c^{\bD}_n$ of maximal length  to be the $n$-th coding node in $T'$,
 then $T'\sim T$.

Let  $m_n$ denote the integer such that the $n$-th coding node in $T$ is in the $m_n$-th level of $T$;
that is, $c^T_n$ is in the maximal level of $r_{m_n+1}(T)$.
To construct $\bD$, begin by
taking  the first $m_0$ levels of $\bD$ to equal those of $T$; that is, let $r_{m_0}(\bD)= r_{m_0}(T)$.
Each of these  levels contains a splitting node.
Let $X$ denote the set of immediate successors in
$\widehat{T}$  of
the maximal nodes in
$r_{m_0}(\bD)$.
By \EEAP,
whatever we choose to be $c^{\bD}_0$,
each node in $X$ can extend to a node in $T$ with  the desired passing type at  $c^{\bD}_0$.

Let $s$ denote the node in $X$ which extends to $c^T_0$.
It only remains to  find a splitting node  extending $s$
whose immediate successors
can be   extended to a coding node
$c^{\bD}_0$ (which will be terminal in $\bD$)
 and another node $z$ of length $|c^{\bD}_0|+1$
satisfying
\begin{equation}
z(c^{\bD}_0)
\sim
c_0^+(c_0);
\end{equation}
or in other words,
$z\re (\bK\re \{v^{\bD}_0\})$ is the same as the type of $c^+_0$ over $\bK\re \{v_0\}$.

To do this, we  utilize  \EEAP:
In this
application of \EEAP,
  $\bfA$ is the empty structure and
$\bfC$ is the structure
$\bK\re \{v^T_0,v^T_i\}$ for any $i>0$ such that
 $c^T_i$ extends $c^T_0$.
Extend $s$ to some splitting node $s'\in T$ long enough so that
the structure $\bK\re (T\re|s'|)$
acts as $\bfA'$ as in the set-up of (B) in \EEAP.
In (B1), we take $\bfC'$ to be a copy of $\bfC$ represented by some coding nodes $c^T_j,c^T_k$, where
 $s'\sse c^T_j\sse c^T_k$.
 In (B2), we let $\bfB=\bfA'$,
 and take $\sigma=\tau=s'$.
Let $t_0,t_1$ denote the immediate successors of $s'$ in $\widehat{T}$,
and take a coding node in $T$, which we denote $c^{\bD}_0$,  extending $t_0$.
(The vertex $v^{\bD}_0$ which
$c^{\bD}_0$
 represents is the $v''$ in (B3).)
Then by  \EEAP,
there is a coding node $c^T_m$ extending $t_1$ such that
\begin{equation}
c^T_m(c^{\bD}_0)
\sim
c_0^+(c_0).
\end{equation}
We let $y=c^T_m\re|c^{\bD}_0|$ and $z=c^T_m\re(|c^{\bD}_0|+1)$.
The passing type of $z$ at $c^{\bD}_0$ is the desired passing type.
We let $\bD\re(|c^{\bD}_0|+1)$ consist of the node
$z$ along with  extensions of the nodes in
$X\setminus \{s\}$ to the length
of $z$ so that their passing types at $c^{\bD}_0$
are as desired;
that is, the $\prec$-preserving bijection between
$T\re (|c^T_0|+1)$ and
$\bD\re (|c^{\bD}_0|+1)$
preserves passing types at $c^T_0$ and $c^{\bD}_0$, respectively.
We let
$\bD\re |c^{\bD}_0|$ equal $\{c^{\bD}_0\}\cup \bD\re |c^{\bD}_0|$.

For the general construction stage, given $\bD$ up to the level of $|c^{\bD}_n|+1$,
let
$X$ denote the level set  $\bD\re(|c^{\bD}_n|+1)$.
Extend the nodes in $X$ in the same way that the nodes in $T\re |c^T_{n+1}|$ extend the nodes in
$T\re (|c^T_{n}|+1)$.
Let
$s$ denote the node in $X$ which needs to be extended to the next coding node $c^{\bD}_{n+1}$, and repeat the argument above find a suitable  splitting node and extensions to a coding node $c^{\bD}_{n+1}$ as well as a non-coding node of the same height with the desired passing type at $c^{\bD}_{n+1}$ over $\{c^{\bD}_i:i\le n\}$.
By \EEAP, the other nodes in $X$ extend to have the desired passing types.
\end{proof}

By Remark \ref{rem.plussimfornoncn}, given two antichains of coding nodes $C$ and $C'$,
it follows that
$C\sim C'$ if and only if for any $k$, the first $k$ levels of the trees induced by $C$ and $C'$, respectively, are $+$-similar.

Let $\mathcal{K}$ be a \Fraisse\ class
with \Fraisse\ limit $\bK$
satisfying \EEAP$^+$.
If $\mathcal{K}$ satisfies \SFAP, or
$\bK$
either has no transitive relation or has no unary relations,
then let $\bT$ be a diagonal coding subtree of $\bS(\bK)$.
Otherwise, let $\bT$ be a diagonal coding subtree of
$\bU(\bK)$.
Recalling Definition \ref{def.ssmap} and recalling that we may identify a subset of $\bT$ with the subtree it induces,
given an antichain of coding nodes $C\sse \bT$,
we  let $\Sim(C)$ denote the set of all antichains $C'$ of coding nodes in $\bT$ such that $C'\sim C$.
Thus, $\Sim(C)$ is a $\sim$-equivalence class, and  we call $\Sim(C)$ a {\em similarity type}.
For $S\sse \bT$, we write $\Sim_S(C)$ for the set of $C'\sse S$ such that $C'\sim C$.

\begin{defn}\label{defn.SimTC}
We say that $C$ {\em represents} a copy of a structure  $\bG\in\mathcal{K}$ when
$\bK\re C\cong \bG$.
Given $\bG\in\mathcal{K}$,
let  $\Sim(\bG)$ denote a set consisting of
one representative  from each
 similarity type
 $\Sim(C)$ of diagonal  antichains of coding nodes $C\sse\bT$ representing
a copy of
$\bG$.
\end{defn}

The next theorem  providing upper bounds follows immediately from  Theorem \ref{thm.onecolorpertreetype}
and  Lemma \ref{lem.bD}.

\begin{thm}[Upper Bounds]\label{thm.onecolorpertype}
Suppose $\mathcal{K}$
is a \Fraisse\ class
with relations of arity at most two and
 with Fraisse\ limit $\bK$ satisfying
\EEAP$^+$.
Then for each  $\bG\in \mathcal{K}$,
the big Ramsey degree of $\bG$ in $\bK$ is bounded by the  number of similarity types of diagonal antichains of coding nodes representing $\bG$; that is,
 $$T(\bG,\bK)\le |\Sim(\bG)|.
$$
Moreover, given any finite collection  $\mathcal{G}$
 of structures in $\mathcal{K}$ and any
 coloring of all copies of   each $\bG\in\mathcal{G}$  in $\bK$ into finitely many colors,
  there is a substructure $\bJ$ of $\bK$
  such that
  $\bJ\cong^{\om}\bK$ and
  each
  $\bG\in\mathcal{G}$ takes at most  $|\Sim(\bG)|$ many colors in $\bJ$.
\end{thm}

\begin{proof}
Let $\mathcal{G}$ be a finite collection of structures in $\mathcal{K}$.
 Given any $T\in\mathcal{T}$,
 apply
Theorem  \ref{thm.onecolorpertreetype}
finitely many times to obtain a coding subtree $S\le T$ such that the coloring  takes one color on
the set
$\Sim_S(C)$,
 for each $C\in \bigcup\{\Sim(\bG):\bG\in\mathcal{G}\}$.
Then apply  Lemma \ref{lem.bD} to take an antichain of coding nodes, $\bD\sse S$, such that $\bK\re \bD\cong^{\om}\bK$.
Letting $\bJ=\bK\re\bD$, we see that
there are at most $|\Sim(\bG)|$ many colors  on the copies of $\bG$ in $\bJ$.
\end{proof}

In the next section, we will show that these bounds are exact.


\section{Simply characterized  big Ramsey degrees and  structures}\label{sec.brd}

In this section we prove
that
if a  \Fraisse\ limit $\bK$ of a \Fraisse\ class $\mathcal{K}$
with relations of arity at most two
satisfies \EEAP$^+$, then we can characterize the exact big Ramsey degrees of $\bK$; furthermore, $\bK$ admits a big Ramsey structure.
We first show, in Theorem \ref{thm.persistence}, that each of the similarity types in
Theorem \ref{thm.onecolorpertype} persists, and hence these similarity types form canonical partitions.  From this, we obtain
a succinct characterization of the exact big Ramsey degrees of $\bK$. We then prove, in Theorem \ref{thm.apply}, that
canonical partitions
characterized via similarity types
satisfy a condition of Zucker
(\cite{Zucker19})
guaranteeing
the existence of big Ramsey structures.
This involves showing how
Zucker's condition, which is
phrased in terms of colorings of embeddings of a given structure, can be
met
by
canonical partitions
that are
in terms of
colorings of copies of a given structure.
The big Ramsey structure for $\bK$ thus obtained also has a simple characterization.
From these
results,
 we deduce Theorem \ref{thm.main}.

\begin{rem}
We point out that Theorem \ref{thm.persistence} also provides lower bounds for  the big Ramsey degrees in a
 \Fraisse\ limit $\bK$ of a \Fraisse\ class $\mathcal{K}$
with relations of any arity.
\end{rem}

Recall from Definition \ref{defn.cp} the notion of {\em persistence}.
We first show, in Theorem \ref{thm.persistence}, that
given $\bG\in\mathcal{K}$,
each of the similarity types in
 $\Sim(\bG)$ persists in any subcopy of $\bK$.
From this, it will follow
  that
  the big Ramsey degree
 $T(\bG,\mathcal{K})$ is exactly the cardinality of
  $\Sim(\bG)$ (Theorem \ref{thm.bounds}).
The proof of  Theorem  \ref{thm.persistence} follows the outline and many ideas of the proof of
Theorem 4.1 in
 \cite{Laflamme/Sauer/Vuksanovic06},
where Laflamme, Sauer, and Vuksanovic
proved persistence of diagonal antichains for
unrestricted
binary relational structures.

Recall that $\Gamma$ denotes  the set of all complete
$1$-types
of elements of $\bK$
over the empty
set.
For $\gamma\in\Gamma$,
we  let $\bC_\gamma$ denote the set of coding nodes $c_n$ in $\bS$
such that
$\gamma(v_n)$ holds in $\bK$, where $v_n$ is the vertex of $\bK$ represented by $c_n$;
 let $\gamma_{c_n}$ denote this  $\gamma$.
The next definition extends the notion of ``passing number preserving map'' from
Theorem 4.1 in \cite{Laflamme/Sauer/Vuksanovic06}.

\begin{defn}\label{defn.ptp}
Given two subsets $S,T\sse \bS$ with
coding nodes
$\lgl c^S_n:n<M\rgl$ and $\lgl c^T_n:n<N\rgl$, respectively, where $M\le N\le\om$,
 we say that
a map $\varphi:S\ra T$ is {\em passing type preserving (ptp)} if and only if  the following hold:
\begin{enumerate}
\item
$|s|<|t|$ implies  that $|\varphi(s)|<|\varphi(t)|$.
\item
$\varphi$ takes each coding node in $S$ to a  coding node in $T$,
and $\gamma_{\varphi(c^S_n)}=\gamma_{c^S_n}$ for each $n\le M$.
\item
$\varphi$  preserves passing types:
For any  $s\in S$  and
 $m<M$ with $|c^S_{m-1}|<|s|$,\\
$\varphi(s)(\varphi(c^S_m); \{\varphi(c^S_0),\dots \varphi(c^S_{m-1})\})
\sim
s(c^S_m;\{c^S_0,\dots, c^S_{m-1}\})$.
\end{enumerate}
\end{defn}

\begin{thm}[Persistence]\label{thm.persistence}
Let $\mathcal{K}$ be a \Fraisse\ class
and $\bK$ an enumerated \Fraisse\ structure for
$\mathcal{K}$.
Suppose that $\bK$ satisfies \EEAP$^+$. Let $\bT$ be
 a diagonal coding tree representing a copy of $\bK$,
 let
$\bD\sse \bT$  be any antichain of coding nodes  representing $\bK$,
and let
 $A$ be any    antichain of coding nodes in $\bD$.
Then for any subset $D\sse\bD$ representing a copy of $\bK$,
there is a similarity copy of $A$ in $D$; that is,
$A$ persists in $D$.
\end{thm}

\begin{proof}
We shall be working under the assumption  that either (a) there is an  antichain  of coding nodes $\bD\sse\bT\sse\bS$
such that $\bK\re\bD\cong \bK$,
or
(b)
that for every  antichain  of coding nodes $\bD\sse\bT\sse \bU$
such that $\bK\re \bD\cong\bK$,
there is a subset $D$ also coding $\bK$
with
the property that for each non-terminal node $t\in D$
 and for each $\gamma\in \Gamma$,
 there is a coding node  in $D\cap\bC_\gamma$ extending $t$.
(Recall
Remark \ref{rem.Gamma_ts}.)
In either case, we  let   $\bD$ be an antichain of coding nodes  in $\bT$ representing a copy of $\bK$,
where $\bD$ is constructed as in Lemma \ref{lem.bD}.
Throughout, we shall use the notation $\bU$, but keep in mind that if (a) above holds, then we are working in $\bS$.

Without loss of generality, we may assume that
 $\bK\re \bD \cong^{\om}\bK$,
 by thinning $\bD$ if necessary.
Let $D\sse\bD$ be any subset
 such that $\bK\re D\cong\bK$;  let
 $\bJ$ denote $\bK\re D$.
Again, without loss of generality, we may assume that
$\bJ\cong^{\om}\bK$.
Let $\bC=\{c_n:n<\om\}$ denote the set of all coding nodes in $\bU$, and note that $D\sse\bD\sse\bC$.
Then  the map $\varphi:\bC\ra  D$ via
$\varphi(c_n)=c^{D}_n$ is passing type preserving, where $\lgl c^{D}_n:n<\om\rgl$ is the enumeration of the nodes in $D$ in order of increasing length.

Define
 \begin{equation}
 \overline{D}
 =\{c^{D}_n\re |c^{D}_m|: m\le n<\om\}.
 \end{equation}
 Then $\overline{D}$ is a union of level sets  (but is not meet-closed).
 We extend the map $\varphi$ to a map $\bar{\varphi}:\bU\ra\overline{D}$ as follows:
 Given $s\in \bU$, let $n$ be least such that $c_n\contains s$ and   $m$ be the integer such that
 $|s|=|c_m|$,
  and define $\bar{\varphi}(s)=\varphi(c_n)\re|\varphi(c_m)|$;
in other words, $\bar{\varphi}(s)=c^{D}_n\re|c^{D}_m|$.

\begin{lem}\label{lem.barvarphiptp}
$\bar{\varphi}$ is passing type preserving.
\end{lem}

\begin{proof}
For $s\in \bU(m)$,
let $n>m$ be  least such that $s=c_n\re |c_m|$.
Then for any
for $i<m$,
\begin{align}
\bar{\varphi}(s)(\varphi(c_i); \{\varphi(c_0),\dots,\varphi(c_{i-1})\})
&=(\varphi(c_n)\re |\varphi(c_m)|)(\varphi(c_i); \{\varphi(c_0),\dots,\varphi(c_{i-1})\})\cr
&=(c^{D}_n\re|c^{D}_m|)(c^{D}_i; \{c^{D}_0,\dots,c^{D}_{i-1}\})\cr
&=c^{D}_n(c^{D}_i; \{c^{D}_0,\dots,c^{D}_{i-1}\})\cr
&\sim c_n(c_i; \{c_0,\dots,c_{i-1}\})\cr
&= (c_n\re|c_m|)(c_i; \{c_0,\dots,c_{i-1}\})\cr
&=s(c_i; \{c_0,\dots,c_{i-1}\})
\end{align}
where the $\sim$ holds
since $\varphi:\bC\ra D$ is ptp.
Therefore, $\bar{\varphi}$ is ptp.
\end{proof}

Given a fixed subset $S\sse \bU$ and $s\in S$, we let $\widehat{s}$  denote the set of all $t\in S$ such that $t\contains s$.
The ambient set $S$ will either be
$\bU$
 or $\overline{D}$,  and will  be clear from the context.
We say that a set $X$ is {\em cofinal} in $\widehat{s}$  (or {\em cofinal} above $s$)  if and only if
for each $t\in\widehat{s}$, there is some  $u\in X$ such that $u\contains t$.
 A subset $L\sse D$ is  called {\em large} if and only if there is some $s\in\bU$ such that
 $\varphi^{-1}[L]$ is cofinal in
 $\widehat{s}$.
 We point out that since $D$ is a set of coding nodes,
 for any $L\sse D$,
 $\varphi^{-1}[L]$ is a subset of $\bC$.

\begin{lem}\label{lem.lemma1}
Let $n<\om$
and $L\sse D$
be given.
Suppose  $L=\bigcup_{i<n}L_i$ for some $L_i\sse D$.  If $L$ is large,
then there is an $i<n$ such that $L_i$ is large.
\end{lem}

\begin{proof}
Suppose not.
Since $L$ is large, there is some $t\in\bU$ such that $\varphi^{-1}[L]$ is cofinal above $t$.
Since $L_0$ is not large,
there  is some $s_0\contains t$ such that $\varphi^{-1}[L_0]\cap \widehat{s_0}=\emptyset$.
Given $i<n-1$ and $s_i$, since $L_{i+1}$ is not large, there is some $s_{i+1}\contains s_i$ such that
$\varphi^{-1}[L_{i+1}]\cap \widehat{s_{i+1}}=\emptyset$.
At the end of this recursive construction, we obtain an $s_{n-1}\in\bS$ such that for all $i<n$,
$\varphi^{-1}[L_{i}]\cap \widehat{s_{n-1}}=\emptyset$.
Hence, $\varphi^{-1}[L]\cap \widehat{s_{n-1}}=\emptyset$, contradicting that
$\varphi^{-1}[L]$ is cofinal above $t$.
\end{proof}

Thus,  any partition of a large set into finitely many pieces
contains
at least one  piece which is large.

Given a subset $I\sse\om$,
let $\bK\re I$ denote the
substructure  of $\bK$ on vertices $\{v_i:i\in I\}$.
Recalling that $\bJ$ denotes $\bK\re D$, we let  $\bJ\re I$ denote the
substructure of $\bJ$ on
vertices $\{v^{D}_i:i\in I\}$, where $v^D_i$ is the vertex  represented by the coding node $c^D_i$.
The next lemma  will be applied in two important ways.
First, it will
aid in  finding  splitting nodes  in the meet-closure of $D$ (denoted by $\cl(D)$)
as  needed   to construct a similarity copy of a given antichain of coding nodes   $A$ inside $D$.
 Second,
 it
 will guarantee that we can find nodes in $D$ which have the needed passing types in order to continue building a similarity copy of $A$ in
 $D$.

Given a subset $L\sse\overline{D}$,
we say that $L$ is {\em large}
exactly when $L\cap D$ is large.
 Note that since $\varphi$ has range $D$,
 $\varphi^{-1}[L]$ is always a subset of $\bC$.
Given  a finite set
$I\sse\om$
and $1$-types $\sigma,\tau$
over $\bJ\re I$ and $\bK\re I$, respectively,
 we write $\sigma\sim\tau$ exactly when
 for each $i\in I$,
 $\sigma(v^D_i; \bJ\re I_i)\sim
 \tau(v_i;\bK\re I_i)$,
 where $I_i=\{j\in I:j<i\}$.

\begin{lem}\label{lem.lemma2}
Suppose  $t$ is in $\overline{D}$
 and     $\widehat{t}$  is   large.
 Let $s_*\in\bU$ be such that
 $\varphi^{-1}[\,\widehat{t}\,]$ is cofinal in $\widehat{s_*}$.
 Let  $i$ be    the index such that
$|s_*|=|c_{i}|$,  and
let  $I\sse i$,
 $n\ge i$,
 $I'=I\cup\{n\}$, and
$\ell=|c^{D}_n|$ be given.
For any
complete
 $1$-type $\sigma$ over
 $\bJ\re I'$ such that $\sigma\re (\bJ\re I)\sim s_*\re (\bK\re  I)$,
let
\begin{equation}
L_{\sigma}=
\bigcup\{\widehat{u}:u\in \widehat{t}\re \ell\mathrm{\ and\ }
u\re (\bJ\re I')\sim\sigma\}.
\end{equation}
Then
$L_{\sigma}$ is large.
\end{lem}

\begin{proof}
Fix an $s\contains s_*$ with $|s|>|c_n|$ such that
 $s\re(\bK\re I')\sim \sigma$ holds.
 Suppose towards a contradiction that
$L_{\sigma}$ is not large, and
fix  an extension  $s'\contains s$ such that
$\varphi^{-1}[L_{\sigma}]
\cap \widehat{s'}=\emptyset$.
Since $\varphi^{-1}[\,\widehat{t}\,]$ is cofinal in $\widehat{s_*}$,
there is a coding node $c_j$
in  $\varphi^{-1}[\,\widehat{t}\,]
$
extending $s'$.
Notice that $c_j$ being in $\varphi^{-1}[\,\widehat{t}\,]$ implies that $\varphi(c_j)$ extends $t$.
Moreover,
since
$c_j$ extends $s$
and  $\varphi$ is passing type preserving,
it follows that
$\varphi(c_j)\re (\bJ\re I')\sim\sigma$.
Thus,
$\varphi(c_j)$
 is in $L_{\sigma}$
 and hence,
$c_j$ is in $\varphi^{-1}[L_{\sigma}]$.
But  then
$c_j\in \varphi^{-1}[L_{\sigma}]
\cap \widehat{s'}$, a contradiction.
\end{proof}

For the remainder of the proof,
fix  a diagonal  antichain of coding nodes $A\sse\bD$.
Let $\lgl c^A_i:i<p\rgl$  enumerate the  nodes in $A$ in order of increasing length, where $p\le \om$, noting that  each $c^A_i$ is a coding node.
For each $i<p$,
let $\gamma_i$ denote $\gamma_{c^A_i}$.

Let $B$ denote the meet-closure of $A$;
 label the nodes of $B$ as $\lgl b_i:i<q\rgl$ in increasing order of length, where $q\le \om$.
Thus, each node in  $B$ is either a member of $A$ (hence, a coding node) or else a splitting node of degree two which is the meet of two  nodes in $A$.
Our goal is to build a similarity copy of $B$ inside the meet-closure of $D$, denoted $\cl(D)$;
that is, we aim to build a  similarity map $f$  from $B$ into $\cl(D)$ so that  $f[B]\sim B$.
Now the map
$\varphi$ is already passing type preserving.
The challenge is to get a $\prec$- and meet- preserving map which is still passing type preserving from  $B$ into $\cl(D)$.

First notice that  $B\re 1=A\re 1$.
If we are working in $\bS$, then $B\re 1$ is a subset of
$D\re 1=\bS(0)=\Gamma$ with  possibly more than one node.
If we are working in $\bU$, then $B\re 1$ is the singleton $D\re 1$.
Without loss of generality, we may assume
that $|c^A_0|>1$.
Let
$f_{-1}$ be the  empty map,
let
$T_{-1}$ denote $B\re 1$,
 let $N_{-1}=1$, and
let $\psi_{-1}$ be the identity map on
$T_{-1}$.
Let $\widehat{D}$ be the tree induced by cl$(D)$.
Let $M_{-1}=1$, and for each $k<q$, let
$M_k=|b_{k-1}|+1$, where we make the convention $|b_{-1}|=0$.

For each $k<q$
we will recursively define meet-closed sets $T_k\sse\widehat{D}$, maps $f_k$ and $\psi_k$, and $N_k<\om$ such that the following hold:
\begin{enumerate}
\item
$f_k$ is a $+$similarity embedding of $\{b_i:i<k\}$ into $T_k$.
\item
$|t|\le N_k$ for all $t\in T_k$.

\item
All maximal nodes of $T_k$ are either in $T_k\re N_k$, or else in the range of $f_k$.

\item
$\theta_k$ is a $\prec$ and passing type preserving  bijection of
$B\re M_k$ to $T_k\re N_k$.

\item
$T_{k-1}\sse T_k$, $f_{k-1}\sse f_k$, and $N_{k-1}<N_k$.
\end{enumerate}

The idea behind $T_k$ is that it will contain a similarity  image of $\{b_i:i<k\}\cup (B\re M_k)$,
the nodes in the image of $ B\re M_k$ being the ones we need to continue extending in order to build a similarity copy of $B$ in $\cl(D)$
If   $p>1$ in the  Extension Property,
then in (1) we further assume that $f_k$  preserves
 the $\psi$ value of the splitting node.

Assume  now that  $k<q$, and (1)--(6) hold for all $k'<k$.
We have two cases.
\vskip.1in

\noindent\bf Case I. \rm
$b_k$ is a splitting node.
\vskip.1in

Let $i<j<p$ be such that  $b_k=c^A_i\wedge c^A_j$.
Let $t_k=\theta_k( b_k\re M_k)$,
recalling that by (5),
 $t_k$ is a member of $T_k\re N_k$.
By (3), $\widehat{t_k}$ is large,
so we can fix  a coding node $c_{n}\in \varphi^{-1}[\widehat{t_k}]$.
Then
 $c^{D}_n=\varphi(c_n)\contains t_k$.
Let $N_{k+1}=|c^{D}_{n+1}|$.

Our goal is to find two incomparable nodes which extend $t_k$
and have cones which are
 large.
Recalling that $N_k=|t_k|$,
 let
 \begin{equation}
 I=\{i<\om: |c_i^{D}|< N_k\},
 \end{equation}
 and let $I'=I\cup\{n\}$.
Let $\sigma$ and $\tau$ be distinct $1$-types over
 $\bJ( I')$
 such that
 both
  $\sigma\re  \bJ( I)$
  and
    $\tau\re \bJ( I)$
 equal  $t_k\re  \bJ( I)$.
For each    $\mu\in\{\sigma,\tau\}$,
let
\begin{equation}
L_{\mu}=\bigcup\{\widehat{u}:u\in\widehat{t_k}\re N_{k+1}\mathrm{\  and \ }
u(c_n^{D}; \bJ( I))=\mu\}.
\end{equation}
By Lemma \ref{lem.lemma2},
both
$L_\sigma$ and  $L_\tau$ are large.
It then follows from Lemma \ref{lem.lemma1} that
there are
 $t_{\sigma},t_{\tau}\in\widehat{t_k}\re N_{k+1}$
such that
$t_{\sigma}\in L_\sigma$ and
$t_{\tau}\in L_\tau$,
and both
$\widehat{t_{\sigma}}$ and $\widehat{t_{\tau}}$  are
large.
Since $\sigma\ne\tau$, it follows that $t_\sigma\ne t_\tau$.
Hence, $t_\sigma$ and $t_\tau$ are incomparable, since they have the same length, $N_{k+1}$.
Since
both $t_\sigma\contains t_k$ and $t_\tau\contains t_k$,
we have $t_\sigma\wedge t_\tau\contains t_k$.

As $(B\cap \widehat{b_k})\re M_k$ has size exactly two,
define $\theta_{k+1}$
on  $(B\cap \widehat{b_k})\re M_{k+1}$
to be the unique $\prec$-preserving map
onto $\{t_\sigma,t_\tau\}$.
Let $E_k$ denote $(B\setminus\widehat{b_k})\re M_k$.
For $s\in E_k$, choose some
$t_s\in\widehat{\theta_k(s)}\re N_{k+1}$ such that
$\widehat{t_s}$ is large.
This is possible by Lemma \ref{lem.lemma1}, since
$\bigcup\{\widehat{t}:t\in \widehat{\theta_k(s)}\re N_{k+1}\}$
is large.
Every $s\in E_k$ has a unique extension  $s'\in B\re M_{k+1}$.
Define $\theta_{k+1}(s')=t_s$.
Let $f_{k+1}$ be the extension of $f_k$ which sends
$b_k$ to $t_\sigma\wedge t_\tau$,
 and let
\begin{equation}
T_{k+1}=T_k\cup \{t_\sigma,t_\tau, t_\sigma\wedge t_\tau\}\cup
\{t_s:s\in E_k\}.
\end{equation}

For
$\mathcal{COE}_{m,p}$,
if $\psi(b_k)=m$, then we require
 $\sigma$ and $\tau$ to  both  include
$\neg E_{m-1}(x, v^D_n)$;
if $\psi(b_k)<m$, then
we require
 $\sigma$ and $\tau$ to  both  include
$E_{\psi(b_k)}(x, v^D_n)$
and
$\neg E_{i}(x, v^D_n)$ for all $i<\psi(b_k)$.
In general,  if
$\bK$ satisfies
(2) of the Extension Property,
and if $\psi(b_k)=m$,
then we take $\sigma$ and $\tau$ above
so that the pair $\{\sigma\re (\bK\re \{v^D_n\}),
\tau\re (\bK\re \{v^D_n\})  \}$
is in $Q_m$ in (2b) of the Extension Property.

This completes Case I.
\vskip.1in

\noindent\bf Case II. \rm
$b_k$ is a coding  node.
\vskip.1in

In this case, $b_k=c^A_j$ for some $j<p$.
By the Induction Hypothesis,
for each $t\in T_k\re N_k$,
$\widehat{t}$ is large;
so we can choose  some $s_t\in\bU$ such that $\varphi^{-1}[\,\widehat{t}\,]$ is  cofinal above $s_t$.
Fix
$t_*=\theta_k(b_k\re  M_k)\in T_k\re N_k$.
Choose a coding node $c_n\contains s_{t_*}$ in $\bU$ such that $|c_n| >\max\{|s_t|:t\in T_k\re N_k\}$
and $\gamma_{c_n}=\gamma_j$, the $\gamma\in\Gamma$ which the vertex $v^A_j$ satisfies.
(In the case that $\bD\sse\bS$, this $\gamma_j$ is
already guaranteed  since $c_n\contains s_{t_*}\contains \gamma_j$.
If $\bD\sse\bU$, there are cofinally many coding nodes extending $s_{t_*}$ which satisfy $\gamma_j$.)
Let
 $d_k$ denote $c^{D}_n=\varphi(c_n)$,
 noting that $\gamma_{d_k}=\gamma_j$.
 Extend $f_k$ by defining
 $f_{k+1}(b_k)=d_k$,
 and
let $N_{k+1}=|c^D_{n+1}|$.
If $q<\om$ and $k=q-1$, we are done.
Otherwise,
we must extend the other members of
$(T_k\re N_k)\setminus\{t_*\}$  to nodes in
$\widehat{D}\re N_{k+1}$ so as to satisfy (1)--(6).

For each  $i\in\{k,k+1\}$,
let
$E_i=(B\re M_i)\setminus \{b_i\re M_i\}$.
Fix an   $s\in E_k$
and let $t=\theta_k(s)$, which is a node in $T_k\re N_k$.
Note that there is a unique $s'\in  E_{k+1}$
such that $s'\contains s$.
Let $A\rl j$ denote $\{c^A_i:i\le j\}$,
$\sigma$ denote $s'\re (A\rl j)$, and
$f_k[A\rl j]$ denote $\{f_k(c^A_i):i<j\}$.
Let
 $I=\{i<\om:c^{D}_{i}\in
f_k[A\rl j]\}$.
Our goal is to find a $t'\contains t$ with $|t'|>|d_k|$
such that $t'(d_k;\bJ(I))\sim \sigma$.

Take  $c_m$ to be any coding node in $\bS$  extending
$s$
such that
$|c_m|>|c_n|$ and $c_m(c_n;A\rl j)\sim\sigma$.
Such a  $c_m$  exists by  \EEAP.
Then  $ \varphi(c_m) (d_k; \bJ(I))\sim \sigma$,
 since $\varphi$ is passing type preserving.
By Lemma \ref{lem.lemma2},
\begin{equation}
L_\sigma:=
\bigcup\{\widehat{u}:u\in\widehat{t}\re N_{k+1}\mathrm{\ and\ }
u(d_k;f_k[A\rl j])
\sim
\sigma\}
\end{equation}
is
large.
Thus, by Lemma \ref{lem.lemma1},
there is some $u_s\in\widehat{t}\re N_{k+1}$ such that $\widehat{u}_s$ is large.
Define $\theta_{k+1}(s)=u_s$.
This builds
\begin{equation}
T_{k+1}=
T_k\cup\{d_k\}\cup\{\psi_{k+1}(s):s\in E_k\}
\end{equation}
and concludes the construction in Case II.
\vskip.1in

Finally, let $f=\bigcup_k f_k$.
Then $f$ is a similarity map from $B$ to $f[B]$,
and thus,
 the antichain of coding nodes in  $f[A]$ is similar to $A$.
Therefore, all similarity types of  diagonal  antichains of coding nodes  persist in $\bJ$.
\end{proof}

As the antichain in the previous theorem can be infinite, we immediately obtain the following corollary.

\begin{cor}\label{cor.presbD}
Suppose $\bK$
satisfies \EEAP$^+$.
Given
$D$ a subset of $\bD$ which represents a copy of $\bK$,
there is a subset $D'$ of $D$ such that $D'\sim \bD$,
assuming
$\bK$
satisfies \EEAP$^+$.
\end{cor}

Combining the previous
results,
we obtain canonical partitions for \Fraisse\ classes
with relations of arity at most two
that have  \Fraisse\ limits
satisfying \EEAP$^+$;
these canonical partitions
are simply described by similarity types.

\begin{thm}[Simply characterized big Ramsey degrees]\label{thm.bounds}
Let
$\bK$ be an enumerated \Fraisse\ structure for a
 \Fraisse\ class $\mathcal{K}$
with relations of arity  at most two
 such that $\bK$ satisfies
 \EEAP$^+$.
Given  $\bG\in\mathcal{K}$,
the partition $\{\Sim(C):C\in\Sim(\bG)\}$ is
 a canonical partition of
 the copies of $\bG$ in $\bK$.
 It follows that
the big Ramsey degree $T(\bG,\bK)$
equals the  number of similarity types of antichains of coding nodes in $\bT$ representing $\bG$.
That is,
$$
T(\bG,\bK)=|\Sim(\bG)|.
$$
\end{thm}

\begin{proof}
Let   $\bG\in\mathcal{K}$ be given, and suppose $h$ is a coloring of all copies of $\bG$ in $\bK$ into finitely many colors.
By Theorem
\ref{thm.onecolorpertype},
there is an antichain of coding nodes $\bD\sse\bT$ which codes a copy of $\bK$, and moreover, for each $C\in\Sim(\bG)$,
$h$ is constant on $\Sim_{\bD}(C)$.
Let $\bJ=\bK\re\bD$.

Given any
subcopy
$\bJ'$ of $\bJ$,
Theorem   \ref{thm.persistence}
implies that $\Sim_D(C)\ne\emptyset$ for each $C\in\Sim(\bG)$,
where $D=\bS\re \bJ'$.
Thus,
$\{\Sim(C):C\in\Sim(\bG)\}$ is
 a canonical partition of
 the copies of $\bG$ in $\bK$.
It follows that
$T(\bG,\mathcal{K})= |\Sim(\bG)|$.
 \end{proof}

We now  apply
Theorem \ref{thm.bounds}
to show that \Fraisse\ structures
with \EEAP$^+$
satisfy the conditions of Zucker's Theorem 7.1 in \cite{Zucker19},  yielding
 Theorem \ref{thm.main}.
Zucker used colorings of embeddings
rather than colorings of copies throughout \cite{Zucker19}.
Our task now is to translate
Theorem \ref{thm.bounds},
which uses
colorings of copies of a given structure, into the setting of \cite{Zucker19}.
To do so, we need to review the following  notions from
\cite{Zucker19}.

Let $\bK$ be an enumerated \Fraisse\ structure for a \Fraisse\ class $\mathcal{K}$.
An {\em exhaustion} of $\bK$ is a sequence
$\{\bfA_n:n<\om\}$ with each $\bfA_n\in\mathcal{K}$, $\bfA_n\sse\bfA_{n+1}\sse\bK$,  such that $\bK=\bigcup_{n<\om}\bfA_n$.
Given $m\le n$, write $H_m:=\Emb(\bfA_m,\bK)$ and $H^n_m:=\Emb(\bfA_m,\bfA_n)$.
For $f\in H^n_m$, the function $\hat{f}: H_n\ra H_m$ is defined by $\hat{f}(s)=s\circ f$, for each $s\in H_n$.
(Here we are using Zucker's notation, so $s$ is denoting an embedding rather than a node in $\bU$.)

The following  terminology is  found in
Definition 4.2  in \cite{Zucker19}.
A set $S\sse H_m$ is {\em unavoidable} if for each embedding $\eta:\bK\ra\bK$, we have $\eta^{-1}(S)\ne\emptyset$.
Fix $k\le r<\om$ and let
 $\gamma: H_m\ra r$  be a coloring.
 We call
 $\gamma$  an
{\em unavoidable $k$-coloring}  if
the image of $\gamma$, written $\mathrm{Im}(\gamma)$, has cardinality $k$,
and for each $i<r$,
we have $\gamma^{-1}(\{i\})\sse H_m$ is either empty
or unavoidable.
Thus, an unavoidable coloring is essentially the same concept as persistence, with the addition that attention is also given to the embedding.

The following is taken from Definition 4.7 in \cite{Zucker19}:
Let $\gamma$ and $\delta$ be colorings of $H_m$.
We say that $\delta$ {\em refines} $\gamma$ and write $\gamma\le\delta$ if whenever $f_0,f_1\in H_m$ and $\delta(f_0)=\delta(f_1)$,
then $\gamma(f_0)=\gamma(f_1)$.
For $m\le n<\om$, $\gamma$ a coloring of $H_m$, and $\delta$ a coloring of $H_n$,
we say that $\delta$ {\em strongly refines} $\gamma$ and write $\gamma\ll\delta$ if for every $f\in H^n_m$, we have that $\gamma\circ \hat{f}\le \delta$.

Theorem 7.1 in \cite{Zucker19}, which we state next,  provides conditions for showing that a \Fraisse\ limit admits a big Ramsey structure.
We will then apply this theorem to show that all \Fraisse\
structures
with relations of arity at most two that have
\EEAP$^+$
admit big Ramsey structures, thus
giving
Theorem \ref{thm.main}.

\begin{thm}[Zucker,  \cite{Zucker19}]\label{thm.Zucker7.1}
Let $\bK=\bigcup_{n<\om}\bfA_n$ be a \Fraisse\ structure,
where $\{ \bA_n : n < \om\}$ is an exhaustion of $\bK$,
and suppose each $\bfA_n$ has
finite big Ramsey degree
$R_n$
in $\bK$.
Assume that for each $m<\om$, there is an unavoidable $R_m$-coloring $\gamma_m$ of $H_m$ so that
$\gamma_m\ll\gamma_n$ for each $m\le n<\om$.
Then $\bK$ admits a big Ramsey structure.
\end{thm}

Now we show how to translate our results so as to apply Theorem \ref{thm.Zucker7.1}.
Given an enumerated \Fraisse\ structure $\bK$,  we point out that  $\{\bK_n:n<\om\}$ is an exhaustion of $\bK$.
Theorem \ref{thm.bounds} shows that
 $\bK_n$ has finite big Ramsey degree
 $T(\bK_n,\bK)=|\Sim(\bK_n)|$ for colorings of {\em copies} of $\bK_n$ in  $\bK$.
Recalling Remark \ref{rem.embvscopy},
 the big Ramsey degree for {\em embeddings} of $\bK_n$ into $\bK$ is
 $T(\bK_n,\bK)\cdot|\mathrm{Aut}(\bK_n)|$.

\begin{thm}\label{thm.apply}
Suppose $\mathcal{K}$ is a
\Fraisse\ class  with
\Fraisse\ limit $\bK$ and with
canonical partitions characterized via diagonal antichains of coding nodes in a coding tree of $1$-types.
 Then the conditions of Theorem \ref{thm.Zucker7.1} are satisfied.
\end{thm}

\begin{proof}
Recalling that $\bD$ denotes the diagonal antichain of coding nodes constructed in Lemma \ref{lem.bD},
we shall abuse notation and use
$\bK$  to denote the  structure $\bK\re \bD$.
Thus, the universe of $\bK$ will (without loss of generality) be $\om$,
and embeddings  $s$ of initial segments $\bK_n$ into $\bK$ will produce diagonal antichains
$\bD\re s[\bK_n]\sse \bD$.
Given $n<\om$,
let $T_n:=T(\bK_n,\bK)$, and
let $\lgl C^n_0,\dots, C^n_{T_n-1}\rgl$ be an enumeration of $\Sim(\bK_n)$, a set of representatives of the similarity types of diagonal antichains of coding nodes representing a copy of $\bK_n$.
Let Aut$(\bK_n)$ denote the set of automorphisms of $\bK_n$.

As $\bK_n$ has vertex set $n=\{0,\dots,n-1\}$,
its vertex set is linearly ordered.
Given $s\in H_n$,
let $\bfA:=s[\bK_n]$, with vertex set
$\lgl \mathrm{a}_0,\dots, \mathrm{a}_{n-1}\rgl$ written in increasing order as a subset of $\om$.
Let $p_s$ denote  the permutation of $n$
defined by $s(j)=\mathrm{a}_{p_s(j)}$, for $j<n$.
Given  $\ell<T_n$,
let $\bfC^n_\ell$ denote the structure $\bK\re C^n_\ell$,
and
 let  $\lgl v^\ell_0,\dots, v^\ell_{n-1}\rgl$  denote the vertex set  of $\bfC^n_\ell$
  in increasing order as a subset of $\om$.
  Let $P_\ell$
 be the set of permutations $p$ of $n$ such that
 the map $j \mapsto v^\ell_{p(j)}$, $j<n$,
 induces an isomorphism from $\bK_n$ to $\bfC^n_\ell$.
Note that   $|P_\ell|=|$Aut$(\bK_n)|$.

Letting $R_n=T(\bK_n,\bK)\cdot|\mathrm{Aut}(\bK_n)|$, we define an unavoidable coloring $\gamma_n:H_n\ra R_n$ as follows:
For $s\in H_n$,
define $\gamma_n(s)=\lgl t,p_s\rgl$,
where
$t<T_n$ is the index satisfying
$\bD\re \bfB_s\sim C^n_t$.
Then $\gamma_n$ is an unavoidable coloring, by Theorem \ref{thm.persistence}.

Let $m\le n<\om$.
To show that $\gamma_m\ll \gamma_n$, we start by fixing  $f\in H^n_m$ and $s,t\in H_n$ such that
$\gamma_n(s)=\gamma_n(t)$.
Note that $f:\bK_m\ra\bK_n$ is completely determined by
its behavior on the sets of vertices.
Thus, we equate $f$ with its induced injection from $m$ into $n$.
Let $\bfA,\bfB$ denote the structures $s[\bK_n],t[\bK_n]$, respectively.
Let  $A=\bD\re\bfA$ and $B=\bD\re \bfB$, the diagonal antichains of coding nodes representing the structures $\bfA,\bfB$, respectively.
Since $\gamma_n(s)=\gamma_n(t)$, it follows that
$A\sim B$ and $p_s=p_t$.
It follows that $p_s\circ f=p_t\circ f$.

Our task is to show that $\gamma_m(\hat{f}(s))=\gamma_m(\hat{f}(t))$.
Letting $\lgl \mathrm{a}_0,\dots,\mathrm{a}_{n-1}\rgl$ denote the increasing enumeration of the vertices in $\bfA$,
we see that $s\circ f$ is an injection from $m$ into
$\{\mathrm{a}_j:j<n\}$.
Letting $\bar{m}=\{j<n:\exists i<m\, (\mathrm{a}_j=s\circ f(i))\}$,
and letting $\mu$ be the strictly increasing injection from $\bar{m}$ into $m$,
we see that
$p_{\hat{f}(s)}$ is the permutation of $m$ given by
$p_{\hat{f}(s)}(i)=
\mu\circ f\circ p_s(i)$.
Likewise,  $t\circ f$ is an injection from $m$ into
$\{\mathrm{b}_j:j<n\}$, where
$\lgl \mathrm{b}_0,\dots,\mathrm{b}_{n-1}\rgl$ denotes the increasing enumeration of the vertices in $\bfB$.
Since $p_s=p_t$,
we see that $f\circ p_s=f\circ  p_t$,
and hence, the set of indices
$\{j<n:\exists i<m\, (\mathrm{b}_j=t\circ f(i))\}$
equals $\bar{m}$.
Thus, $p_{\hat{f}(t)}(i)=
\mu\circ f\circ p_t(i)$ for each $i<m$.
Hence,
$p_{\hat{f}(s)}=p_{\hat{f}(t)}$.

$\hat{f}\circ s$ maps $\bK_m$ to the substructure $\bfA'$ of  $\bfA$ on vertices $\{\mathrm{a}_{p_s\circ f(i)}:i<m\}$.
This substructure induces the antichain of coding nodes $A':=\{c^A_{p_s\circ f (i)}:i<m\}\sse A$;
that is, $A'=A\re \bfA'$.
Similarly,
$t\circ \hat{f}$ maps $\bK_m$ to the  substructure
$\bfB'$ of  $\bfB$ on vertices $\{\mathrm{b}_{p_t\circ f(i)}:i<m\}$;
this induces  the antichain of coding nodes  $B':=B\re\bfB'=\{c^B_{ p_s\circ f(i)}:i<m\}\sse B$.
Since $p_s=p_t$,
we have $p_s\circ f= p_t\circ f$, and since $A\sim B$, it follows that $A'\sim B'$.
Let $\ell<T_m$ be the index such that $A'\sim B'\sim C^m_{\ell}$.
Then   $\gamma_m(\hat{f}(s))=(\ell,p_{\hat{f}(s)})=\gamma_m(\hat{f}(t))$, since
$p_{\hat{f}(s)}=p_{\hat{f}(t)}$.
Therefore, $\gamma_m\ll \gamma_n$.
\end{proof}

\begin{rem}
We point out that Theorem \ref{thm.apply} holds for \Fraisse\ classes  with relations  of any arity.
However, it looks unlikey that any \Fraisse\ classes with non-trivial relations of arity at least three will satisfy the hypthesis of that theorem.
\end{rem}

For languages with relations of arity at most two,
the big Ramsey structure  of a  \Fraisse\
limit $\bK $
with \EEAP$^+$ is obtained simply  by  expanding the language $\mathcal{L}$ of $\bK$ to the language
$\mathcal{L}^*=\mathcal{L}\cup\{\triangleleft,\mathscr{Q}\}$,
where $\triangleleft$ and $\mathscr{Q}$ are not in $\mathcal{L}$, $\triangleleft$ is a binary relation symbol, and $\mathscr{Q}$ is a quaternary relation symbol.
In fact, by
Theorem \ref{thm.apply},
this will be the case for any \Fraisse\ class with canonical partitions characterized via diagonal antichains of coding nodes in a  coding tree of $1$-types.
The big Ramsey $\mathcal{L}^*$-structure $\bK^*$ for $\bK$ is described as follows.

Let $\bD$ be the diagonal antichain of coding nodes from the proof of Theorem \ref{thm.apply}, and
recall the linear order $\prec$ on $\bS$ described in Subsection \ref{subsec.3.2}
(see paragraph following Fact \ref{fact.simsamestructure}).
Note that $(\bD,\prec)$
is isomorphic to the rationals as a linear order.
Following Zucker
in Section 6 of \cite{Zucker19},
let $R$ be the quaternary relation on $\bD$ given by:
For
$p\preccurlyeq q\preccurlyeq r\preccurlyeq s\in \bD$,
set
\begin{equation}
R(p,q,r,s)\Longleftrightarrow |p\wedge q|\le |r\wedge s|,
\end{equation}
where $p\preccurlyeq q$ means either $p\prec q$ or $p=q$.
Without loss of generality, we may
use $\bK$ to denote $\bK\re \bD$.
Define $\bK^*$
be
the expansion of $\bK$ to the
language
$\mathcal{L}^*$ in which
$\triangleleft$ is interpreted as $\prec$ and $\mathscr{Q}$ is interpreted as $R$.
Then we have the following.

\begin{thm}\label{thm.BRS}
Let $\mathcal{K}$ be a \Fraisse\ class in language $\mathcal{L}$
with relation symbols of arity at most two
and $\bK$ a \Fraisse\ limit of $\mathcal{K}$.
Suppose that $\bK$
satisfies
\EEAP$^+$,
 and let   $\mathcal{L}^*=\mathcal{L}\cup\{\triangleleft,\mathscr{Q}\}$, where
 $\triangleleft$ is a binary relation symbol and $\mathscr{Q}$ is a quaternary relation symbol.
Then  the  $\mathcal{L}^*$-structure
$\bK^*$
is a big Ramsey structure for $\bK$.
\end{thm}

\begin{proof}
Theorems
\ref{thm.bounds} and
 \ref{thm.apply}
  imply the existence of a big Ramsey structure for $\bK$.
  Moreover, the proof of  Theorem
 \ref{thm.apply} shows that $\bK^*$ satisfies
  Definition \ref{defn.bRs}
  of a big Ramsey structure.
\end{proof}


We
 now
can quickly deduce Theorem \ref{thm.SESAPimpliesORP}
below:
The ordered expansion of the age of any \Fraisse\ structure
 with relations of arity at most two
satisfying \EEAP$^+$ is a Ramsey class.
This theorem offers  a new approach
 for  proving that
 such
\Fraisse\ classes have
ordered expansions which are Ramsey,
 complementing the much more general,
famous partite construction method of \Nesetril\ and \Rodl\ (see \cite{Nesetril/Rodl77} and \cite{Nesetril/Rodl83}) which is  at the heart  of   finite  structural  Ramsey  theory.

For the rest of this section, we work only with
\Fraisse\ classes in a finite relational  language $\mathcal{L}$
with relation symbols of arity at most two.
Let  $<$ be an additional binary relation symbol not in $\mathcal{L}$, and
let $\mathcal{L}'=\mathcal{L}\cup\{<\}$.
Let $\mathcal{K}^{<}$ denote
the class of all ordered expansions of structures in $\mathcal{K}$, namely,
the  collection of
all
$\mathcal{L}'$-structures in which
$<$ is interpreted as a linear order and whose reducts to the language $\mathcal{L}$  are members of $\mathcal{K}$.
Since  $\mathcal{K}$ has disjoint amalgamation by assumption, $\mathcal{K}^<$ will be a \Fraisse\ class
with disjoint amalgamation.
 We denote the \Fraisse\ limit of $\mathcal{K}^<$ by $\bK^<$, and note that $\bK^<$ is universal for all countable
 $\mathcal{L}'$-structures in which the relation symbol $<$ is interpreted as a linear order.
 We shall write
$\bfM' := \lgl \bfM,<'\rgl$
for any $\mathcal{L}'$-structure interpreting $<$  as a linear order; it will be understood that $\bfM$ is an $\mathcal{L}$-structure and that  $<'$ is
   the linear order on
   ${\rm M}$ interpreting $<$.

\begin{defn}\label{defn.comb}
Given   a \Fraisse\ class $\mathcal{K}$ and an enumerated \Fraisse\ structure $\bK$,
let $\bU$ be the unary-colored coding tree of $1$-types for $\bK$.
We call  a finite  antichain $C$ of coding nodes in $\bU$
 a {\em  comb} if and only if
for any two coding nodes $c,c'$ in $C$,
\begin{equation}
|c|<|c'|\ \Longleftrightarrow   \   c\prec c',
\end{equation}
where $\prec$ is the lexicographic order on $T$.
\end{defn}

\begin{thm}\label{thm.SESAPimpliesORP}
Let $\mathcal{K}$ be a \Fraisse\ class in a finite relational language $\mathcal{L}$
with relation symbols of arity at most two, and suppose that the \Fraisse\ limit of
$\mathcal{K}$ has \EEAP$^+$.
Then the ordered expansion $\mathcal{K}^<$ of $\mathcal{K}$ has the Ramsey property.
 \end{thm}

\begin{proof}
Let $\bK$ be any enumerated \Fraisse\ limit of $\mathcal{K}$.
Then
$\bK$ has universe $\om$, and may be
regarded  as a linearly ordered structure in order-type $\om$,
that is, as an $\mathcal{L}'$-structure $\lgl \bK,\in\rgl$
in which the relation symbol $<$ is interpreted as the order inherited from $\om$.
Let $\bU$ be the coding tree of $1$-types associated with $\bK$.

Let $\bfA',\bfB'$ be members of $\mathcal{K}^{<}$ such that
$\bfA'$ embeds into $\bfB'$.
Fix  a finite coloring $f$  of all
copies of
$\bfA'$ in $\lgl \bK,\in\rgl$.
Note that in this context,
a substructure $\lgl \bfA^*,\in\rgl$ of $\lgl \bK,\in\rgl$
is a
copy
of $\bfA'$
 when there is an
 $\mathcal{L}'$-isomorphism between $\lgl \bfA, <' \rgl$ and $\lgl \bfA^*, \in \rgl$.

Let $\bT$ be a diagonal coding subtree of $\bU$, and let $A\sse \bT$ be a comb representing $\bfA'$.
Thus,  if $\lgl c^A_i:i<m\rgl$ is the enumeration of $A$ in order of increasing length,
then  the coding node  $c^A_i$ represents the  $i$-th vertex  of $\bfA'$
(according to its linear ordering $<'$).
Let $f^*$ be the coloring  on $\Sim(A)$ induced by $f$.
By Theorem \ref{thm.onecolorpertreetype},
there is a diagonal coding subtree $T\sse \bT$ in which all
similarity copies of $A$ have the same $f^*$ color.

Let $D\sse T$ be an antichain of coding nodes
 representing  a copy of $\bK$.
(This is guaranteed by
Lemma \ref{lem.bD}.)
By Theorem \ref{thm.persistence},
there is a subset $B^*\sse  D$ such that $B^*$ is a comb representing a copy of  $\bfB'$ in the order inherited on the coding nodes in $B^*$.
Then every copy of $\bfA'$ represented by a set of coding nodes in $B^*$
is represented by a comb, and hence
 has the same $f$-color.
 Since
 $\lgl \bK,\in\rgl$
 is an $\mathcal{L}'$-structure interpreting the relation symbol $<$ as a linear order,  $\lgl \bK,\in\rgl$
 embeds into the \Fraisse\ limit of $\mathcal{K}^{<}$, and so
 it follows from
 Definition
 \ref{defn.RP}
that $\mathcal{K}^{<}$ has the Ramsey property.
\end{proof}

\begin{rem}
It is impossible for any comb   to represent a copy of
a
\Fraisse\ structure $\bK$
satisfying \EEAP$^+$
when
$\bK$ has at least one non-trivial relation
of arity at least two.
The contrast  between similarity types of diagonal  antichains of $1$-types persisting in every copy of $\bK$
in a coding tree
 and  combs (or any other fixed similarity type)
 being sufficient
 to prove
 the Ramsey property for the ordered expansion of its age
lies at the heart of the difference between big Ramsey degrees for $\bK$ and the Ramsey property for $\mathcal{K}^<$.
\end{rem}

In the paper \cite{Hubicka/Nesetril19},
\Hubicka\ and \Nesetril\ prove general theorems
 from
which the majority of  Ramsey classes can be deduced.
In particular,
Corollary 4.2 of \cite{Hubicka/Nesetril19} implies that every relational \Fraisse\ class with  free amalgamation  has
an ordered expansion with the Ramsey property.
So  for \Fraisse\ classes satisfying \SFAP,
 Theorem \ref{thm.SESAPimpliesORP} provides a new proof of special case of  a known result.
However, we  are not aware  of a  prior  result implying
Theorem \ref{thm.SESAPimpliesORP}
in its full generality.

A different approach to recovering the  ordered Ramsey property is given in  \cite{Hubicka_CS20}.
In that paper, \Hubicka's  results on big Ramsey degrees via the Ramsey theory  of   parameter spaces
recover
a special case of
the \Nesetril-\Rodl\ theorem \cite{Nesetril/Rodl77},
that the class of finite  ordered triangle-free graphs has the Ramsey property.

These approaches to proving the Ramsey property for  ordered \Fraisse\ classes may seem at first glance very different from the partite construction method.
However,
the  methods  must be  related at some fundamental level, similarly to the relationship between the  Halpern-\Lauchli\  and
Hales-Jewett theorems.
It will be interesting to see  if this could lead to new  Hales-Jewett theorems corresponding to the various forcing constructions
(in \cite{DobrinenJML20}, \cite{DobrinenH_k19}, \cite{Zucker19}, and this paper)
which have been used to determine finite and exact big Ramsey degrees.


\section{Concluding remarks and open problems}\label{sec.infdiml}

In Section \ref{sec.EEAPClasses},
we gave examples of \Fraisse\ classes with \Fraisse\ limits satisfying \EEAP$^+$.
By
Theorem  \ref{thm.indivisibility},
any  such \Fraisse\ limit is indivisible, and by
Theorem \ref{thm.main},
 any such
\Fraisse\ limit
with relations of arity at most two
has finite big Ramsey degrees and admits a big Ramsey structure that has a simple characterization.

\begin{question}
Which other \Fraisse\ classes either satisfy \SFAP, or more generally, have \Fraisse\ limits satisfying \EEAP$^+$?
\end{question}

\Fraisse\ structures consisting of finitely many independent linear orders
present an interesting case as we do not know whether they satisfy the
Diagonal Coding Property, and hence whether they have \EEAP$^+$, but their
ages do have \EEAP,
and their coding trees have bounded branching.
This motivates the formulation of the following properties:
For $k\ge 2$,  we say that
the \Fraisse\ limit $\bK$ of
a \Fraisse\ class $\mathcal{K}$ satisfies
{\em $k$-\EEAP$^+$}
if
$\bK$
satisfies \EEAP; there is a perfect  subtree $\bT$ of the coding tree of $1$-types $\bU$
for $\bK$
such that $\bT$ represents a copy of $\bK$ and $\bT$ has splitting nodes with degree $\le k$; and the appropriately formulated Extension Property holds.
Note that here, we are not requiring $\bT$ to be a skew tree:  $\bT$ is allowed to have more than one splitting node on any given level.
 We let
 \BEEAP$^+$
 stand for
 {\em Bounded \EEAP$^+$},
meaning that there is a $k \ge 2$ such that
$k$-\EEAP$^+$
holds.
This brings us to the following implications.

\begin{fact}\label{SFAPEEAP}
\SFAP\ $\Lra$ \EEAP$^+$ $\Lra$ \TwoEEAP$^+$
$\Lra$ \BEEAP$^+$ $\Lra$ \EEAP.
\end{fact}

Theorem \ref{thm.SFAPimpliesDCT} showed
that
\SFAP\  implies  \EEAP$^+$.
By definition, \EEAP$^+$\ implies
\TwoEEAP$^+$,
which in turn implies
\BEEAP$^+$.
 Each of these
 properties implies
 \EEAP, again by definition.
 The example of finitely many independent linear orders shows that
 \BEEAP$^+$ does not imply \EEAP$^+$.
On the other hand,
all examples considered in this paper satisfying \EEAP\ also satisfy
\BEEAP$^+$.
 It could well be the case   that \EEAP\   is equivalent to
  \BEEAP$^+$.
 The methods in this paper can be adjusted  to handle structures with
  \BEEAP$^+$,
 so the following question becomes interesting.

\begin{question}
Are \EEAP, \BEEAP$^+$, and \TwoEEAP$^+$ equivalent?
In other words, does  \EEAP\ imply \BEEAP$^+$, and does \BEEAP$^+$ imply \TwoEEAP$^+$?
\end{question}

Throughout this paper, we have mentioned  known results
regarding   finite big Ramsey degrees.
Actual calculations of  big Ramsey degrees, however, are still sparse, and have only been found
for the rationals by Devlin in \cite{DevlinThesis},
the Rado graph by Larson
 in
\cite{Larson08},
 the structures $\bQ_n$  and $\mathbf{S}(2)$ by Laflamme, Nguyen Van Th\'{e}, and Sauer in
 \cite{Laflamme/NVT/Sauer10},
 and the rest of the circular digraphs $\mathbf{S}(n)$, $n\ge 3$, by Barbosa in
\cite{Barbosa20}.
The canonical partitions in Theorem \ref{thm.bounds}
 provide a template for calculating the big Ramsey degrees for all \Fraisse\ structures satisfying
 \EEAP$^+$.

\begin{problem}
Calculate the big Ramsey degrees $T(\bfA,\bK)$,
$\bfA\in\mathcal{K}$,
 for each \Fraisse\ class
 $\mathcal{K}$
with relations of arity at most two and
 with a \Fraisse\ limit
 satisfying
 \EEAP$^+$.
\end{problem}

Lastly,
it is our hope that using combs in trees of $1$-types might lead to smaller bounds  for the ordered  Ramsey property.

\begin{problem}\label{prob.combbetterbound}
Suppose $\mathcal{K}$
is a \Fraisse\ class
with relations of arity at most two and
with \Fraisse\ limit satisfying
\EEAP$^+$.
Use Theorem \ref{thm.SESAPimpliesORP} to  find  better bounds for the smallest size of a structure $\bfC\in\mathcal{K}^{<}$
such that
\begin{equation}
\bfC\ra (\bfB)^{\bfA}
\end{equation}
for any given $\bfA\le\bfB$ inside $\mathcal{K}^{<}$.
\end{problem}

\bibliographystyle{amsplain}
\bibliography{references}

\end{document}